\documentclass[11pt]{amsart}
\usepackage[T1]{fontenc}
\usepackage{lmodern}
\usepackage{amsmath,amssymb,mathtools}
\usepackage{tikz-cd}
\usepackage{enumitem}
\usepackage{microtype}
\usepackage{mathrsfs}
\usepackage[a4paper,margin=30mm]{geometry}
\usepackage[
  colorlinks=true,
  linkcolor=blue,
  citecolor=blue,
  urlcolor=blue
]{hyperref}
\theoremstyle{plain}
\newtheorem{theorem}{Theorem}[section]
\newtheorem{proposition}[theorem]{Proposition}
\newtheorem{lemma}[theorem]{Lemma}

\newtheorem{maintheorem}{Theorem}

\theoremstyle{definition}
\newtheorem{definition}[theorem]{Definition}
\newtheorem{remark}[theorem]{Remark}

\numberwithin{equation}{section}
\setlist[enumerate,1]{label=\textup{(\arabic*)},leftmargin=2.5em}

\newcommand{\OO}{\mathcal O}
\newcommand{\RG}{\mathbf R\Gamma}
\newcommand{\cris}{\mathrm{cris}}
\newcommand{\tors}{\mathrm{tors}}
\newcommand{\Fil}{\operatorname{Fil}}
\newcommand{\Def}{\operatorname{Def}}
\newcommand{\Per}{\operatorname{Per}}
\newcommand{\Hom}{\operatorname{Hom}}
\newcommand{\Ext}{\operatorname{Ext}}
\newcommand{\Spec}{\operatorname{Spec}}

\newcommand{\im}{\operatorname{im}}
\newcommand{\coker}{\operatorname{coker}}
\newcommand{\cofib}{\operatorname{cofib}}

\newcommand{\length}{\operatorname{length}}

\title[Witt liftability of quasi-$F$-split Calabi--Yau varieties]
{On the Witt liftability of quasi-F-split Calabi--Yau varieties}
\author{Teppei Takamatsu}
\address{Department of Mathematics, Saitama University, Saitama, Japan}
\email{teppeitakamatsu.math@gmail.com}
\date{}
\hypersetup{
 pdftitle={Witt liftings of quasi-F-split Calabi-Yau varieties},
 pdfauthor={Teppei Takamatsu}
}

\begin{document}
\begin{abstract}
We prove that a quasi-$F$-split strict Calabi--Yau variety $X$ of dimension
$d$ over an algebraically closed field $k$ of odd characteristic $p$
admits a smooth projective lifting over $W(k)$ if $p\geq d-1$,
or if $p=d-2$ and $H^0(X,T_X)=0$.
We also prove $W_3(k)$-liftability in odd characteristic.
\end{abstract}
\maketitle

\section{Introduction}

Let $k$ be an algebraically closed field of characteristic $p>2$,
and write $W=W(k)$ and $W_n=W/p^nW$.
A strict Calabi--Yau variety of dimension $d\geq1$ is a smooth
projective geometrically integral variety $X/k$ such that
\[
 \omega_X\simeq\OO_X,\qquad
 H^i(X,\OO_X)=0\quad(0<i<d).
\]
We study the existence of smooth projective liftings of such
varieties over $W$.
In dimensions one and two, these varieties are genus-one curves
and K3 surfaces, respectively, and admit such liftings in every
characteristic; for K3 surfaces, see
\cite[Proposition~2.6]{Ito18}, and also \cite{Deligne81,Ogus79}.
The situation changes in dimension three: Hirokado constructed a
non-liftable example in characteristic three \cite{Hirokado99},
and Schr\"oer constructed examples in characteristics two and three
\cite{Schroer04}.
Thus additional hypotheses are needed for a general lifting theorem.

A natural candidate for such an additional hypothesis is
quasi-$F$-splitting.
Yobuko proved that the quasi-$F$-split height of a strict
Calabi--Yau variety agrees with the height of its Artin--Mazur
formal group, and that every smooth quasi-$F$-split variety
admits a $W_2$-lifting \cite[Theorems~4.4 and~4.5]{YobukoQ}.
In particular, finite Artin--Mazur height gives the first step
towards a Witt-vector lifting.
Achinger and Zdanowicz made this construction explicit for a
chosen quasi-$F$-splitting \cite[Appendix~A]{AZ}.
See also \cite{BBKW24} for related results on Witt liftings
of globally $F$-split normal projective surfaces.

The relationship between $W_2$-liftability and Hodge theory
also plays an important role in our approach. Deligne and Illusie proved Hodge--de Rham
degeneration from $W_2$-liftability in the range $\dim X\leq p$
\cite[Corollaire~2.4]{DI87}.
For smooth proper quasi-$F$-split varieties, Petrov proved
degeneration without a dimension restriction, by decomposing
the de Rham complex in all degrees \cite[Theorem~1.1]{Petrov}.

Extending a $W_2$-lifting through all higher Witt rings requires
further control of deformation obstructions.
Earlier mixed-characteristic unobstructedness results of
Schr\"oer and Ekedahl--Shepherd-Barron use lifting criteria
together with an already existing characteristic-zero lifting
and additional cohomological hypotheses \cite{Schroer03,ES05}.
Brantner and Taelman proved a mixed-characteristic
Bogomolov--Tian--Todorov theorem assuming Hodge--de Rham
degeneration and torsion-freeness of crystalline cohomology
\cite[Theorem~A]{BT}.
Their sharper criterion uses a decomposition of the crystalline
complex and torsion-freeness of an associated mapping group
\cite[Theorem~7.16]{BT}.
Every polarized abelian variety over $k$ admits a lifting
to characteristic zero, by work of Mumford, Norman, and Oort
\cite[Corollary~3.2]{NormanOort80}.
Serre--Tate theory identifies infinitesimal deformations
of an abelian variety with deformations of its
$p$-divisible group \cite{LST64};
see also \cite[Theorem~1.2.1]{Katz81}.
For ordinary abelian varieties, Serre--Tate theory
provides canonical liftings over $W(k)$
\cite[Section~2]{Katz81}.
For Calabi--Yau varieties, Achinger and Zdanowicz
established an analogue under additional ordinarity
and unobstructedness hypotheses
\cite[Theorem~1]{AZ}.
Brantner and Taelman subsequently obtained canonical projective
liftings for smooth projective varieties with trivial canonical
bundle that are Bloch--Kato $2$-ordinary and satisfy an
\'etale cohomology condition
\cite[Theorems~B and~C and Corollary~1.7]{BT}; that condition
is automatic for strict Calabi--Yau varieties of dimension at
least two \cite[Remark~1.5]{BT}.
See also \cite{Moore26}, where quasicanonical liftings are
characterized in terms of rational $p$-adic Hodge theory.

In this paper, we construct Witt-vector liftings of
quasi-$F$-split strict Calabi--Yau varieties without assuming
torsion-freeness of crystalline cohomology.
Our main result is the following.

\begin{maintheorem}
\label{main:witt}
Let $X$ be a quasi-$F$-split strict Calabi--Yau variety of dimension
$d$ over $k$, and suppose that $p>2$.
Assume that either $p\geq d-1$, or $p=d-2$ and
$H^0(X,T_X)=0$.
Then $X$ admits a smooth projective lifting over $W(k)$.
\end{maintheorem}

In particular, quasi-$F$-split strict Calabi--Yau threefolds and
fourfolds lift over $W$ in odd characteristic.

Theorem~\ref{main:witt} establishes the existence of a
$W$-lifting, but does not assert unobstructedness of the
entire deformation functor.
In particular, it does not assert that an arbitrary
$W_n$-lifting extends to $W_{n+1}$.
Our construction allows finite-level liftings to be
modified before they are extended.
Such obstructions do occur even at height one:
for the Addington--Bragg threefold in characteristic three,
Ward shows that some $W_2$-liftings do not extend to $W_3$
\cite[Proposition~A.1]{Ward26}.

Without the dimension bound, we obtain one further Witt level.

\begin{maintheorem}
\label{main:W3}
Let $X$ be a quasi-$F$-split strict Calabi--Yau variety over $k$.
If $p>2$, then $X$ admits a smooth projective lifting over $W_3(k)$.
\end{maintheorem}

More precisely, in odd characteristic the $W_2$-lifting
associated with any chosen quasi-$F$-splitting extends to $W_3$.
We also record a consequence of the sharper criterion of
Brantner--Taelman: if $X$ is quasi-$F$-split and
$H^d_{\cris}(X/W)$ is torsion-free, then its deformations over
local Artinian $W$-algebras are unobstructed, without a dimension
bound (Proposition~\ref{common:torsion-free}).

The vanishing of global vector fields cannot be assumed
for quasi-$F$-split strict Calabi--Yau varieties.
Addington and Bragg construct such a threefold over
$\overline{\mathbf F}_3$ with
$\dim_k H^0(X,T_X)=1$
\cite[Propositions~2.2 and~6.1]{AddingtonBragg}.
Their computation also exhibits torsion in
$H^3_{\cris}(X/W)$ \cite[Section~6]{AddingtonBragg}.
This example illustrates that the range $d\leq p+1$ in
Theorem~\ref{main:witt} allows both crystalline torsion
and nonzero global vector fields.

While preparing this paper, the author became aware of
Ward's newly posted preprint on arXiv \cite{Ward26}.
For strict Calabi--Yau threefolds of height one, Ward gives a different construction of
Witt-vector liftings by keeping a torsion coordinate
of the crystalline period point equal to zero
\cite[Theorem~4.13]{Ward26}.
The author intends to examine the precise relationship
between Ward's methods and those of the present paper
in subsequent work.

We now outline the proof, assuming $d\geq3$.
We use the period domain and obstruction comparison of
\cite[Sections~6--7]{BT}.
The first input is an integral splitting supplied by finite
Artin--Mazur height.
If this height is $h$, then
\[
 \RG(X,W\OO_X)\simeq W\oplus D_h[-d],
 \qquad D_h:=H^d(X,W\OO_X)\simeq W^{\oplus h}.
\]
Verschiebung is topologically nilpotent on $D_h$
(Lemma~\ref{common:witt}).
We use this nilpotence to split the morphism of crystalline $F$-gauges
induced by the projection $W\Omega_X^\bullet\to W\OO_X$
(Lemma~\ref{edge:lower-split}).
The resulting summand has Hodge weights $0$ and $1$.
It supplies a decomposition
\[
 C:=\RG_{\cris}(X/W)\simeq P\oplus Q,
 \qquad Q\simeq W\oplus W[-d],
\]
for which $Q\otimes_W^{\mathbf L}k\to\RG(X,\OO_X)$ is an
equivalence in $\mathcal D(k)$ (Lemma~\ref{common:split}).
Thus the Hodge quotient is lifted integrally even when
crystalline cohomology has torsion.

In the graph chart for this decomposition, the Hodge quotient
of a lifting $X_n/W_n$ takes the form
\[
 (a_n,\mathrm{id}):P_n\oplus Q_n\longrightarrow Q_n,
 \qquad P_n=P\otimes_W^{\mathbf L}W_n,
 \quad Q_n=Q\otimes_W^{\mathbf L}W_n.
\]
For the initial lifting associated with a quasi-$F$-splitting
in odd characteristic, we choose the decomposition so that
$a_2=0$ (Lemma~\ref{common:canonical}).
This is where the assumption $p>2$ enters the construction.
The zero graph condition kills the first period obstruction.
The obstruction comparison, whose injectivity follows from
Petrov's degeneration theorem and Serre duality, then gives
the $W_3$-lifting of Theorem~\ref{main:W3}
(Proposition~\ref{common:w3}).

To proceed to higher Witt levels, we prove a lifting criterion
in terms of the first Hodge edge
(Proposition~\ref{tower:extension}).
For every smooth proper lifting $Y/W_{r+1}$, with
$Y_r=Y\otimes_{W_{r+1}}W_r$, the required condition is the
surjectivity of
\[
 e_r:
 H^{d-1}(Y_r,\Fil^1\Omega^\bullet_{Y_r/W_r})
 \longrightarrow H^{d-2}(Y_r,\Omega^1_{Y_r/W_r}).
\]
By relative Serre duality, this gives injectivity of the
relative obstruction comparison
(Lemma~\ref{tower:injectivity}).
We integrate tangent directions along finite divided-power
curves and evaluate these curves to correct the obstruction
to extending an existing lifting
(Lemmas~\ref{tower:pd-curves} and~\ref{tower:correction}).
The corrections need not preserve that lifting at its highest
Witt levels. Their effect is nevertheless bounded uniformly:
for the perfect period complex
$D=\mathbf R\Hom_W(P,Q)$, choose $E\geq1$ such that
$p^E H^1(D)_{\tors}=0$.
At step $n\geq3$, the correction preserves the reduction to
$W_{\max\{2,n-E\}}$.
Consequently, the corrected liftings stabilize at every fixed
level and yield a compatible tower extending the chosen
$W_2$-lifting.
An ample line bundle extends along the tower, and formal
existence algebraizes it to a smooth projective $W$-lifting
(Lemma~\ref{common:algebraization}).

The characteristic bound and the additional vanishing assumption
at $d=p+2$ enter in proving the surjectivity of $e_r$.
When $d<p$, this already follows from Ogus's degeneration
theorem: since $Y_r$ lifts to $W_{r+1}$, the Hodge--de Rham
spectral sequence
\[
 E_1^{i,j}=H^j(Y_r,\Omega^i_{Y_r/W_r})
 \Longrightarrow H^{i+j}_{\mathrm{dR}}(Y_r/W_r)
\]
degenerates at $E_1$
\cite[Theorem~8.2.6]{Ogus94}.
In particular, $e_r$ is surjective.

To treat smaller characteristics, we use the finite-height
summand constructed in Lemma~\ref{edge:lower-split}.
A finite-level form of Ogus's crystalline comparison relates
the Nygaard filtration on crystalline cohomology to the
Hodge filtration of $Y_r$
(Lemma~\ref{edge:comparison}).
We then use the finite-height summand and its counterpart
on the Hodge side (Lemma~\ref{edge:descent}) to reduce the
surjectivity question to the matrix calculation of
Lemma~\ref{edge:kernel}.
This proves that $e_r$ is surjective provided that
\[
 H^{d-1-i}(X,\Omega_X^i)=0
 \qquad(p<i\leq d-1)
\]
(Proposition~\ref{edge:finite}).
This condition is automatic when $d\leq p+1$.
When $d=p+2$, it reduces to
\[
 H^0(X,\Omega_X^{d-1})=0,
\]
or equivalently $H^0(X,T_X)=0$.

Thus the hypotheses of Theorem~\ref{main:witt} ensure the
surjectivity of $e_r$, and the lifting criterion completes
the proof.

The organization of the paper is as follows:
Section~\ref{sec:common} constructs the finite-height summand
and the initial lifting, and proves Theorem~\ref{main:W3} and the
torsion-free case.
Section~\ref{sec:tower} establishes the lifting criterion and
the correction procedure.
Section~\ref{sec:edge} proves finite Hodge-edge surjectivity
and completes the proof of Theorem~\ref{main:witt}.

\section{Crystalline splitting and the initial lifting}
\label{sec:common}

We use Brantner--Taelman's period domain and obstruction comparison
\cite[Sections~6--7]{BT}. Their graph chart starts from a crystalline
decomposition lifting the Hodge quotient. We construct such a
decomposition for quasi-$F$-split strict Calabi--Yau varieties without
assuming crystalline torsion-freeness, using a splitting of $F$-gauges.
Under the hypotheses of Lemma~\ref{common:canonical}, we also arrange
compatibility with the initial $W_2$-lifting. The graph descriptions
and obstruction comparison below are applications of their theory;
the finite-height splitting and the initial compatibility are the
additional inputs established here.

\subsection{Notation and terminology}
Throughout this paper, $k$ is an algebraically closed field of
characteristic $p>2$, $W=W(k)$, and $W_n=W/p^nW$.
A lifting of $X$ over $W_n$ or $W$ means a smooth projective
scheme over that ring, together with an identification of
its special fibre with $X$.

We say that $X$ is \emph{quasi-$F$-split} if, for some $m\geq1$, the
map
\[
 \OO_X\longrightarrow F_*W_m\OO_X/p,
 \qquad f\longmapsto[f]^p,
\]
has an $\OO_X$-linear retraction. Equivalently, there is an additive map
$\tau:W_m\OO_X\to\OO_X$ satisfying
\[
 \tau(F(a)b)=a_0\tau(b),\qquad \tau(1)=1,
 \qquad a_0=R^{m-1}a,
\]
where $R^{m-1} \colon W_m \OO_X \to \OO_X$ is the restriction morphism.
Such a map is called a quasi-$F$-splitting of level $m$.

For a commutative ring $R$, let $\mathcal D(R)$ be the derived
$\infty$-category of $R$-modules and let $D(R)$ be its homotopy
category. Let $\operatorname{Perf}(R)\subset\mathcal D(R)$ be the
full subcategory of perfect complexes. Write $\mathcal S$ for the
$\infty$-category of spaces, and
$\operatorname{Map}_{\mathcal D(R)}(M,N)\in\mathcal S$ for the
mapping space. An equivalence in $\mathcal D(R)$ induces an
isomorphism in $D(R)$; an equivalence in $\mathcal S$ means an
equivalence of homotopy types.

All complexes are cohomologically indexed. Tensor products of
complexes, fibres, and cofibres are derived. We write $\RG$ for derived
global sections and $\Fil^i\Omega^\bullet$ for the stupid filtration
of the de Rham complex. If $K$ is a $W$-complex, we put
$K_n=K\otimes_W^{\mathbf L}W_n$.
For the rest of this section, $X$ is a quasi-$F$-split strict
Calabi--Yau variety of dimension $d\geq3$.

\subsection{Witt cohomology}
Yobuko's theorem identifies the quasi-$F$-splitting height of $X$
with the finite height $h$ of its Artin--Mazur formal group and gives
a smooth $W_2$-lifting \cite[Theorems~4.4 and~4.5]{YobukoQ}.
Set
\[
 C=\RG_{\cris}(X/W),\qquad
 N_h=\RG(X,W\OO_X),\qquad
 D_h=H^d(X,W\OO_X).
\]
The crystalline--de Rham--Witt comparison $ \RG_{\cris}(X/W) \simeq \RG (X, W\Omega_X^{\bullet})$ \cite{Illusie} gives the projection
\[
 \epsilon:C\longrightarrow N_h.
\]
The same map is obtained by evaluating crystalline cohomology on
the Witt divided-power thickenings and taking the derived inverse
limit \cite[Sections~2.1 and~3.2--3.3]{BerthelotQ}.

\begin{lemma}\label{common:witt}
There is an equivalence in $\mathcal D(W)$
\[
 N_h\simeq W\oplus D_h[-d],
\]
and $D_h$ is a free $W$-module of rank $h$.
On $D_h$, the operators $F,V$ satisfy $FV=VF=p$,
$V$ is topologically nilpotent, and
\[
 pD_h\subseteq VD_h\subseteq D_h,
 \qquad D_h/VD_h\simeq H^d(X,\OO_X)\simeq k.
\]
\end{lemma}

\begin{proof}
The restriction sequences for Witt vectors and the vanishing of the
intermediate cohomology of $\OO_X$ show that
$H^i(X,W_n\OO_X)=0$ for $i\notin\{0,d\}$.
The restriction maps on the remaining groups are surjective.
Their inverse limits compute $N_h$, and $H^0(N_h)=W$.
The module $D_h$ is the Dieudonn\'e module of the
one-dimensional Artin--Mazur formal group of $X$
\cite[Chapter~II, Corollary~4.3]{ArtinMazur77}.
Since the formal group has finite height $h$,
classical Cartier--Dieudonn\'e theory shows that
$D_h$ is a free $W$-module of rank $h$ and that
$V$ is topologically nilpotent for the $p$-adic topology
\cite[Chapter~I, Section~1, p.~96]{ArtinMazur77}.
The exact sequence for $V:W\OO_X\to W\OO_X$ gives the displayed
quotient, while $p=VF$ gives $pD_h\subseteq VD_h$.
The freeness of the two cohomology groups gives the stated
decomposition in $\mathcal D(W)$.
\end{proof}

\subsection{The finite-height summand}
\label{subsec:finite-height}

We briefly recall the Rees description of crystalline $F$-gauges
over $k$; see \cite[Sections~3.3--3.4 and~4.2]{BhattFG}.
Put
\[
 A=W[v_-,v_+]/(v_-v_+-p),\qquad
 \deg(v_-)=-1,\quad \deg(v_+)=1.
\]
A crystalline gauge over $k$ is represented by a graded
$A$-complex $M=\bigoplus_{r\in\mathbb Z}M^{(r)}$ that is
derived $p$-complete in each graded degree.
Here $r$ is the internal grading, distinct from the cohomological
degree. The actions of $v_-$ and $v_+$ give maps
\[
 v_-:M^{(r)}\longrightarrow M^{(r-1)},\qquad
 v_+:M^{(r)}\longrightarrow M^{(r+1)},
\]
whose composites in either order are multiplication by $p$.
For a filtered complex $(K,\Fil^\bullet)$ compatible with the
$p$-adic filtration, its Rees complex has
$M^{(r)}=\Fil^rK$, with $v_-$ acting by the filtration map
and $v_+$ by multiplication by $p$.

The two endpoint complexes of $M$ are
\[
 M_-:=M|_{v_-=1},\qquad M_+:=M|_{v_+=1},
\]
where specialization means derived $p$-completed base change
along the maps $A\to W$ sending $(v_-,v_+)$ to $(1,p)$ and
$(p,1)$, respectively. In each internal degree $r$, specialization
induces maps
\[
 \rho_{\pm,r}:M^{(r)}\longrightarrow M_\pm;
 \qquad \rho_\pm:=\rho_{\pm,0}.
\]
We call these the endpoint maps.

Write $\sigma:W\to W$ for the Witt Frobenius, and set
$\sigma^*L=W\otimes_{W,\sigma}^{\mathbf L}L$.
The natural map $\eta_L:L\to\sigma^*L$, $x\mapsto1\otimes x$,
is $\sigma$-semilinear.
An $F$-gauge is a gauge $M$ equipped with a Frobenius gluing
equivalence in $\mathcal D(W)$
\[
 \alpha_M:\sigma^*M_+\xrightarrow{\sim}M_-.
\]
A morphism of $F$-gauges is a morphism of gauges $f:M\to N$
together with a homotopy
\[
 f_-\circ\alpha_M\simeq\alpha_N\circ\sigma^*f_+,
\]
where $f_\pm:M_\pm\to N_\pm$ are the induced maps on the
endpoint complexes. We will use the maps in internal degree one
\[
 \operatorname{can}_{1,M}:=\rho_{-,1},\qquad
 \varphi_{1,M}:=\alpha_M\circ\eta_{M_+}\circ\rho_{+,1}.
\]
The latter is the $\sigma$-semilinear first divided Frobenius.

Let $\mathcal N_X$ be the Rees complex of the Nygaard
filtration on crystalline cohomology. Its graded pieces are
\[
 \mathcal N_X^{(r)}
 =\RG\bigl(X,\mathcal N^{\geq r}W\Omega_X^\bullet\bigr),
\]
where
\[
 \mathcal N^{\geq r}W\Omega_X^j=
 \begin{cases}
 p^{r-j-1}VW\Omega_X^j,&j<r,\\
 W\Omega_X^j,&j\geq r
 \end{cases}
\]
\cite[Definition~8.1]{BMS2}.
The divided Frobenius maps $\varphi_r=p^{-r}\varphi$, where
$\varphi|_{W\Omega_X^j}=p^jF$, give its $F$-gauge structure,
i.e., they induce the Frobenius gluing equivalence
$\alpha_{\mathcal N_X}:\sigma^*(\mathcal N_X)_+
\xrightarrow{\sim}(\mathcal N_X)_-$ in $\mathcal D(W)$.
In internal degree one, $\operatorname{can}_{1,\mathcal N_X}$
is induced by inclusion, and $\varphi_{1,\mathcal N_X}$ by
\[
 \varphi_1:
 [\,VW\OO_X\longrightarrow W\Omega_X^1
       \longrightarrow W\Omega_X^2\longrightarrow\cdots\,]
 \longrightarrow W\Omega_X^\bullet,
\]
with
\[
 \varphi_1(Va)=a,\qquad
 \varphi_1|_{W\Omega_X^j}=p^{j-1}F\quad(j\geq1).
\]
These formulas define a map of complexes by $FdV=d$ and
$dF=pFd$. Their compatibility with the endpoint description
is explained in
\cite[Example~3.4.7 and Proposition~4.4.2]{BhattFG}.

For a perfect gauge $M$, put
\[
 M_{\mathrm H}
 :=
 M\otimes_A^{\mathbf L}k[v_-],
 \qquad k[v_-]=A/(v_+).
\]
We define filtration (called the Hodge filtration of $M$) on
\[
 M_{\mathrm H}|_{v_-=1}
 \simeq M_-\otimes_W^{\mathbf L}k
\]
by
\[
 \Fil_{\mathrm H}^i(M_-\otimes_W^{\mathbf L}k)
 :=M_{\mathrm H}^{(i)},
\]
with transition maps induced by $v_-$.
Its associated graded pieces are
\[
 \begin{aligned}
 \operatorname{gr}_{\mathrm H}^iM
 &:=
 \cofib\bigl(
 M_{\mathrm H}^{(i+1)}
 \xrightarrow{v_-}
 M_{\mathrm H}^{(i)}
 \bigr)\\
 &\simeq
 \bigl(M\otimes_A^{\mathbf L}k\bigr)^{(i)},
 \qquad k=A/(v_-,v_+).
 \end{aligned}
\]

Let $\mathcal U_h$ be the Rees complex with
\[
 \mathcal U_h^{(r)}=
 \begin{cases}
 \RG(X,W\OO_X),&r\leq0,\\
 \RG(X,p^{r-1}VW\OO_X),&r\geq1,
 \end{cases}
\]
where $v_-$ acts by inclusion and $v_+$ by multiplication by $p$.
Its $F$-gauge structure is induced by the divided Frobenius
maps $p^{-r}F$ on these levels.

The following $F$-gauge splitting will produce the decomposition
required in \cite[Corollary~7.4 and Theorem~7.16]{BT} after forgetting
the additional structure. We retain that structure for the
finite-level comparison in Section~\ref{sec:edge}.

\begin{lemma}\label{edge:lower-split}
\begin{enumerate}
\item\label{edge:lower-split-model}
The projection
$W\Omega_X^\bullet\to W\OO_X$ induces a morphism of $F$-gauges
\[
 q:\mathcal N_X\longrightarrow\mathcal U_h.
\]
The minus endpoint of $\mathcal U_h$ satisfies
\[
 (\mathcal U_h)_-\simeq N_h\simeq W\oplus D_h[-d]
 \quad\text{in }\mathcal D(W).
\]
The underlying graded $A$-complex of $\mathcal U_h$
admits the following finite free model:
\refstepcounter{equation}
\[
 \begin{gathered}
 \mathcal U_h
 \simeq
 A\mathbf 1
 \oplus
 \left(Az\oplus\bigoplus_{a=1}^{h-1}Ae_a\right)[-d],\\
 \deg(\mathbf 1)=\deg(z)=0,
 \qquad
 \deg(e_a)=1.
 \end{gathered}
 \tag{\theequation}\label{eq:lower-split-free-model}
\]
Here $\deg$ denotes the internal grading, and the right-hand
side has zero differential.
In particular, $\mathcal U_h$ is perfect.
\item\label{edge:lower-split-graded}
The Hodge graded pieces of $\mathcal U_h$ are
\[
 \operatorname{gr}_{\mathrm H}^{i}\mathcal U_h
 \simeq
 \begin{cases}
 \RG(X,\OO_X)\simeq k\oplus k[-d],&i=0,\\
 (VD_h/pD_h)[-d]\simeq k^{\oplus(h-1)}[-d],&i=1,\\
 0,&i\notin\{0,1\}.
 \end{cases}
\]
Under the identifications
\[
 \operatorname{gr}_{\mathrm H}^{0}\mathcal N_X
 \simeq\RG(X,\OO_X),
 \qquad
 \operatorname{gr}_{\mathrm H}^{0}\mathcal U_h
 \simeq\RG(X,\OO_X),
\]
the map $\operatorname{gr}_{\mathrm H}^{0}q$ is the identity.

\item\label{edge:lower-split-fibre}
The fibre $\mathcal E_0:=\operatorname{fib}(q)$ admits a bounded
finite graded free $A$-model whose homogeneous generators
all have grading degree at least one. Its minus endpoint is
\[
 E_0:=(\mathcal E_0)_-\simeq\RG(X,W\Omega_X^{\geq1}).
\]
Its underlying Frobenius is $p\varphi_E$, where $\varphi_E$
acts in form degree $j\geq1$ as $p^{j-1}F$.

\item\label{edge:lower-split-section}
The morphism $q$ admits a section in the $\infty$-category
of crystalline $F$-gauges over $k$.
\end{enumerate}
\end{lemma}

\begin{proof}
\textit{(1)}
The projection
$W\Omega_X^\bullet\to W\OO_X$
preserves the Nygaard filtrations and is compatible
with their transition maps and the divided Frobenius maps.
It therefore induces $q$. 
By definition, the minus endpoint of $\mathcal U_h$ is $N_h$,
and the minus endpoint of $q$ is the projection $C\to N_h$.
Lemma~\ref{common:witt} gives
$N_h\simeq W\oplus D_h[-d]$ in $\mathcal D(W)$.
It remains to construct the graded free model
\eqref{eq:lower-split-free-model}.

Choose a $k$-point $x:\Spec k\to X$ and write
$a:X\to\Spec k$ for the structure morphism. Let $\mathbf 1$ denote the unit $F$-gauge over $k$.
By functoriality, $a^*$ and $x^*$ induce morphisms of
$F$-gauges
\[
 \mathbf 1\xrightarrow{a^*}\mathcal U_h
 \xrightarrow{x^*}\mathbf 1,
 \qquad x^*a^*=\mathrm{id}_{\mathbf 1}.
\]
Hence
$\mathcal U_h\simeq\mathbf 1\oplus\operatorname{fib}(x^*)$
in the $\infty$-category of crystalline $F$-gauges over $k$.
The Nygaard filtration on the unit summand $W$ is
\[
 \mathcal N^{\geq r}W=p^{\max(r,0)}W,
\]
since $VW=pW$.
The strict Calabi--Yau vanishings imply that $\operatorname{fib}(x^*)$
is concentrated in cohomological degree $d$, with module
$D_h$ and filtration
\[
 \mathcal N^{\geq r}D_h=
 \begin{cases}
 D_h,&r\leq0,\\
 p^{r-1}VD_h,&r\geq1.
 \end{cases}
\]
By Lemma~\ref{common:witt}, $pD_h\subseteq VD_h$ and
$D_h/VD_h\simeq k$. Choose a $W$-basis $e_0,\ldots,e_{h-1}$
adapted to this lattice, so that
\[
 VD_h=pWe_0\oplus\bigoplus_{a=1}^{h-1}We_a.
\]
For $r\geq1$, the filtration is therefore
\[
 \mathcal N^{\geq r}D_h
 =p^rWe_0\oplus\bigoplus_{a=1}^{h-1}p^{r-1}We_a.
\]
The rank-one graded $A$-module $A\ell$ with $\deg(\ell)=w$
corresponds, via its minus endpoint, to $W\ell$ with filtration
$p^{\max(r-w,0)}W\ell$ in degree $r$.
Applying this description to the unit summand and the above
basis of $D_h$ gives the model
\eqref{eq:lower-split-free-model}, with $z$ corresponding to $e_0$.
This is a bounded finite free graded $A$-model,
so $\mathcal U_h$ is perfect.

\smallskip\noindent
\textit{(2)}
Tensoring the graded free model
\eqref{eq:lower-split-free-model} with
$k=A/(v_-,v_+)$ gives
\[
 \mathcal U_h\otimes_A^{\mathbf L}k
 \simeq
 k\mathbf 1
 \oplus
 \left(kz\oplus\bigoplus_{a=1}^{h-1}ke_a\right)[-d].
\]
Since $\mathbf 1,z$ have grading degree zero and the $e_a$
have grading degree one, we obtain
\[
 \operatorname{gr}_{\mathrm H}^{i}\mathcal U_h
 \simeq
 \begin{cases}
 k\oplus(D_h/VD_h)[-d]
   \simeq k\oplus k[-d],&i=0,\\
 (VD_h/pD_h)[-d]
   \simeq k^{\oplus(h-1)}[-d],&i=1,\\
 0,&i\notin\{0,1\}.
 \end{cases}
\]

By \cite[Theorem~3.3.5(6)]{BhattFG}, we have
\[
 \operatorname{gr}_{\mathrm H}^{0}\mathcal N_X
 \simeq\RG(X,\OO_X).
\]
Under this identification, the morphism
$\operatorname{gr}_{\mathrm H}^{0}q$ is induced by
the map on $W\OO_X/VW\OO_X=\OO_X$
coming from the projection
$W\Omega_X^\bullet\to W\OO_X$.
This map is the identity. Hence
\[
 \begin{gathered}
 \operatorname{gr}_{\mathrm H}^{0}\mathcal N_X
 \simeq\RG(X,\OO_X)
 \simeq\operatorname{gr}_{\mathrm H}^{0}\mathcal U_h,\\
 \operatorname{gr}_{\mathrm H}^{0}q
 =\mathrm{id}_{\RG(X,\OO_X)}
 \quad\text{under these identifications}.
 \end{gathered}
\]
\smallskip\noindent
\textit{(3)}
By \cite[Theorem~3.3.5(6) and Remark~3.3.6]{BhattFG},
$\mathcal N_X$ is perfect and
\begin{equation}
\label{eqn:Nygaard's Hodge}
 \operatorname{gr}_{\mathrm H}^{i}\mathcal N_X
 \simeq\RG(X,\Omega^i_{X/k})[-i]=0
 \qquad(i\notin[0,d]).
\end{equation}
Parts (1) and (2) imply that $\mathcal E_0$ is perfect and
\[
 \operatorname{gr}_{\mathrm H}^{i}\mathcal E_0\simeq0
 \qquad(i\leq0).
\]
To translate this vanishing into a statement about generators,
use \cite[Lemma~3.4.11(1)]{BhattFG} to take an algebraic
graded $A$-model. Since $A$ is graded local with homogeneous
maximal ideal $\mathfrak m=(v_-,v_+)$,
\cite[Lemma~4.10(ii)]{DellAmbrogioStevenson2013} supplies a
minimal graded free resolution $P^\bullet$, with
$d(P^n)\subseteq\mathfrak mP^{n+1}$.
The complex $P^\bullet\otimes_A k$ has zero differential;
finite Tor amplitude and graded Nakayama therefore make
$P^\bullet$ bounded.
The above vanishing then gives
\[
 (P^n\otimes_A k)^{(i)}
 =H^n(\operatorname{gr}_{\mathrm H}^{i}\mathcal E_0)=0
 \qquad(i\leq0),
\]
so every homogeneous generator of $P^\bullet$ has grading
degree at least one.

At the minus endpoint, $q$ is the projection
$\RG(X,W\Omega_X^\bullet)\to\RG(X,W\OO_X)$.
Its fibre is $E_0=\RG(X,W\Omega_X^{\geq1})$.
On this complex, crystalline Frobenius factors as
$p\varphi_E$, with
\[
 \varphi_E|_{W\Omega_X^j}=p^{j-1}F\quad(j\geq1).
\]
The identity $dF=pFd$ shows that $\varphi_E$ is a map of
complexes \cite{IR}. The first Nygaard level of the fibre is
also $E_0$, so its first divided Frobenius is $\varphi_E$.

\smallskip\noindent
\textit{(4)}
We show that the entire mapping complex from $\mathcal U_h$
to $\mathcal E_0$ vanishes. Let
\[
 \begin{gathered}
 \mathcal H:=\mathbf R\underline{\operatorname{Hom}}_A
       (\mathcal U_h,\mathcal E_0),\\
 K:=\mathbf R\operatorname{Hom}_W(N_h,E_0),
 \end{gathered}
\]
where $\mathcal H$ carries the induced Frobenius gluing.
By (1) and (3), a graded free model for $\mathcal H$ has
homogeneous generators of grading degrees $j-i\geq0$,
where $i\in\{0,1\}$ and $j\geq1$.
For a generator $\ell$ of grading degree $w\geq0$,
the minus endpoint map is
\[
 (A\ell)^{(0)}=Wv_-^w\ell\xrightarrow{\sim}W\ell,
 \qquad v_-^w\ell\longmapsto\ell.
\]
Applying this term by term gives an equivalence in $\mathcal D(W)$
\[
 \rho_-:\mathcal H^{(0)}
 \xrightarrow{\sim}\mathcal H_-.
\]
Since $\mathcal U_h$ is perfect, endpoint base change gives
a canonical equivalence in $\mathcal D(W)$
\[
 c_{\mathcal H}:\mathcal H_-
 \xrightarrow{\sim}
 \mathbf R\operatorname{Hom}_W
       ((\mathcal U_h)_-,(\mathcal E_0)_-)
 \simeq K,
\]
where $(\mathcal U_h)_-\simeq N_h$ and
$(\mathcal E_0)_-\simeq E_0$.

Using $c_{\mathcal H}$ to identify $\mathcal H_-$ with $K$,
define
\[
 T:=
 c_{\mathcal H}\circ
 \bigl(\alpha_{\mathcal H}\circ\eta_{\mathcal H_+}
       \circ\rho_+\bigr)
 \circ\rho_-^{-1}\circ c_{\mathcal H}^{-1}:
 K\longrightarrow K.
\]
Thus $T$ is the plus endpoint map followed by Frobenius gluing,
transported to the underlying $\mathbb Z_p$-complex $K$.
The gluing formula \cite[Remark~4.2.8]{BhattFG} now reads
\refstepcounter{equation}
\[
 \mathbf R\operatorname{Hom}_{F\text{-}\mathrm{Gauge}}
       (\mathcal U_h,\mathcal E_0)
 \simeq\operatorname{fib}(1-T:K\longrightarrow K)
 \quad\text{in }D(\mathbb Z_p).
\tag{\theequation}\label{eq:common-gauge-mapping-fibre}
\]

It remains to compute $T$ and prove that $1-T$ is invertible.
Under the point-induced splitting $N_h=W\oplus D_h[-d]$, put
\[
 V_N:=p\sigma^{-1}\oplus V.
\]
This is $\sigma^{-1}$-semilinear, and the first Nygaard level
$\mathcal U_h^{(1)}$ is identified with $V_NN_h$.
Since $FV=p$, its first divided Frobenius is the
$\sigma$-semilinear isomorphism
\[
 \varphi_{1,\mathcal U_h}:V_NN_h\xrightarrow{\sim}N_h,
 \qquad\varphi_{1,\mathcal U_h}(V_Nx)=x.
\]
Using the graded free models above and the associated endpoint
and Hom complexes, the map
$c_{\mathcal H}\circ\rho_-:\mathcal H^{(0)}\to K$
is an isomorphism of complexes.
For $n\in\mathbb Z$ and a cochain
\[
 f\in K^n
 =\Hom_W^n(N_h,E_0)
 =\prod_i\Hom_W(N_h^i,E_0^{i+n}),
\]
let
\[
 \widetilde f
 :=(c_{\mathcal H}\circ\rho_-)^{-1}(f)
 \in(\mathcal H^{(0)})^n.
\]
Its internal-degree-one component
\[
 \widetilde f^{(1)}
 \in\Hom_W^n
 \bigl(\mathcal U_h^{(1)},\mathcal E_0^{(1)}\bigr)
\]
satisfies
\[
 \operatorname{can}_{1,\mathcal E_0}\circ\widetilde f^{(1)}
 =
 f\circ\operatorname{can}_{1,\mathcal U_h}.
\]
Under the identifications
$\mathcal U_h^{(1)}=V_NN_h$ and
$\mathcal E_0^{(1)}=E_0$,
the maps $\operatorname{can}_{1,\mathcal U_h}$ and
$\operatorname{can}_{1,\mathcal E_0}$ are inclusion and
the identity, respectively. Hence
\[
 \widetilde f^{(1)}
 =
 f|_{V_NN_h}
 \in\Hom_W^n(V_NN_h,E_0).
\]
By the gluing description of $F$-gauges
\cite[Remark~4.2.4]{BhattFG}, Frobenius gluing on internal Hom
is given by postcomposition with the target gluing map
and precomposition with the inverse source gluing map.
Together with the definition of $T$ and the above identification
of $\widetilde f^{(1)}$, this gives
\refstepcounter{equation}
\[
 \begin{gathered}
 T(f)\circ\varphi_{1,\mathcal U_h}
 =\varphi_E\circ f|_{V_NN_h},\\
 \text{hence}\qquad T(f)=\varphi_E\circ f\circ V_N.
 \end{gathered}
 \tag{\theequation}\label{eq:common-T-formula}
\]
Both factors are defined over $W$. Their opposite
semilinearities make $T(f)$ $W$-linear, while
$T(af)=\sigma(a)T(f)$. Thus $T$ is $\sigma$-semilinear on $K$
and $1-T$ is an endomorphism of its underlying
$\mathbb Z_p$-complex.

Topological nilpotence of $V$ gives an integer $a\geq1$ with
$V^aD_h\subseteq pD_h$. Hence $V_N^aN_h\subseteq pN_h$,
and \eqref{eq:common-T-formula} gives
\[
 T^a(f)=\varphi_E^a\circ f\circ V_N^a\equiv0\pmod p.
\]
Consequently $1-T$ is invertible modulo $p$, with inverse
$1+T+\cdots+T^{a-1}$. The complex $K$ is derived $p$-complete,
so derived Nakayama \cite[Tag~0G1U]{Stacks} makes $1-T$ an
equivalence in $\mathcal D(\mathbb Z_p)$.
By \eqref{eq:common-gauge-mapping-fibre}, the mapping complex
from $\mathcal U_h$ to $\mathcal E_0$ therefore vanishes.
In particular, the connecting morphism
$\mathcal U_h\to\mathcal E_0[1]$ in the fibre triangle of $q$
is zero, which gives the required section.
\end{proof}

\subsection{Hodge quotients}

The next lemma supplies the crystalline decomposition assumed in
\cite[Corollary~7.4 and Theorem~7.16]{BT}. It plays the role of the
decomposition constructed in the proof of \cite[Corollary~7.5]{BT},
with finite height replacing the torsion-freeness hypothesis used there.

\begin{lemma}\label{common:split}
The projection $\epsilon:C\to N_h$ has a section in $D(W)$.
There is a decomposition in $\mathcal D(W)$
\[
 C\simeq P\oplus Q,\qquad Q:=W\oplus W[-d],
\]
and an equivalence in $\mathcal D(k)$
\[
 \alpha:Q_1=Q\otimes_W^{\mathbf L}k
 \xrightarrow{\sim}\RG(X,\OO_X)
\]
such that the following diagram commutes in $D(k)$:
\[
\begin{tikzcd}[column sep=large]
 C_1 \arrow[r,"\pi_1"] \arrow[d,"\sim"'] &
 Q_1 \arrow[d,"\alpha"] \\
 \RG_{\mathrm{dR}}(X/k) \arrow[r,"q_{\mathrm H}"'] &
 \RG(X,\OO_X).
\end{tikzcd}
\]
Here $\pi_1=\operatorname{pr}_Q\otimes_W^{\mathbf L}k$,
the left vertical arrow is the crystalline--de Rham comparison,
and $q_{\mathrm H}$ is induced by the projection
$\Omega_X^\bullet\to\OO_X[0]$.
Both $P$ and $Q$ are perfect.
\end{lemma}

\begin{proof}
Lemma~\ref{edge:lower-split} gives a section $s:N_h\to C$ of
$\epsilon$ after forgetting the $F$-gauge structure.
Fix $H^d(X,\OO_X)\simeq k$ and lift the natural quotient
\[
 \bar r:D_h/pD_h\twoheadrightarrow D_h/VD_h\simeq k
\]
to a $W$-linear functional $\ell:D_h\to W$.
Since $D_h$ is finite free and $\bar r$ is surjective, such a lift
exists and is surjective. Thus
\[
 \lambda:=\mathrm{id}_W\oplus\ell[-d]:N_h\longrightarrow Q
\]
has a section $t$, and $\pi:=\lambda\epsilon$ has section $st$.
Set $P=\operatorname{fib}(\pi)$ to obtain $C\simeq P\oplus Q$.

Under $Q_1\simeq k\oplus k[-d]\simeq\RG(X,\OO_X)$, the
reduction of $\lambda$ is $\mathrm{id}_k\oplus\bar r[-d]$,
the natural map $N_h\otimes_W^{\mathbf L}k\to\RG(X,\OO_X)$.
Compatibility of the crystalline--de Rham--Witt comparison with
the projection $W\Omega_X^\bullet\to W\OO_X$
identifies the reduction of $\pi$ with the Hodge quotient.
Finally, $P$ is perfect because $C$ and $Q$ are perfect.
\end{proof}

Petrov's theorem gives degeneration of the Hodge--de Rham and
conjugate spectral sequences \cite[Theorem~1.1]{Petrov}. In particular,
\begin{equation}
\label{eqn:Petrov}
 H^{d-1}\bigl(X,\Fil^1\Omega_X^\bullet\bigr)
 \longrightarrow H^{d-2}(X,\Omega_X^1)
\end{equation}
is surjective.

The quotient and fibre descriptions below are a special case of
the graph chart in \cite[Lemma~7.1]{BT}. We record the coordinates
and the fibre inclusion for later use.

\begin{lemma}\label{common:graph}
Let $X_n/W_n$ be a smooth proper lifting of $X$, and fix
a decomposition $C\simeq P\oplus Q$ as in
Lemma~\ref{common:split}. Then:
\begin{enumerate}
\item
We have
\[
 \omega_{X_n/W_n}\simeq\OO_{X_n},\qquad
 V_n:=\RG(X_n,\OO_{X_n})\simeq W_n\oplus W_n[-d].
\]

\item
The Hodge quotient $q_{\mathrm H,n}:C_n\to V_n$
restricts to an equivalence
$Q_n\xrightarrow{\sim}V_n$ in $\mathcal D(W_n)$.
Using its inverse to identify $V_n$ with $Q_n$, the quotient
has the form
\[
 (a_n,\mathrm{id}):P_n\oplus Q_n\longrightarrow Q_n,
 \qquad a_n:P_n\longrightarrow Q_n.
\]

\item
The fibre
\[
 K_n:=\RG(X_n,\Fil^1\Omega_{X_n/W_n}^\bullet)
\]
is equivalent to $P_n$ in $\mathcal D(W_n)$.
Under this equivalence and the fixed decomposition
$C_n\simeq P_n\oplus Q_n$, the natural map $K_n\to C_n$
corresponds to
\[
 (\mathrm{id},-a_n):P_n\longrightarrow P_n\oplus Q_n.
\]
The quotient and fibre identifications commute with reduction.
\end{enumerate}
\end{lemma}

\begin{proof}
The vanishing of $H^1(X,\OO_X)$ makes
$\operatorname{Pic}(X_n)\to\operatorname{Pic}(X)$ injective,
so $\omega_{X_n/W_n}\simeq\OO_{X_n}$.
By derived base change and Lemma~\ref{common:split}, the map
$\beta_n:=q_{\mathrm H,n}|_{Q_n}:Q_n\to V_n$ reduces to an
equivalence in $\mathcal D(k)$.
Derived Nakayama makes $\beta_n$ an equivalence
in $\mathcal D(W_n)$, which also gives
$V_n\simeq W_n\oplus W_n[-d]$.

As in the graph chart of \cite[Lemma~7.1]{BT}, define
\begin{equation}
\label{eqn:an}
 a_n:=\beta_n^{-1}\circ q_{\mathrm H,n}|_{P_n}.
\end{equation}
Then $\beta_n^{-1}q_{\mathrm H,n}=(a_n,\mathrm{id})$.
The Hodge filtration triangle identifies $K_n$ with the fibre
of this quotient, namely $P_n$ with inclusion
$(\mathrm{id},-a_n)$. These constructions commute with reduction.
\end{proof}

\subsection{The initial period}
For a chosen quasi-$F$-splitting $\tau$ of level $m$, let
$X_2(\tau)$ denote the lifting of Achinger--Zdanowicz \cite{AZ} (cf.\ \cite{YobukoQ}):
\[
 \OO_{X_2(\tau)}=W_{m+1}\OO_X/V(\ker\tau)
 \qquad\text{\cite[Corollary~A.1.2]{AZ}}.
\]

We next refine the decomposition so that the specified $W_2$-lifting
has zero graph coordinate. This supplies the initial condition for
our application of the period-space methods of \cite[Sections~6--7]{BT};
the required compatibility is verified using crystalline evaluation.

\begin{lemma}\label{common:canonical}
Suppose that $p>2$ and $X_2^0=X_2(\tau)$ for a chosen
quasi-$F$-splitting.
The decomposition $C \simeq P \oplus Q$ in Lemma~\ref{common:split} can be chosen
so that the map $a_2:P_2\to Q_2$ defined by \eqref{eqn:an}
for the Hodge quotient of $X_2^0$ vanishes.
\end{lemma}

\begin{proof}
Put $V_2:=\RG(X_2^0,\OO_{X_2^0})$, and let
$q_{\mathrm H,2}:C_2\to V_2$ be its Hodge quotient.
Writing $r_2:C\to C_2$ for reduction, we first show that
$q_{\mathrm H,2}\circ r_2=b_2\circ\epsilon$
for a $W$-linear map $b_2:N_h\to V_2$.

Let
\[
 \pi:W_{m+1}\OO_X\twoheadrightarrow
 \OO_{X_2^0}=W_{m+1}\OO_X/V(\ker\tau)
\]
be the quotient map defining $X_2^0$.
We show that $\pi$ is compatible with the canonical
divided power structures on $VW_m\OO_X$ and
$p\OO_{X_2^0}$.
Write $\gamma_j^{\mathrm{can}}$ for the canonical divided powers
on $p\OO_{X_2^0}$, induced by those on $pW_2$.
Since $p>2$, we have $p^j/j!\in p^2\mathbb Z_p$ for $j\geq2$.
Thus, for every local section $u$ of $\OO_{X_2^0}$,
\refstepcounter{equation}
\[
 \gamma_j^{\mathrm{can}}(pu)
 =
 \frac{p^j}{j!}u^j
 =
 0
 \qquad(j\geq2).
 \tag{\theequation}\label{eq:initial-pd-vanishing}
\]
Since $\tau$ is additive with values in $\OO_X$, we have
\[
 \tau(pz)=p\tau(z)=0
 \qquad(z\in W_m\OO_X).
\]
Thus $pW_m\OO_X\subseteq\ker\tau$.
For $p>2$ and $j\geq2$, we have
$p^{j-1}/j!\in p\mathbb Z_p$.
The formula for the Witt divided powers therefore gives,
for every local section $b$ of $W_m\OO_X$,
\[
 \gamma_j(Vb)
 =
 V\!\left(\frac{p^{j-1}}{j!}b^j\right)
 \in V(pW_m\OO_X)
 \subseteq V(\ker\tau)
 =\ker\pi.
\]
This inclusion and
\eqref{eq:initial-pd-vanishing}  give
\[
 \pi(\gamma_j(Vb))
 =
 0
 =
 \gamma_j^{\mathrm{can}}(\pi(Vb))
 \qquad(j\geq2).
\]
Compatibility for $j=1$ is automatic, so $\pi$ respects
divided powers.

Let $\operatorname{res}_{m+1}:W\OO_X\to W_{m+1}\OO_X$
denote Witt truncation.
By the construction of the crystalline--de Rham--Witt comparison
\cite[Sections~3.2--3.3]{BerthelotQ},
the composite of crystalline base change to $W_{m+1}$
and restriction to the Witt PD thickening
$(|X|,W_{m+1}\OO_X)$ is given by
\[
 C\xrightarrow{\epsilon}N_h
 \xrightarrow{\RG(\operatorname{res}_{m+1})}
 \RG(X,W_{m+1}\OO_X).
\]
By the construction of the crystalline--de Rham comparison
\cite[Section~3.1]{BerthelotQ}, crystalline restriction to
$X_2^0$ agrees with the Hodge quotient
\[
 q_{\mathrm H,2}:
 C_2\xrightarrow{\sim}
 \RG(X_2^0,\Omega^\bullet_{X_2^0/W_2})
 \xrightarrow{\RG(\operatorname{pr}_0)}V_2,
\]
where $\operatorname{pr}_0$ is the projection
$\Omega^\bullet_{X_2^0/W_2}\to\OO_{X_2^0}[0]$.
Thus crystalline base change to $W_2$ followed by restriction
to $X_2^0$ is the composite
\[
 C\xrightarrow{r_2}C_2\xrightarrow{q_{\mathrm H,2}}V_2.
\]
Set
\[
 b_2:=
 \RG(\pi)\circ\RG(\operatorname{res}_{m+1}):
 N_h\longrightarrow V_2.
\]
Since $\pi$ respects divided powers, the associated morphism
\[
 X_2^0\longrightarrow (|X|,W_{m+1}\OO_X)
\]
is a morphism of PD thickenings of $X$ over $W_{m+1}$.
Naturality of crystalline evaluation along this morphism gives
\refstepcounter{equation}
\[
 q_{\mathrm H,2}\circ r_2=b_2\circ\epsilon.
 \tag{\theequation}\label{eq:initial-period-factorization}
\]

We now choose the integral decomposition to match this
Hodge quotient.
By Lemmas~\ref{common:witt} and~\ref{common:graph},
\[
 N_h\simeq W\oplus D_h[-d],
 \qquad
 V_2\simeq W_2\oplus W_2[-d].
\]
Choose the identification of $H^0(V_2)$ using the unit,
and choose a generator of $H^d(V_2)$.
Since $W$ and $D_h$ are free, $b_2$ is determined by
its maps in degrees zero and $d$.
These are the standard reduction $W\to W_2$ and a
functional
\[
 \ell_2:D_h\longrightarrow W_2.
\]
Its reduction modulo $p$ is the quotient
$D_h/pD_h\twoheadrightarrow D_h/VD_h\simeq k$,
so $\ell_2$ is primitive.

Lift $\ell_2$ to a functional $\ell:D_h\to W$, and use
this $\ell$ in the construction of Lemma~\ref{common:split}.
This gives a split quotient
\[
 C\xrightarrow{\epsilon}N_h
 \xrightarrow{(\mathrm{id},\ell)}
 Q=W\oplus W[-d].
\]
Let $P$ be its fibre and choose the resulting decomposition
$C\simeq P\oplus Q$.
By \eqref{eq:initial-period-factorization} and the choice
of $\ell$ lifting $\ell_2$, the reduction of this quotient
modulo $p^2$ agrees with $q_{\mathrm H,2}$ after identifying
$Q_2$ with $V_2$.
Hence $q_{\mathrm H,2}|_{P_2}=0$, and
\eqref{eqn:an} gives $a_2=0$.
\end{proof}

\begin{remark}
For globally $F$-split varieties in characteristic $p>2$,
the factorization \eqref{eq:initial-period-factorization}
follows from \cite[Corollary~6.6.1]{AZ}.
The proof above extends their argument to an arbitrary
chosen quasi-$F$-splitting and then chooses an integral
decomposition $C\simeq P\oplus Q$ for which $a_2=0$.
We have not been able to verify the proof of
\cite[Corollary~6.6.1]{AZ} in characteristic two,
and therefore retain the assumption $p>2$ in
Lemma~\ref{common:canonical}.
\end{remark}

\subsection{The period map and the third Witt vectors}

Let $\operatorname{Art}_W(k)$ be the category of local Artinian
$W$-algebras equipped with a $W$-algebra augmentation $A\to k$
inducing an isomorphism $A/\mathfrak m_A\cong k$.
Morphisms in $\operatorname{Art}_W(k)$
are $W$-algebra maps compatible with the augmentations to $k$.
For $A\in\operatorname{Art}_W(k)$ and $E\in\mathcal D(W)$, put
\[
 I_A:=\ker(A\to k)=\mathfrak m_A,\qquad
 E_A:=E\otimes_W^{\mathbf L}A,\qquad
 E_1:=E\otimes_W^{\mathbf L}k.
\]

Let $\operatorname{Art}^{\mathrm{pd}}_W(k)$ be the category whose
objects are pairs $(A,\gamma)$ with $A\in\operatorname{Art}_W(k)$
and a divided-power structure $\gamma$ on $I_A$ compatible with
the canonical divided powers on $pW$. A morphism
$(A,\gamma)\to(B,\delta)$ is a morphism
$f:A\to B$ in $\operatorname{Art}_W(k)$ such that
\[
 f(\gamma_n(x))=\delta_n(f(x))
 \qquad(x\in I_A,\ n\geq1).
\]

\begin{definition}[The space associated to a complex]
\label{common:omega-infinity}
For $K\in\mathcal D(R)$, define the pointed space
\[
 \Omega^\infty K
 :=\operatorname{Map}_{\mathcal D(R)}(R[0],K),
\]
with base point the zero morphism. There are natural
isomorphisms of abelian groups
\[
 \pi_0(\Omega^\infty K)\cong H^0(K),\qquad
 \pi_i(\Omega^\infty K,0)\cong H^{-i}(K)
 \quad(i\geq1).
\]
For a cochain representative $(K^\bullet,d)$, the Dold--Kan
correspondence applied to the cohomological truncation
$\tau_{\leq0}K^\bullet$, with indices reversed, gives a
simplicial $R$-module model of this space.
Its vertices are elements of $Z^0(K^\bullet)$, and a
$1$-simplex from $x$ to $y$ is specified by
\[
 b\in K^{-1},\qquad d^{-1}b=y-x.
\]
\end{definition}

Put $V_1:=\RG(X,\OO_X)$, and regard the Hodge quotient as the
following morphism in $\mathcal D(k)$:
\[
 q_{\mathrm H}:C_1
 \xrightarrow{\,c_X\,}\RG(X,\Omega^\bullet_{X/k})
 \xrightarrow{\,\RG(\operatorname{pr}_0)\,}V_1.
\]
Here $c_X$ is the crystalline--de Rham comparison equivalence
in $\mathcal D(k)$,
and $\operatorname{pr}_0$ is the projection
$\Omega^\bullet_{X/k}\to\OO_X[0]$.

We use the quotient description of the period domain of
\cite[Definition~6.13 and Remark~6.14]{BT}, restricted to
ordinary divided-power Artinian $W$-algebras.

\begin{definition}[The period domain]
\label{common:period-domain}
For $(A,\gamma)\in\operatorname{Art}^{\mathrm{pd}}_W(k)$, let
\[
 \mathcal U_C(A):=
 \bigl(\operatorname{Perf}(A)_{C_A/}\bigr)^{\simeq}.
\]
The undercategory $\operatorname{Perf}(A)_{C_A/}$ has objects
$q_A:C_A\to V_A$ with $V_A$ perfect. The superscript
$\simeq$ means its maximal $\infty$-groupoid: morphisms are
equivalences between such arrows, with all coherent higher
homotopies retained. Derived reduction defines a map of spaces
\[
 r_A:\mathcal U_C(A)\longrightarrow\mathcal U_C(k),\qquad
 (C_A\xrightarrow{q_A}V_A)\longmapsto
 (C_1\xrightarrow{q_A\otimes_A^{\mathbf L}k}
                  V_A\otimes_A^{\mathbf L}k).
\]
Define
\[
 \Per_X(A,\gamma):=
 \operatorname{hofib}_{q_{\mathrm H}}(r_A)
 =\mathcal U_C(A)\mathop{\times}^{h}_{\mathcal U_C(k)}
   \{q_{\mathrm H}\}\quad\text{in }\mathcal S.
\]
We write $\Per_X(A)$ for $\Per_X(A,\gamma)$ when the
divided-power structure $\gamma$ has been fixed.
Explicitly, a point is a quadruple $(V_A,q_A,\rho_A,h_A)$ with
\[
 \begin{gathered}
 q_A:C_A\longrightarrow V_A,\qquad
 \rho_A:V_A\otimes_A^{\mathbf L}k\xrightarrow{\sim}V_1,\\
 h_A:\rho_A\circ(q_A\otimes_A^{\mathbf L}k)
           \simeq q_{\mathrm H}
 \quad\text{in }\operatorname{Map}_{\mathcal D(k)}(C_1,V_1).
 \end{gathered}
\]
The homotopy $h_A$ is part of the data. An equivalence of points
includes an equivalence of the targets in $\mathcal D(A)$
compatible, up to coherent homotopy, with $q_A$, $\rho_A$,
and $h_A$. Derived extension of
scalars along a morphism $(A,\gamma)\to(B,\delta)$ defines the
functor
\[
 \Per_X:\operatorname{Art}^{\mathrm{pd}}_W(k)
       \longrightarrow\mathcal S.
\]
\end{definition}

Fix an equivalence
$s:P\oplus Q\xrightarrow{\sim}C$ in $\mathcal D(W)$ as in
Lemma~\ref{common:split}, and put
$D:=\mathbf R\Hom_W(P,Q)$. Fix also an equivalence
$\alpha:Q_1\xrightarrow{\sim}V_1$ in $\mathcal D(k)$ and a homotopy
\[
 h_0:\alpha\circ\operatorname{pr}_{Q_1}
       \simeq q_{\mathrm H}\circ s_1
 \quad\text{in }
 \operatorname{Map}_{\mathcal D(k)}(P_1\oplus Q_1,V_1).
\]
The equality in $D(k)$ in Lemma~\ref{common:split} ensures the
existence of $h_0$; the choice of $h_0$ is included in the fixed
data for the coordinates. The quotient
$\operatorname{pr}_{Q_A}\circ s_A^{-1}:C_A\to Q_A$,
with reduction identification $\alpha$ and homotopy induced by
$h_0$, defines a base point $o_A\in\Per_X(A)$.

For a point $(V_A,q_A,\rho_A,h_A)$, define
\[
 \begin{aligned}
 \beta_A&:=q_A\circ s_A|_{Q_A}:Q_A\longrightarrow V_A,\\
 a_A&:=\beta_A^{-1}\circ q_A\circ s_A|_{P_A}:
                P_A\longrightarrow Q_A.
 \end{aligned}
\]
The reduction of $\beta_A$ is an equivalence in $\mathcal D(k)$,
so derived Nakayama makes $\beta_A$ an equivalence in $\mathcal D(A)$.
The same conclusion holds if the target $V_A$ is initially an
arbitrary object of $\mathcal D(A)$: derived reduction to $k$ is
conservative for an Artinian local ring $A$. Hence $V_A\simeq Q_A$
is automatically perfect. Restricting the targets in the quotient
description of \cite[Remark~6.14]{BT} to perfect complexes therefore
does not change the space; compare the proof of \cite[Lemma~7.1]{BT}.
The homotopies $h_A$ and $h_0$ induce a specified homotopy
\[
 \bar h_A:a_A\otimes_A^{\mathbf L}k\simeq0
 \quad\text{in }\operatorname{Map}_{\mathcal D(k)}(P_1,Q_1).
\]
Indeed, restricting $h_A$ and $h_0$ to $P_1$ gives
\[
 \rho_A\circ(\beta_A\otimes_A^{\mathbf L}k)
       \circ(a_A\otimes_A^{\mathbf L}k)
 \simeq q_{\mathrm H}\circ s_1|_{P_1}\simeq0.
\]
The graph chart of \cite[Lemma~7.1]{BT} therefore gives
equivalences of pointed spaces, natural in
$(A,\gamma)\in\operatorname{Art}^{\mathrm{pd}}_W(k)$,
whose composite we denote by $\operatorname{Graph}_A$:
\refstepcounter{equation}
\[
 \begin{aligned}
 (\Per_X(A),o_A)
 &\simeq\operatorname{hofib}_0\!\left(
   \operatorname{Map}_{\mathcal D(A)}(P_A,Q_A)
   \longrightarrow
   \operatorname{Map}_{\mathcal D(k)}(P_1,Q_1)\right)\\
 &\simeq\Omega^\infty\operatorname{fib}
          (D_A\longrightarrow D_1)\\
 &\simeq\Omega^\infty(D\otimes_W^{\mathbf L}I_A).
 \end{aligned}
 \tag{\theequation}\label{eq:common-graph-chart}
\]
The quotient-to-graph equivalence sends $(V_A,q_A,\rho_A,h_A)$
to $(a_A,\bar h_A)$. Perfectness of $P$ gives
\[
 \mathbf R\Hom_A(P_A,Q_A)\simeq D_A
       \quad\text{in }\mathcal D(A).
\]
Tensoring $0\to I_A\to A\to k\to0$ with $D$ gives
\[
 \operatorname{fib}(D_A\to D_1)
       \simeq D\otimes_W^{\mathbf L}I_A
       \quad\text{in }\mathcal D(W).
\]
The map $\operatorname{Graph}_A$ depends on the fixed data
$(s,\alpha,h_0)$.

For $A\in\operatorname{Art}_W(k)$, let $\Def_X(A)$ be the
nerve of the groupoid of pairs $(X_A,\iota_A)$, where $X_A$ is
smooth and proper over $A$ and
$\iota_A:X_A\times_A k\xrightarrow{\sim}X$ is a $k$-isomorphism;
morphisms are $A$-isomorphisms preserving $\iota_A$.
To define the period map, fix
$(A,\gamma)\in\operatorname{Art}^{\mathrm{pd}}_W(k)$.
For $(X_A,\iota_A)\in\Def_X(A)$, put
\[
 \begin{gathered}
 V_{X_A}:=\RG(X_A,\OO_{X_A}),\\
 c_{X_A}^{\gamma}:C_A\xrightarrow{\sim}
          \RG(X_A,\Omega^\bullet_{X_A/A})
          \quad\text{in }\mathcal D(A),\\
 q_{X_A}^{\gamma}:=
 \RG(\operatorname{pr}_{0,X_A})\circ c_{X_A}^{\gamma}:
 C_A\longrightarrow V_{X_A}.
 \end{gathered}
\]
The map $c_{X_A}^{\gamma}$ is crystalline comparison for the
divided-power thickening determined by $(A,I_A,\gamma)$, and
$\operatorname{pr}_{0,X_A}$ is the projection
$\Omega^\bullet_{X_A/A}\to\OO_{X_A}[0]$.
Derived base change and $\iota_A$ provide an equivalence
$\rho_{X_A}:V_{X_A}\otimes_A^{\mathbf L}k\xrightarrow{\sim}V_1$
in $\mathcal D(k)$ and a compatible homotopy
\[
 h_{X_A}:\rho_{X_A}\circ
       (q_{X_A}^{\gamma}\otimes_A^{\mathbf L}k)
       \simeq q_{\mathrm H}.
\]
The period map of \cite[Definition~6.13]{BT} is the map of spaces
\[
 \begin{aligned}
 \operatorname{per}_{X,A}^{\gamma}:\Def_X(A)&\longrightarrow\Per_X(A),\\
 (X_A,\iota_A)&\longmapsto
       (V_{X_A},q_{X_A}^{\gamma},\rho_{X_A},h_{X_A}).
 \end{aligned}
\]

We use the tangent and obstruction identifications, including
their sign conventions, of \cite[Theorem~6.18]{BT}.
The splitting data $(s,\alpha,h_0)$ induce an equivalence
\[
 u:P_1\xrightarrow{\sim}\RG(X,\Fil^1\Omega^\bullet_{X/k})
       \quad\text{in }\mathcal D(k).
\]
Indeed, $(s_1,\alpha,h_0)$ gives an equivalence between the
arrows $\operatorname{pr}_{Q_1}:P_1\oplus Q_1\to Q_1$ and
$q_{\mathrm H}:C_1\to V_1$.
Functoriality of fibres then gives $u$, using
$\operatorname{fib}(\operatorname{pr}_{Q_1})\simeq P_1$ and
$\operatorname{fib}(q_{\mathrm H})
 \simeq\RG(X,\Fil^1\Omega^\bullet_{X/k})$.
Let $\operatorname{pr}_1:\Fil^1\Omega^\bullet_{X/k}
\to\Omega^1_{X/k}[-1]$ be the first Hodge projection.
For
\[
 \xi\in H^2(X,T_X)
   \cong\Ext^2_{\OO_X}(\Omega^1_{X/k},\OO_X),
\]
view $\xi$ as a morphism $\Omega^1_{X/k}\to\OO_X[2]$ in
$D(\OO_X)$. By \cite[Theorem~6.18]{BT}, the map
\[
 \theta_1:H^2(X,T_X)\longrightarrow H^1(D_1)
       \cong\Hom_{D(k)}(P_1,Q_1[1])
\]
is given by the composite in $D(k)$
\refstepcounter{equation}
\[
 \theta_1(\xi)=
 (\alpha[1])^{-1}\circ\RG(\xi[-1])\circ
       \RG(\operatorname{pr}_1)\circ u.
 \tag{\theequation}\label{eq:theta1-contraction}
\]
Petrov's degeneration theorem (cf.\ (\ref{eqn:Petrov})) makes
\[
 e_X:=H^{d-1}\!\left(\RG(\operatorname{pr}_1)\circ u\right):
 H^{d-1}(P_1)\longrightarrow H^{d-2}(X,\Omega^1_{X/k})
\]
surjective. Serre duality and $\omega_X\simeq\OO_X$ make the
contraction pairing
\[
 \begin{aligned}
 \kappa:H^2(X,T_X)&\longrightarrow
 \Hom_k\!\left(H^{d-2}(X,\Omega^1_{X/k}),H^d(X,\OO_X)\right),\\
 \kappa(\xi)(v)&=\operatorname{ev}(\xi\smile v)
 \end{aligned}
\]
isomorphism, where $\operatorname{ev}:T_X\otimes_{\OO_X}
\Omega^1_{X/k}\to\OO_X$ is evaluation.
By \eqref{eq:theta1-contraction}, we have
\[
 H^{d-1}\!\left(\alpha[1]\circ\theta_1(\xi)\right)
       =\kappa(\xi)\circ e_X.
\]
Surjectivity of $e_X$ and injectivity of $\kappa$ imply
injectivity of $\theta_1$, as in \cite[Proposition~7.6]{BT}.

Let $A'\twoheadrightarrow A$ be a divided-power elementary
extension over $W$ in the sense of \cite[Definition~7.7]{BT}.
Explicitly, $(A',\gamma')\to(A,\gamma)$ is a morphism in
$\operatorname{Art}^{\mathrm{pd}}_W(k)$ whose underlying ring map
is surjective and whose kernel $J$ satisfies
\[
 I_{A'}J=0,\qquad
 \gamma_r'(x)=0\quad(x\in J,\ r\geq2).
\]
Here $\gamma'$ denotes the divided powers on $I_{A'}$.
In particular, $J$ is a finite-dimensional $k$-vector space.
The exact sequence $0\to J\to I_{A'}\to I_A\to0$ induces
the connecting homomorphism
\[
 \delta_{A'/A}:H^0(D\otimes_W^{\mathbf L}I_A)
 \longrightarrow H^1(D\otimes_W^{\mathbf L}J)
 \cong H^1(D_1)\otimes_k J.
\]
For a deformation $X_A$, set
\[
 \eta_A:=\pi_0(\operatorname{Graph}_A)
       \bigl([\operatorname{per}_{X,A}^{\gamma}(X_A)]\bigr)
       \in H^0(D\otimes_W^{\mathbf L}I_A).
\]
The obstruction comparison of
\cite[Theorem~6.18 and Proposition~7.15]{BT} reads
\begin{equation}
\label{eqn:obstruction_period}
 (\theta_1\otimes_k\mathrm{id}_J)
       \bigl(\operatorname{ob}_{X_A}(A'/A)\bigr)
       =\delta_{A'/A}(\eta_A),
 \qquad
 \operatorname{ob}_{X_A}(A'/A)\in H^2(X,T_X)\otimes_k J.
\end{equation}
Thus, by the injectivity of $\theta_1$, the lifting $X_A$
extends to $A'$ if and only if $\delta_{A'/A}(\eta_A)=0$.
Equivalently, $X_A$ extends to $A'$ if and only if its period $\eta_A$
lifts to $D\otimes_W^{\mathbf L}I_{A'}$.

For $A=W_n$, use the canonical divided powers on
$I_A=pW_n$ and write $\eta_n:=\eta_{W_n}$ for the period of
$X_n$. If $j_n:pW_n\hookrightarrow W_n$ denotes inclusion,
then
\refstepcounter{equation}
\[
 \begin{gathered}
 \eta_n\in H^0(D\otimes_W^{\mathbf L}pW_n),\qquad
 a_n\in H^0(D_n)\cong\Hom_{D(W_n)}(P_n,Q_n),\\
 H^0(\mathrm{id}_D\otimes_W^{\mathbf L}j_n)(\eta_n)=a_n.
 \end{gathered}
 \tag{\theequation}\label{eq:common-period-coordinate}
\]
The element $a_n$ is the quotient coordinate in
Lemma~\ref{common:graph}; passing from $\eta_n$ to $a_n$ forgets
the specified homotopy of the reduced graph map to zero.
For the initial lifting $X_2^0$ and the decomposition of
Lemma~\ref{common:canonical}, the initial condition is precisely
\[
 H^0(\mathrm{id}_D\otimes_W^{\mathbf L}j_2)(\eta_2)
       =a_2=0\quad\text{in }H^0(D_2).
\]
The proof of Proposition~\ref{common:w3} uses this equality;
no vanishing of $\eta_2$ in
$H^0(D\otimes_W^{\mathbf L}pW_2)$ is assumed.

\begin{proposition}\label{common:w3}
The lifting $X_2^0$ of Lemma~\ref{common:canonical} extends to a
smooth projective lifting over $W_3$.
\end{proposition}

\begin{proof}
Choose the decomposition with $a_2=0$, and set $I_n=pW_n$.
There is a commutative diagram with exact rows
\[
\begin{tikzcd}[column sep=small]
0\arrow[r]&k\arrow[r,"p^2"]\arrow[d,equal]&I_3\arrow[r]\arrow[d]&I_2\arrow[r]\arrow[d]&0\\
0\arrow[r]&k\arrow[r,"p^2"]&W_3\arrow[r]&W_2\arrow[r]&0.
\end{tikzcd}
\]
After tensoring with $D$, \eqref{eq:common-period-coordinate}
and naturality of the connecting homomorphisms show that
the connecting images of $\eta_2$ and $a_2$ agree in
$H^1(D_1)$. Since $a_2=0$, we have
$\delta_{W_3/W_2}(\eta_2)=0$.
Since $W_3\to W_2$ is a divided-power elementary extension
\cite[Corollary~7.14]{BT}, \eqref{eqn:obstruction_period}
and the injectivity of $\theta_1$ show that $X_2^0$ extends
over $W_3$.
An ample line bundle extends because $H^2(X,\OO_X)=0$, so the
lifting is projective.
\end{proof}

This proves Theorem~\ref{main:W3} for $d\geq3$.
For $d=1,2$, the assertion follows from $W(k)$-liftability
(see \cite[Proposition~2.6]{Ito18} for $d=2$).

\subsection{Algebraization and the torsion-free case}
\begin{lemma}\label{common:algebraization}
A compatible sequence of smooth proper liftings of $X$ over all
$W_n$ is induced by a smooth projective lifting over $W$.
\end{lemma}

\begin{proof}
Since $H^2(X,\OO_X)=0$, an ample line bundle extends along the
sequence. The algebraization theorem \cite[Tag~089A]{Stacks}
gives a projective $W$-scheme inducing the formal lifting.
It is smooth along the special fibre; properness then implies
smoothness over $W$.
\end{proof}

The next result is an application of
\cite[Theorem~7.16]{BT}. The decomposition of Lemma~\ref{common:split}
allows us to check its torsion-freeness condition using only
$H^d_{\cris}(X/W)$.

\begin{proposition}\label{common:torsion-free}
If $H^d_{\cris}(X/W)$ is torsion-free, then the deformations of $X$
over local Artinian $W$-algebras are unobstructed.
In particular, every smooth $W_2$-lifting of $X$ extends to a smooth
projective lifting over $W$.
\end{proposition}

\begin{proof}
Take the decomposition $C\simeq P\oplus Q$ of
Lemma~\ref{common:split}, with $Q=W\oplus W[-d]$.
Since the projection induces an isomorphism
$H^0(C)\to H^0(Q)$, we have $H^0(P)=0$. Hence
\[
 H^0\mathbf R\Hom_W(Q,P)
 \simeq H^0(P)\oplus H^d(P)
 \simeq H^d(P).
\]
This module is torsion-free, being a direct summand of
$H^d(C)=H^d_{\cris}(X/W)$.

Conditions~(3) and~(4) of \cite[Theorem~7.16]{BT}
were verified in Subsection~2.6 for the maps $e_X$ and $\kappa$.
Thus \cite[Theorem~7.16]{BT} implies that $\Def_X$ is unobstructed.

Starting from any smooth $W_2$-lifting of $X$, we can therefore
choose successive extensions over $W_n$ forming a compatible
system. Lemma~\ref{common:algebraization} realizes this system
as a smooth projective lifting over $W$.
\end{proof}

\section{A lifting criterion}\label{sec:tower}

In this section, let $X$ be a strict Calabi--Yau variety of dimension $d\geq3$.
Suppose that its crystalline complex admits a decomposition
\[
 C\simeq P\oplus Q,\qquad Q\simeq W\oplus W[-d],
\]
for which the composite
$Q\otimes_W^{\mathbf L}k\to C\otimes_W^{\mathbf L}k
\to\RG(X,\OO_X)$ is an equivalence in $\mathcal D(k)$.
This decomposition is fixed throughout this section.
For quasi-$F$-split $X$, its existence follows from
Lemma~\ref{common:split}.
Thus the reduction $P_1=P\otimes_W^{\mathbf L}k$ of $P$ is the
first Hodge filtration $\RG(X,\Fil^1\Omega^\bullet_{X/k})$.
Fix an equivalence $s:P\oplus Q\xrightarrow{\sim}C$
in $\mathcal D(W)$ realizing this decomposition.
Retain the notation $C_A,P_A,Q_A,D_A$ of that section; in particular,
$D:=\mathbf R\Hom_W(P,Q)$ and $D_r:=D\otimes_W^{\mathbf L}W_r$.
Fix a smooth projective lifting $X_2^0/W_2$, and let
\[
 q_{\mathrm H,2}:C_2\longrightarrow
 V_2:=\RG(X_2^0,\OO_{X_2^0})
\]
be its Hodge quotient. The restriction
$\beta_2:=q_{\mathrm H,2}|_{Q_2}:Q_2\to V_2$
is an equivalence in $\mathcal D(W_2)$ by derived Nakayama.
We assume that
\[
 a_2:=\beta_2^{-1}\circ q_{\mathrm H,2}|_{P_2}=0
 \quad\text{in }\Hom_{D(W_2)}(P_2,Q_2)\cong H^0(D_2).
\]
Reducing this equality gives $a_1=0$. We may therefore choose
the reduction data $(\alpha,h_0)$ of Section~2 for $s$;
fix these data for all the graph coordinates below.
Lemma~\ref{common:canonical} constructs such a decomposition
and lifting under its hypotheses; the following criterion
applies to any pair satisfying the assumptions above.

The following is a variant of the lifting argument in
\cite[Theorem~7.16]{BT}. In place of its torsion-freeness condition,
we use the initial zero graph and edge surjectivity over finite
Witt rings. The conclusion concerns extension of the chosen
lifting, rather than unobstructedness of all deformations.

\begin{proposition}\label{tower:extension}
Under these assumptions, suppose that, for every $r\geq1$ and every
smooth proper lifting
$Y/W_{r+1}$ of $X$, the first Hodge edge map
\[
 H^{d-1}\bigl(Y_r,\Fil^1\Omega^\bullet_{Y_r/W_r}\bigr)
 \longrightarrow H^{d-2}(Y_r,\Omega^1_{Y_r/W_r}),
 \qquad Y_r=Y\otimes_{W_{r+1}}W_r,
\]
is surjective. Then $X_2^0$ extends to a smooth projective $W$-lifting.
\end{proposition}

The edge-map hypothesis will be in force throughout this section.
For a lifting $X_r/W_r$, put
\[
 V_r:=\RG(X_r,\OO_{X_r}),\qquad
 K_r:=\RG(X_r,\Fil^1\Omega^\bullet_{X_r/W_r}).
\]
\subsection{Relative periods and coefficient maps}

We need a relative version of the differential and obstruction
comparisons for the period map. The obstruction comparison below is stated for extensions
$A'\twoheadrightarrow A$ whose kernel is finite free over $W_r$;
this includes the extensions used in Lemma~\ref{tower:pd-curves}.

Fix $r\geq1$ and a lifting $Y=X_r$ over $W_r$.
Crystalline--de Rham comparison gives
\[
 c_Y:C_r\xrightarrow{\sim}\RG(Y,\Omega^\bullet_{Y/W_r}).
\]
Use this comparison to define the Hodge quotient
\[
 q_Y:=\RG(\operatorname{pr}_0)\circ c_Y:
 C_r\longrightarrow V_r.
\]
Thus $q_Y$ is the map denoted $q_{\mathrm H,r}$ in
Lemma~\ref{common:graph}, and its fibre is $K_r$.
Put
\[
 \beta_r:=q_Y\circ s_r|_{Q_r},\qquad
 a_r:=\beta_r^{-1}\circ q_Y\circ s_r|_{P_r}.
\]
Lemma~\ref{common:graph} identifies $V_r$ with $Q_r$ via
$\beta_r$ and $K_r$ with $P_r$ via the fibre inclusion
$s_r\circ(\mathrm{id},-a_r)$. Together with the section
$s_r|_{Q_r}\circ\beta_r^{-1}:V_r\to C_r$, these give a splitting
\[
 C_r\simeq K_r\oplus V_r
\]
under which $q_Y$ is the second projection. We use these
identifications, including their homotopy data, compatibly
with reduction and with the fixed data $(s,\alpha,h_0)$.
In particular, at $r=1$ the equivalence $P_1\simeq K_1$ is
the equivalence $u$ used in Section~2 to define $\theta_1$.
For every $W_r$-algebra $A$, we have
$C_r\otimes_{W_r}^{\mathbf L}A\simeq C_A$.

An augmented PD $W_r$-algebra means a local Artinian algebra
\[
 W_r\longrightarrow A\longrightarrow W_r
\]
whose composite is the identity, together with divided powers
on $J_A=\ker(A\to W_r)$; morphisms preserve both structure maps
and the divided powers. The relative PD constructions use the
zero PD ideal on $W_r$.
To compare with Section~2, put
\[
 I_A=\ker(A\to W_r\to k)=pA+J_A.
\]
The divided powers on $J_A$ and the canonical divided powers
on $pA$ agree on $pA\cap J_A=pJ_A$, and hence combine to give
divided powers on $I_A$ compatible with those on $pW$.
With this absolute PD structure, $W_r\to A\to W_r$ also consists
of morphisms in $\operatorname{Art}^{\mathrm{pd}}_W(k)$, where
$W_r$ has its canonical divided powers on $pW_r$.

Let $\Def_{Y/W_r}(A)$ be the space of smooth proper deformations
of $Y$ over $A$ with a specified identification of their
reduction over $W_r$ with $Y$. Using $\mathcal U_C$ from
Definition~\ref{common:period-domain}, define
\[
 \Per_{Y/W_r}(A)
 :=\operatorname{hofib}_{q_Y}\bigl(
 \mathcal U_C(A)\longrightarrow\mathcal U_C(W_r)\bigr).
\]
Explicitly, a point is a quadruple $(V_A,q_A,\rho_A,h_A)$ with
$V_A$ perfect and
\[
 \begin{gathered}
 q_A:C_A\longrightarrow V_A,\qquad
 \rho_A:V_A\otimes_A^{\mathbf L}W_r\xrightarrow{\sim}V_r,\\
 h_A:\rho_A\circ(q_A\otimes_A^{\mathbf L}W_r)\simeq q_Y
 \quad\text{in }\operatorname{Map}_{\mathcal D(W_r)}(C_r,V_r).
 \end{gathered}
\]
The equivalences between such points retain all coherent
homotopies, as in Definition~\ref{common:period-domain}.
Thus the source is still $C_A$, and the reduction is now fixed
over $W_r$. If $y=\operatorname{per}_{X,W_r}(Y)\in\Per_X(W_r)$,
including the marking over $k$, associativity of homotopy
pullbacks gives
\refstepcounter{equation}
\[
 \Per_{Y/W_r}(A)\simeq
 \operatorname{hofib}_{y}\bigl(
 \Per_X(A)\longrightarrow\Per_X(W_r)\bigr).
  \tag{\theequation}\label{eq:relative-period-fibre}
\]
Its base point is the constant quotient
$q_Y\otimes_{W_r}^{\mathbf L}A:C_A\to V_r\otimes_{W_r}^{\mathbf L}A$,
with its natural reduction data.

For a relative period point, define, as in Section~2,
\[
 \beta_A:=q_A\circ s_A|_{Q_A},\qquad
 a_A:=\beta_A^{-1}\circ q_A\circ s_A|_{P_A}.
\]
Here $\beta_A$ is an equivalence in $\mathcal D(A)$
by derived Nakayama.
The relative marking gives a specified homotopy
$a_A\otimes_A^{\mathbf L}W_r\simeq a_r$, so the graph description
already gives
\[
 \Per_{Y/W_r}(A)\simeq
 \operatorname{hofib}_{a_r}\bigl(
 \Omega^\infty D_A\longrightarrow\Omega^\infty D_r\bigr).
\]

To place the constant quotient at the origin, we now change
coordinates by the automorphism
\[
 \tau_r:=
 \begin{pmatrix}
  \mathrm{id}&0\\
  -a_r&\mathrm{id}
 \end{pmatrix}
 :P_r\oplus Q_r\xrightarrow{\sim}P_r\oplus Q_r,
 \qquad s'_r:=s_r\circ\tau_r.
\]
Its first summand maps by $s_r\circ(\mathrm{id},-a_r)$,
so $s'_r$ realizes the splitting $C_r\simeq K_r\oplus V_r$
described above. Writing
$s'_A:=s'_r\otimes_{W_r}^{\mathbf L}A$, we obtain
\[
 \beta_A^{-1}\circ q_A\circ s'_A
 =
 \bigl(a_A-a_r\otimes_{W_r}^{\mathbf L}A,\mathrm{id}\bigr).
\]
The relative marking supplies a nullhomotopy of the reduction
of the first component. Thus this change of coordinates gives
\refstepcounter{equation}
\[
 \begin{aligned}
 \operatorname{Graph}_{Y/W_r,A}:\Per_{Y/W_r}(A)
 &\xrightarrow{\sim}\Omega^\infty\operatorname{fib}(D_A\to D_r)\\
 &\simeq\Omega^\infty(D_r\otimes_{W_r}^{\mathbf L}J_A),
 \end{aligned}
  \tag{\theequation}\label{eq:relative-graph-chart}
\]
which is the graph chart associated with
$\mathbf R\Hom_{W_r}(K_r,V_r)\simeq D_r$.

\begin{remark}
\begin{enumerate}
\item
In Section~2, the fixed splitting $s:P\oplus Q\simeq C$
and the reduction data $(\alpha,h_0)$ give a base point
$o_A\in\Per_X(A)$ represented by
\[
 C_A\xrightarrow{\operatorname{pr}_{Q_A}\circ s_A^{-1}}Q_A,
 \qquad
 \operatorname{Graph}_A(o_A)=0.
\]
No further translation or $k$-algebra structure on $A$
is required.

\item
Here the structure map $W_r\to A$ gives a relative base point
$o_{Y/W_r,A}\in\Per_{Y/W_r}(A)$ represented by the constant
quotient
\[
 q_Y\otimes_{W_r}^{\mathbf L}A:
 C_A\longrightarrow V_r\otimes_{W_r}^{\mathbf L}A,
\]
with its canonical reduction data.
For this point, the original and relative graph coordinates
satisfy
\[
 a_A=a_r\otimes_{W_r}^{\mathbf L}A,
 \qquad
 \operatorname{Graph}_{Y/W_r,A}(o_{Y/W_r,A})=0.
\]
The second equality uses the translation
in \eqref{eq:relative-graph-chart}; it does not require
$a_r=0$ in the original coordinates.
\end{enumerate}
\end{remark}

\begin{lemma}\label{tower:relative-period}
With the notation above, the following assertions hold.
\begin{enumerate}
\item\label{tower:relative-period-differential}
There is a natural relative period map
\[
 \operatorname{per}_{Y/W_r}:
 \Def_{Y/W_r}\longrightarrow\Per_{Y/W_r}.
\]
Its differential at the constant deformation is the morphism
\[
 \Theta_Y:
 \RG(Y,T_{Y/W_r})[1]
 \longrightarrow
 \mathbf R\Hom_{W_r}(K_r,V_r)\simeq D_r
\]
adjoint to
\[
 \begin{aligned}
 \RG(Y,T_{Y/W_r})[1]\otimes_{W_r}^{\mathbf L}K_r
 &\longrightarrow
 \RG(Y,T_{Y/W_r})\otimes_{W_r}^{\mathbf L}
       \RG(Y,\Omega^1_{Y/W_r})\\
 &\longrightarrow V_r.
 \end{aligned}
\]
The first arrow is induced by
$\operatorname{pr}_1:
\Fil^1\Omega^\bullet_{Y/W_r}\to\Omega^1_{Y/W_r}[-1]$,
and the second is cup product followed by evaluation.

More precisely, let $M$ be a connective perfect $W_r$-complex.
Here $\operatorname{sqz}^{\mathrm{pd}}_{W_r}(M)$ denotes the
trivial PD square-zero extension of $W_r$ by $M$ from
\cite[Definition~4.32]{BT}.
There is a commutative diagram of pointed spaces,
natural in $M$,
\[
 \begin{tikzcd}[column sep=large]
 \Def_{Y/W_r}(\operatorname{sqz}^{\mathrm{pd}}_{W_r}(M))
   \arrow[r,"\operatorname{per}_{Y/W_r}"]
   \arrow[d,"\simeq"']
 &
 \Per_{Y/W_r}(\operatorname{sqz}^{\mathrm{pd}}_{W_r}(M))
   \arrow[d,"\simeq"]
 \\
 \Omega^\infty\!
 \bigl(\RG(Y,T_{Y/W_r})[1]
       \otimes_{W_r}^{\mathbf L}M\bigr)
   \arrow[r,"\Omega^\infty(\Theta_Y\otimes^{\mathbf L}M)"']
 &
 \Omega^\infty\!
 \bigl(D_r\otimes_{W_r}^{\mathbf L}M\bigr).
 \end{tikzcd}
\]
The base points are the constant deformation
$\OO_Y\otimes_{W_r}^{\mathbf L}\operatorname{sqz}^{\mathrm{pd}}_{W_r}(M)$ and the zero graph
coordinate, respectively.

\item\label{tower:relative-period-obstruction}
Let $A'\twoheadrightarrow A$ be a surjection of augmented
PD $W_r$-algebras with kernel $I$. Suppose that $I$ is finite
free over $W_r$ and that
\[
 J_{A'}I=0,\qquad
 \gamma_m(x)=0\quad(x\in I,\ m\geq2).
\]
For $\mathcal Y\in\Def_{Y/W_r}(A)$, write
\[
 \eta_{\mathcal Y}^{\mathrm{rel}}
 \in H^0(D_r\otimes_{W_r}^{\mathbf L}J_A)
\]
for the image of its relative period under
$\pi_0(\operatorname{Graph}_{Y/W_r,A})$, and let
\[
 \operatorname{ob}(\mathcal Y,A'/A)
 \in H^2(Y,T_{Y/W_r})\otimes_{W_r} I
\]
be its geometric obstruction class. Then
\[
 H^1(\Theta_Y\otimes_{W_r}^{\mathbf L}I)
 \bigl(\operatorname{ob}(\mathcal Y,A'/A)\bigr)
 =
 \delta\bigl(\eta_{\mathcal Y}^{\mathrm{rel}}\bigr),
\]
where
\[
 \delta:
 H^0(D_r\otimes_{W_r}^{\mathbf L}J_A)
 \longrightarrow
 H^1(D_r\otimes_{W_r}^{\mathbf L}I)
\]
is the boundary associated with
$0\to I\to J_{A'}\to J_A\to0$.
We use the compatible obstruction and boundary conventions
of Section~2.

\item\label{tower:relative-period-compatibility}
For $1\leq s\leq r$, put $Y_s=Y\otimes_{W_r} W_s$.
Under the derived base-change and graph identifications,
\[
 \Theta_Y\otimes_{W_r}^{\mathbf L}W_s
 \simeq \Theta_{Y_s}.
\]
\end{enumerate}
\end{lemma}

\begin{proof}
The proof follows the arguments of
\cite[Theorem~6.18 and Proposition~7.15]{BT},
adapted to the relative setting over $W_r$.
We focus on the points requiring verification over $W_r$,
making explicit the period map on square-zero extensions
and its compatibility with obstruction classes.

We first construct the relative period map.
For a deformation $\mathcal Y/A$, divided-power de Rham
invariance and base change give
\[
 \begin{aligned}
 C_A
 &\simeq
 \RG(Y,\Omega^\bullet_{Y/W_r})
       \otimes_{W_r}^{\mathbf L}A\\
 &\simeq
 \RG(\mathcal Y,\Omega^\bullet_{\mathcal Y/A})
 \longrightarrow \RG(\mathcal Y,\OO_{\mathcal Y}),
 \end{aligned}
\]
where the first equivalence uses $c_Y$ and the last map
is induced by the projection
$\Omega^\bullet_{\mathcal Y/A}\to\OO_{\mathcal Y}$;
see \cite[Theorem~5.22 and Lemma~6.12]{BT}.
The specified identification
$\mathcal Y\otimes_A W_r\simeq Y$ identifies the reduction
of the displayed composite with $q_Y$, including the
required homotopy.
This composite therefore defines
$\operatorname{per}_{Y/W_r}(A)(\mathcal Y)$.
To compare with Section~2, recall that $I_A=pA+J_A$.
The comparison using $I_A$ and the comparison through $Y$
using $J_A$ fit into the homotopy commutative diagram
\[
 \begin{tikzcd}[column sep=small]
 C_A\arrow[r,"\sim","I_A"']
    \arrow[d,"c_Y\otimes_{W_r}^{\mathbf L}A"']&
 \RG(\mathcal Y,\Omega^\bullet_{\mathcal Y/A})\\
 \RG(Y,\Omega^\bullet_{Y/W_r})
       \otimes_{W_r}^{\mathbf L}A
       \arrow[ur,"\sim"',"J_A"]&
 \end{tikzcd}
\]
by the naturality of
\cite[Corollary~5.12 and Theorem~5.22]{BT}, compatibly
with reduction to $W_r$.
Thus composing either path with projection to
$\RG(\mathcal Y,\OO_{\mathcal Y})$ gives the same period
point, with the same reduction marking.
It follows from \eqref{eq:relative-period-fibre} that
$\operatorname{per}_{Y/W_r}$ is the map on homotopy fibres
induced by the period map of Section~2.

We also use augmented animated PD $W_r$-algebras in the
sense of \cite[Definition~4.4]{BT}, with specified PD
quotient $W_r$ and underlying Artinian animated
$W$-algebra as in \cite[Definition~3.1]{BT}.
We write $\Def_X^{\mathrm{an}}$ for the functor of
\cite[Construction~3.22 and Remark~3.23]{BT}:
$\Def_X^{\mathrm{an}}(A)$ is the space of sheaves
$\mathcal B$ of animated $A$-algebras on $X_{\mathrm{Zar}}$
equipped with an equivalence of sheaves of animated $k$-algebras
\[
 \mathcal B\otimes_A^{\mathbf L}k
 \xrightarrow{\sim}\OO_X,
\]
with equivalences of such pairs and all higher homotopies
retained.
For ordinary $A$, this agrees with $\Def_X(A)$ from
Section~2 by \cite[Corollaries~3.32 and~3.35]{BT}.
For an augmented animated PD $W_r$-algebra $A$, define
\[
 \Def_{Y/W_r}(A)
 :=
 \operatorname{hofib}_Y\bigl(
 \Def_X^{\mathrm{an}}(A)\to
 \Def_X^{\mathrm{an}}(W_r)\bigr).
\]
For ordinary $A$, this recovers the space of deformations
equipped with an identification of their reduction over
$W_r$ with $Y$.

To extend the period domain to animated $A$, use the
same definition by perfect quotients.
In the graph description, define
\[
 J_A:=\operatorname{fib}(A\to W_r)
 \quad\text{in }\mathcal D(W_r).
\]
The graph argument giving \eqref{eq:relative-graph-chart}
then yields
\[
 \Per_{Y/W_r}(A)
 \simeq
 \Omega^\infty
 \bigl(D_r\otimes_{W_r}^{\mathbf L}J_A\bigr).
\]
For animated bases, the derived divided-power de Rham
construction of \cite[Definition~5.6]{BT} gives the
relative period map by the same construction as above.
The construction commutes with base change by
\cite[Proposition~5.11, Corollary~5.12,
Theorem~5.22, and Lemma~6.12]{BT}.
The hypothesis of Lemma~6.12 holds because every morphism
of augmented PD $W_r$-algebras induces
$\mathrm{id}_{W_r}$ on the PD quotients.
For a diagram $A_1\to A_0\leftarrow A_2$ whose two maps
are surjective on $\pi_0$, we have
\refstepcounter{equation}
\[
 \begin{gathered}
 F(A_1\times^h_{A_0}A_2)
 \simeq F(A_1)\times^h_{F(A_0)}F(A_2),\\
 F=\Def_{Y/W_r}\ \text{or}\ \Per_{Y/W_r}.
 \end{gathered}
 \tag{\theequation}\label{eq:relative-period-pullbacks}
\]
For $\Def_{Y/W_r}$, apply homotopy base change along the
point
\[
 \{Y\}\longrightarrow\Def_X^{\mathrm{an}}(W_r)
\]
to the pullback equivalence for $\Def_X^{\mathrm{an}}$
from \cite[Proposition~3.25]{BT}.
Forgetting the PD structure preserves the underlying
pullback of animated rings by
\cite[Corollary~4.7]{BT}, so the same argument applies
to the PD algebras considered here.
For $\Per_{Y/W_r}$, it follows from
\eqref{eq:relative-graph-chart}, since
$A\mapsto J_A$, derived tensor product, and $\Omega^\infty$
preserve the homotopy pullbacks in
\eqref{eq:relative-period-pullbacks}.

We next prove (1).
Let $M$ be a connective perfect $W_r$-complex.
Here connective means $H^i(M)=0$ for $i>0$.
The underlying algebra of
$\operatorname{sqz}^{\mathrm{pd}}_{W_r}(M)$ is $W_r\oplus M$;
for discrete $M$, its divided powers satisfy
$\gamma_m(x)=0$ for $x\in M$ and $m\geq2$.
At the constant deformation with structure sheaf
$\OO_Y\otimes_{W_r}^{\mathbf L}
\operatorname{sqz}^{\mathrm{pd}}_{W_r}(M)$ and the period
point with zero graph coordinate, respectively, we have
\refstepcounter{equation}
\[
 \begin{aligned}
 \Omega\Def_{Y/W_r}(\operatorname{sqz}^{\mathrm{pd}}_{W_r}(M))
 &\simeq
 \operatorname{Map}_{\OO_Y}
 \bigl(\Omega^1_{Y/W_r},
       \OO_Y\otimes_{W_r}^{\mathbf L}M\bigr),\\
 \Omega\Per_{Y/W_r}(\operatorname{sqz}^{\mathrm{pd}}_{W_r}(M))
 &\simeq
 \operatorname{Map}_{\mathcal D(W_r)}
 \bigl(K_r[1],V_r\otimes_{W_r}^{\mathbf L}M\bigr).
 \end{aligned}
 \tag{\theequation}\label{eq:relative-period-loops}
\]
Here $\Omega$ denotes the based loop space, and
$\operatorname{Map}_{\OO_Y}$ denotes the mapping space
in the derived $\infty$-category of $\OO_Y$-modules.
In the first line, a loop is an automorphism of the constant
deformation with a specified homotopy of its reduction
to the identity.
By the scalar-extension adjunction, such an automorphism
corresponds to a section, in sheaves of animated
$W_r$-algebras, of the projection
$\OO_Y\oplus(\OO_Y\otimes_{W_r}^{\mathbf L}M)\to\OO_Y$,
and hence to a derivation.
The square-zero adjunction of \cite[Proposition~2.31]{BT}
identifies these derivations with the first mapping space
in \eqref{eq:relative-period-loops}, since smoothness gives
$L_{Y/W_r}\simeq\Omega^1_{Y/W_r}$ for the cotangent complex.
The second line follows from
\eqref{eq:relative-graph-chart} and
$D_r\simeq\mathbf R\Hom_{W_r}(K_r,V_r)$.

To compute $\Omega\operatorname{per}_{Y/W_r}(\operatorname{sqz}^{\mathrm{pd}}_{W_r}(M))$
under \eqref{eq:relative-period-loops}, use the natural diagram
of fibre sequences of sheaves of $W_r$-modules from
\cite[Proposition~5.16]{BT}:
\refstepcounter{equation}
\[
 \begin{tikzcd}[column sep=small]
 \Fil^1\Omega^\bullet_{Y/W_r}
   \arrow[r]\arrow[d,"\operatorname{pr}_1"']&
 \Omega^\bullet_{Y/W_r}\arrow[r]\arrow[d]&
 \OO_Y\arrow[d,"d"]\\
 \Omega^1_{Y/W_r}[-1]\arrow[r]&
 0\arrow[r]&
 \Omega^1_{Y/W_r}.
 \end{tikzcd}
 \tag{\theequation}\label{eq:relative-hodge-projection}
\]
Here $\operatorname{pr}_1$ is the quotient to
$\operatorname{gr}^1\Omega^\bullet_{Y/W_r}
=\Omega^1_{Y/W_r}[-1]$, and $d$ is the universal derivation.
To apply the arguments of \cite[Lemmas~6.21--6.23]{BT}
over $W_r$, we use the base-change compatibility established
above and the PD cotangent adjunction of
\cite[Proposition~4.36]{BT}.
For this adjunction, $W_r$ and $\OO_Y$ carry zero PD
ideals, so \cite[Proposition~4.27]{BT} and smoothness give
$L^{\mathrm{pd}}_{Y/W_r}\simeq L_{Y/W_r}
\simeq\Omega^1_{Y/W_r}$.
For a map
$\lambda:\Omega^1_{Y/W_r}\to
\OO_Y\otimes_{W_r}^{\mathbf L}M$,
the arguments of \cite[Lemmas~6.22--6.23]{BT},
applied to \eqref{eq:relative-hodge-projection}, show that
the image of the loop represented by $\lambda$ under
$\Omega\operatorname{per}_{Y/W_r}(\operatorname{sqz}^{\mathrm{pd}}_{W_r}(M))$ is obtained by
applying derived global sections to
\[
 \Fil^1\Omega^\bullet_{Y/W_r}[1]
 \xrightarrow{\operatorname{pr}_1[1]}
 \Omega^1_{Y/W_r}
 \xrightarrow{\lambda}
 \OO_Y\otimes_{W_r}^{\mathbf L}M.
\]
Taking derived global sections gives
\refstepcounter{equation}
\[
 \lambda\longmapsto
 \left(
 K_r[1]
 \xrightarrow{\RG(\operatorname{pr}_1)[1]}
 \RG(Y,\Omega^1_{Y/W_r})
 \xrightarrow{\RG(\lambda)}
 V_r\otimes_{W_r}^{\mathbf L}M
 \right).
 \tag{\theequation}\label{eq:relative-period-loop-map}
\]
Here the projection formula identifies
$\RG(Y,\OO_Y\otimes_{W_r}^{\mathbf L}M)$ with
$V_r\otimes_{W_r}^{\mathbf L}M$.
Formula \eqref{eq:relative-period-loop-map} describes
$\Omega\operatorname{per}_{Y/W_r}(\operatorname{sqz}^{\mathrm{pd}}_{W_r}(M))$ as a map of spaces,
naturally in $M$ and including the compatibility homotopies.

To recover the period map itself, apply
\eqref{eq:relative-period-pullbacks} to
\[
 \operatorname{sqz}^{\mathrm{pd}}_{W_r}(M)\simeq W_r\times^h_{\operatorname{sqz}^{\mathrm{pd}}_{W_r}(M[1])}W_r.
\]
Since both relative functors take $W_r$ to a contractible
space, this gives
\[
 F(\operatorname{sqz}^{\mathrm{pd}}_{W_r}(M))\simeq\Omega F(\operatorname{sqz}^{\mathrm{pd}}_{W_r}(M[1])),
 \qquad F=\Def_{Y/W_r},\Per_{Y/W_r}.
\]
Naturality identifies $\operatorname{per}_{Y/W_r}(\operatorname{sqz}^{\mathrm{pd}}_{W_r}(M))$
with $\Omega\operatorname{per}_{Y/W_r}(\operatorname{sqz}^{\mathrm{pd}}_{W_r}(M[1]))$
under these equivalences.
Thus \eqref{eq:relative-period-loops} with $M[1]$ gives
\refstepcounter{equation}
\[
 \Def_{Y/W_r}(\operatorname{sqz}^{\mathrm{pd}}_{W_r}(M))
 \simeq
 \Omega^\infty\bigl(
 \RG(Y,T_{Y/W_r})[1]\otimes_{W_r}^{\mathbf L}M\bigr).
 \tag{\theequation}\label{eq:relative-period-tangent-space}
\]
Applying \eqref{eq:relative-period-loop-map} with $M[1]$
gives a map
\[
 K_r[1]\longrightarrow
 \bigl(V_r\otimes_{W_r}^{\mathbf L}M\bigr)[1].
\]
Shifting this map by $[-1]$ gives the corresponding
graph coordinate
\[
 K_r\longrightarrow V_r\otimes_{W_r}^{\mathbf L}M.
\]
Under \eqref{eq:relative-period-tangent-space} and
\eqref{eq:relative-graph-chart}, we therefore obtain
\refstepcounter{equation}
\[
 \operatorname{per}_{Y/W_r}(\operatorname{sqz}^{\mathrm{pd}}_{W_r}(M))
 \simeq
 \Omega^\infty(\Theta_Y\otimes_{W_r}^{\mathbf L}M).
 \tag{\theequation}\label{eq:relative-period-square-zero}
\]
Tensor--Hom adjunction applied to
\eqref{eq:relative-period-loop-map}, after replacing $M$
by $M[1]$ and shifting the resulting map by $[-1]$,
identifies $\Theta_Y$ with the composite
of the first Hodge projection, cup product, and evaluation
in (1).
This proves (1).

We now prove (2).
For the extension $A'\twoheadrightarrow A$ in the statement,
$I^2=0$ and all higher divided powers on $I$ vanish.
Both PD quotients are $W_r$.
Hence \cite[Proposition~4.39]{BT} gives, for the relative
PD cotangent complex,
$\tau_{\geq-1}L^{\mathrm{pd}}_{A/A'}\simeq I[1]$.
Since $J_{A'}I=0$, the $A'$-module structure on $I$
factors through the augmentation $A'\to W_r$.
By the PD cotangent adjunction
\cite[Corollary~4.35]{BT}, the truncation map
$L^{\mathrm{pd}}_{A/A'}\to I[1]$
corresponds to the first arrow in
\[
 A
 \longrightarrow
 \operatorname{sqz}^{\mathrm{pd}}_A(I[1])
 \longrightarrow
 \operatorname{sqz}^{\mathrm{pd}}_{W_r}(I[1]).
\]
The first arrow is a morphism of PD $A'$-algebras over $A$.
The second is induced by the augmentation $A\to W_r$
and the identity on $I[1]$, since the $A$-module structure
on $I$ factors through $W_r$.
The $A'$-algebra compatibility gives a specified homotopy
between the composites
\[
 \begin{gathered}
 A'\longrightarrow A\longrightarrow \operatorname{sqz}^{\mathrm{pd}}_{W_r}(I[1])\\
 \text{and}\\
 A'\longrightarrow W_r\longrightarrow \operatorname{sqz}^{\mathrm{pd}}_{W_r}(I[1]),
 \end{gathered}
\]
where $W_r\to \operatorname{sqz}^{\mathrm{pd}}_{W_r}(I[1])$ is the zero section.
The specified homotopy gives a canonical comparison map
\[
 A'\longrightarrow
 A\times^h_{\operatorname{sqz}^{\mathrm{pd}}_{W_r}(I[1])}W_r.
\]
We show that this map is an equivalence by the argument of
\cite[Proposition~7.15]{BT}.
The source and target both map to $A$, with fibre $I$,
and the comparison induces the identity on $A$.
By the Hurewicz description in
\cite[Proposition~4.39]{BT}, the composite
\[
 I[1]\longrightarrow L^{\mathrm{pd}}_{A/A'}
 \longrightarrow
 \tau_{\geq-1}L^{\mathrm{pd}}_{A/A'}
 \simeq I[1]
\]
is the identity.
Thus the comparison also induces the identity on the
fibre $I$ and is therefore an equivalence.
Consequently, the square is a homotopy pullback:
\refstepcounter{equation}
\[
 A'\simeq
 A\times^h_{\operatorname{sqz}^{\mathrm{pd}}_{W_r}(I[1])}W_r.
 \tag{\theequation}\label{eq:relative-obstruction-pullback}
\]

Applying \eqref{eq:relative-period-pullbacks} to
\eqref{eq:relative-obstruction-pullback} and using the
naturality of the period map gives a commutative diagram
of homotopy fibre sequences
\refstepcounter{equation}
\[
 \begin{tikzcd}[column sep=small]
 \Def_{Y/W_r}(A')\arrow[r]\arrow[d]&
 \Def_{Y/W_r}(A)\arrow[r]\arrow[d]&
 \Def_{Y/W_r}(\operatorname{sqz}^{\mathrm{pd}}_{W_r}(I[1]))\arrow[d]\\
 \Per_{Y/W_r}(A')\arrow[r]&
 \Per_{Y/W_r}(A)\arrow[r]&
 \Per_{Y/W_r}(\operatorname{sqz}^{\mathrm{pd}}_{W_r}(I[1])).
 \end{tikzcd}
 \tag{\theequation}\label{eq:relative-obstruction-diagram}
\]
The obstruction construction from a square-zero extension
(cf. \cite[Definition~3.12]{BT}), applied to
\eqref{eq:relative-obstruction-pullback}, identifies the
image of $[\mathcal Y]$ under
\[
 \pi_0\Def_{Y/W_r}(A)
 \longrightarrow
 \pi_0\Def_{Y/W_r}(\operatorname{sqz}^{\mathrm{pd}}_{W_r}(I[1]))
 \simeq H^2(Y,T_{Y/W_r})\otimes_{W_r}I
\]
with $\operatorname{ob}(\mathcal Y,A'/A)$.
By \eqref{eq:relative-period-square-zero} with $M=I[1]$,
the rightmost vertical map in
\eqref{eq:relative-obstruction-diagram}
induces on $\pi_0$ the homomorphism
\[
 H^1(\Theta_Y\otimes_{W_r}^{\mathbf L}I):
 H^2(Y,T_{Y/W_r})\otimes_{W_r}I
 \longrightarrow H^1(D_r\otimes_{W_r}^{\mathbf L}I),
\]
where we use that $I$ is finite free.
Taking augmentation fibres in
\eqref{eq:relative-obstruction-pullback} identifies the
map $J_A\to I[1]$ induced by $A\to \operatorname{sqz}^{\mathrm{pd}}_{W_r}(I[1])$ with the
connecting map of the triangle
\[
 I\longrightarrow J_{A'}\longrightarrow J_A
 \longrightarrow I[1].
\]
Under \eqref{eq:relative-graph-chart}, tensoring this
triangle with $D_r$ and applying $\Omega^\infty$ gives the
lower row of \eqref{eq:relative-obstruction-diagram}.
Consequently, the map
\[
 \pi_0\Per_{Y/W_r}(A)
 \longrightarrow
 \pi_0\Per_{Y/W_r}(\operatorname{sqz}^{\mathrm{pd}}_{W_r}(I[1]))
\]
identifies with
\[
 \delta:
 H^0(D_r\otimes_{W_r}^{\mathbf L}J_A)
 \longrightarrow
 H^1(D_r\otimes_{W_r}^{\mathbf L}I).
\]
The commutativity of \eqref{eq:relative-obstruction-diagram}
now gives
\[
 H^1(\Theta_Y\otimes_{W_r}^{\mathbf L}I)
 \bigl(\operatorname{ob}(\mathcal Y,A'/A)\bigr)
 =
 \delta\bigl(\eta_{\mathcal Y}^{\mathrm{rel}}\bigr),
\]
with the conventions of Section~2.
This proves (2).

Finally, $K_r$, $V_r$, and $\RG(Y,T_{Y/W_r})$ are
perfect over $W_r$, and their derived base changes to
$W_s$ are the corresponding complexes for $Y_s$.
The first Hodge projection, cup product, and evaluation
also commute with derived base change.
Applying these compatibilities to the description of
$\Theta_Y$ in (1) gives
\[
 \Theta_Y\otimes_{W_r}^{\mathbf L}W_s\simeq\Theta_{Y_s}.
\]
This proves (3).
\end{proof}

Write
\[
 \theta_r=H^1(\Theta_{X_r}):
 H^2(X_r,T_{X_r/W_r})\longrightarrow H^1(D_r).
\]
For $r=1$, formula \eqref{eq:relative-period-loop-map},
under the identifications $u:P_1\simeq K_1$ and
$\alpha:Q_1\simeq V_1$, gives
$H^1(\Theta_X)=\theta_1$ as defined in
\eqref{eq:theta1-contraction}.
For an existing lifting $X_{r+1}$, the coefficient inclusion
$k\to W_{r+1}$, $1\mapsto p^r$, gives the commutative diagram
\refstepcounter{equation}
\[
 \begin{array}{ccc}
 H^2(X,T_X)&\xrightarrow{\theta_1}&H^1(D_1)\\[2pt]
 {\scriptstyle p^r}\downarrow&&\downarrow{\scriptstyle p^r}\\[2pt]
 H^2(X_{r+1},T_{X_{r+1}/W_{r+1}})
 &\xrightarrow{\theta_{r+1}}&H^1(D_{r+1}).
 \end{array}
  \tag{\theequation}\label{eq:relative-theta-coefficients}
\]
The vertical arrows are coefficient maps over $W_{r+1}$.

The following is the analogue over $W_r$ of the injectivity
criterion in \cite[Proposition~7.6]{BT}.

\begin{lemma}\label{tower:injectivity}
If $X_r$ extends to $W_{r+1}$, then $\theta_r$ is injective.
\end{lemma}

\begin{proof}
Since $W_r$ is self-injective, relative Serre duality gives an isomorphism
\[
 \kappa_r:
 H^2(X_r,T_{X_r/W_r})
 \xrightarrow{\sim}
 \Hom_{W_r}\bigl(H^{d-2}(X_r,\Omega^1_{X_r/W_r}),W_r\bigr).
\]
Let
\[
 e_r:
 \mathbb H^{d-1}\bigl(X_r,\Fil^1\Omega^\bullet_{X_r/W_r}\bigr)
 \longrightarrow H^{d-2}(X_r,\Omega^1_{X_r/W_r})
\]
be induced by
$\Fil^1\Omega^\bullet_{X_r/W_r}\to\Omega^1_{X_r/W_r}[-1]$.
Its source is $H^{d-1}(K_r)$.

The projection $V_r\simeq W_r\oplus W_r[-d]\to W_r[-d]$
induces a map
\[
 \begin{aligned}
 \rho_r:H^1(D_r)
 &\longrightarrow
 H^1\bigl(\mathbf R\Hom_{W_r}(K_r,W_r[-d])\bigr)\\
 &\simeq \Hom_{W_r}\bigl(H^{d-1}(K_r),W_r\bigr),
 \end{aligned}
\]
where the last isomorphism again uses self-injectivity.
Formula \eqref{eq:relative-period-loop-map} gives
\[
 \rho_r\bigl(\theta_r(\xi)\bigr)
 =\kappa_r(\xi)\circ e_r
 \qquad
 \bigl(\xi\in H^2(X_r,T_{X_r/W_r})\bigr).
\]
Since $X_r$ extends to $W_{r+1}$, the hypothesis of
Proposition~\ref{tower:extension} makes $e_r$ surjective.
Thus $\theta_r(\xi)=0$ implies $\kappa_r(\xi)\circ e_r=0$,
hence $\kappa_r(\xi)=0$ and therefore $\xi=0$.
\end{proof}

\subsection{Integrating tangent directions and correcting obstructions}

For integers $r,j\geq1$, put
\[
 \mathcal A_{r,j}
 :=W_r\langle t\rangle/(\gamma_a(t):a>j),
 \qquad
 \epsilon_{r,j}:\mathcal A_{r,j}\longrightarrow W_r,
 \quad \gamma_a(t)\longmapsto0\quad(a\geq1),
\]
with its induced divided-power structure.
We use this notation for the rest of this section.

\begin{lemma}\label{tower:pd-curves}
Put $Y=X_r$, and suppose that $Y$ extends to $W_{r+1}$.
For every $u\in H^1(Y,T_{Y/W_r})$ and every finite $j\geq1$,
there exists a deformation $\mathcal Y_j/\mathcal A_{r,j}$ such that
\[
 \mathcal Y_j\otimes_{\mathcal A_{r,j}}W_r\simeq Y,
 \qquad
 \bigl[\mathcal Y_j\otimes_{\mathcal A_{r,j}}\mathcal A_{r,1}\bigr]=u.
\]
The second equality uses the usual identification
\[
 \pi_0\Def_{Y/W_r}(\mathcal A_{r,1})
 \simeq H^1(Y,T_{Y/W_r}),
 \qquad
 \mathcal A_{r,1}=W_r[t]/(t^2),
\]
determined by the generator $t$.
\end{lemma}

\begin{proof}
Choose the first-order deformation
$\mathcal Y_1/\mathcal A_{r,1}$ representing $u$.
We extend it inductively along
$\mathcal A_{r,j+1}\twoheadrightarrow\mathcal A_{r,j}$.
Write
\[
 J_j:=\ker(\epsilon_{r,j})
     =\bigoplus_{a=1}^{j}W_r\gamma_a(t).
\]
The kernel of $\mathcal A_{r,j+1}\to\mathcal A_{r,j}$ is
$
 I_{j+1}:=W_r\gamma_{j+1}(t)\simeq W_r,$
and we have
\[
 J_{j+1}I_{j+1}=0,
 \qquad
 \gamma_m(x)=0
 \quad(x\in I_{j+1},\ m\geq2).
\]
Moreover, the exact sequence
\[
 0\longrightarrow I_{j+1}
 \longrightarrow J_{j+1}
 \longrightarrow J_j
 \longrightarrow0
\]
splits as a sequence of $W_r$-modules, by sending
$\gamma_a(t)$ to $\gamma_a(t)$ for $1\leq a\leq j$.
Therefore,
\[
 \delta_j:
 H^0(D_r\otimes_{W_r}^{\mathbf L}J_j)
 \longrightarrow
 H^1(D_r\otimes_{W_r}^{\mathbf L}I_{j+1})
\]
is a zero-map, and every relative period point over $\mathcal A_{r,j}$
extends to $\mathcal A_{r,j+1}$.

Now suppose that $\mathcal Y_j$ has been constructed.
Let
$
 \eta_j^{\mathrm{rel}}
 \in H^0(D_r\otimes_{W_r}^{\mathbf L}J_j)
$
be its relative period class, and let
\[
 o_j:=
 \operatorname{ob}
 (\mathcal Y_j,\mathcal A_{r,j+1}/\mathcal A_{r,j})
 \in H^2(Y,T_{Y/W_r})\otimes_{W_r} I_{j+1}
\]
be its obstruction to extension.
Lemma~\ref{tower:relative-period} gives
\[
 (\theta_r\otimes_{W_r}\mathrm{id}_{I_{j+1}})(o_j)
 =\delta_j(\eta_j^{\mathrm{rel}})
 =0.
\]
The PD structures on the augmentation ideals $J_{j+1}$ and
$J_j$ are needed here to apply
Lemma~\ref{tower:relative-period}.
By Lemma~\ref{tower:injectivity}, $\theta_r$ is injective.
Since $I_{j+1}\simeq W_r$, it follows that $o_j=0$.
Thus there exists a deformation
$\mathcal Y_{j+1}/\mathcal A_{r,j+1}$ with
\[
 \mathcal Y_{j+1}
 \otimes_{\mathcal A_{r,j+1}}\mathcal A_{r,j}
 \simeq\mathcal Y_j,
\]
compatibly with the markings over $W_r$.
Induction preserves both the constant term $Y$ and the
first-order class $u$, proving the assertion.
\end{proof}

The use of divided powers in mixed-characteristic deformation theory
is central to Schr\"oer's $T^1$-lifting method
\cite[Sections~2--3]{Schroer03} and the tangent-lifting method of
Ekedahl--Shepherd-Barron \cite{ES05};
see also \cite[Section~7.3]{BT}.
In the following lemma, we use a finite divided-power curve to modify
a given $W_n$-lifting while preserving its reduction over $W_{n-r}$.

For a lifting $X_n/W_n$ of $X$, write
\[
 X_s:=X_n\otimes_{W_n}W_s,\qquad
 T_s:=T_{X_s/W_s}
 \quad(1\leq s\leq n).
\]
For $1\leq r<n$, let
\[
 \delta_r:H^1(X_r,T_r)\longrightarrow H^2(X,T_X)
\]
be the connecting homomorphism associated with
\refstepcounter{equation}
\[
 0\longrightarrow T_X
 \xrightarrow{\,p^r\,}T_{r+1}
 \longrightarrow T_r
 \longrightarrow0.
 \tag{\theequation}\label{eq:tower-tangent-coefficients}
\]
Let
\[
 o_n(X_n)\in H^2(X,T_X)
\]
denote the obstruction to extending $X_n$ to $W_{n+1}$.

\begin{lemma}\label{tower:correction}
Let $r\geq1$ and $n\geq r+2$, and let $X_n/W_n$ be a
lifting of $X$.
For every $u\in H^1(X_r,T_r)$, there exists a lifting
$X_n^u/W_n$ and an isomorphism
\[
 X_n^u\otimes_{W_n}W_{n-r}
 \xrightarrow{\sim}
 X_n\otimes_{W_n}W_{n-r}
\]
compatible with the markings on $X$, such that
\[
 o_n(X_n^u)-o_n(X_n)=\delta_r(u)
 \quad\text{in }H^2(X,T_X).
\]
\end{lemma}

\begin{proof}
Fix $u\in H^1(X_r,T_r)$.
Since $r+1\leq n$, the lifting $X_r/W_r$ extends to
\[
 X_{r+1}=X_n\otimes_{W_n}W_{r+1}.
\]
Applying Lemma~\ref{tower:pd-curves} with $j=n$, choose a
deformation $\mathcal Y/\mathcal A_{r,n}$ equipped with a
marking
\[
 \rho:\mathcal Y\otimes_{\mathcal A_{r,n}}W_r
 \xrightarrow{\sim}X_r
\]
such that
\[
 \bigl[\mathcal Y\otimes_{\mathcal A_{r,n}}\mathcal A_{r,1}\bigr]=u
 \quad\text{under}\quad
 \pi_0\Def_{X_r/W_r}(\mathcal A_{r,1})
 \simeq H^1(X_r,T_r).
\]
Here the first-order restriction carries the marking
induced by $\rho$.

Choose a miniversal formal deformation $\mathfrak X/R$ of $X$,
where $R$ is a complete Noetherian local $W$-algebra with
residue field $k$; see \cite[Tags~0ET5 and 06IL]{Stacks}.
Choose a presentation
$
 R=\Lambda/J$, $\Lambda:=W[[x_1,\ldots,x_m]].
$
By formal versality, choose
\[
 \varphi_n:R\longrightarrow W_n,\qquad
 \sigma_n:\varphi_n^*\mathfrak X\xrightarrow{\sim}X_n,
\]
and write $\varphi_s$ and $\sigma_s$ for their reductions
over $W_s$. Choose $a_i\in pW$ lifting $\varphi_n(x_i)$
and put $a=(a_1,\ldots,a_m)$. Then
\[
 f(a)\in p^nW\qquad(f\in J).
\]

Apply formal versality to
$\epsilon_{r,n}:\mathcal A_{r,n}\twoheadrightarrow W_r$ and the
marking $\rho$. This gives
\[
 \beta:R\longrightarrow\mathcal A_{r,n},\qquad
 \iota:\beta^*\mathfrak X\xrightarrow{\sim}\mathcal Y
\]
such that
\[
 \epsilon_{r,n}\circ\beta=\varphi_r,\qquad
 \rho\circ(\iota\otimes_{\mathcal A_{r,n}}W_r)=\sigma_r.
\]
Consequently, we can write
\[
 \bigl(\beta(x_i)\bigr)_{i=1}^{m}
 =
 (a\bmod p^r)+\sum_{j=1}^{n}\bar b_j\gamma_j(t),
 \qquad \bar b_j\in W_r^m.
\]
Choose lifts $b_j\in W^m$ of $\bar b_j$ and set
\[
 x(t):=a+\sum_{j=1}^{n}b_j\gamma_j(t),
 \qquad
 \gamma_j(t)=\frac{t^j}{j!}
 \quad\text{in }W[1/p][t].
\]
For $f\in J$, its divided-power expansion has the form
\[
 f(x(t))=f(a)+\sum_{j\geq1}c_j\gamma_j(t),
 \qquad c_j\in W.
\]
Since $\beta(f)=0$, comparison of coefficients in
$\mathcal A_{r,n}$ gives
\[
 c_j\in p^rW\quad(1\leq j\leq n),
 \qquad c_1=Df(a)b_1,
\]
where
\[
 Df(a):=
 \left(
 \frac{\partial f}{\partial x_1}(a),\ldots,
 \frac{\partial f}{\partial x_m}(a)
 \right)\in W^{1\times m},
\]
and $b_1$ is regarded as a column vector.

Put $s=n-r\geq2$. Using $v_p(j!)\leq j-1$, we obtain
\[
 \begin{aligned}
 r+sj-v_p(j!)&\geq n+1
 &&(2\leq j\leq n),\\
 sj-v_p(j!)&\geq j+1\geq n+2
 &&(j>n).
 \end{aligned}
\]
Thus
\[
 a':=x(p^s)
     =a+\sum_{j=1}^{n}b_j\frac{p^{sj}}{j!}
     \in(pW)^m
\]
satisfies
\refstepcounter{equation}
\[
 a'\equiv a\pmod{p^s},\qquad
 f(a')\equiv f(a)+p^sDf(a)b_1\pmod{p^{n+1}}
 \quad(f\in J).
  \tag{\theequation}\label{eq:pd-evaluation-congruences}
\]
In particular, $f(a')\in p^nW$ for every $f\in J$.
Hence the assignments
\[
 \varphi_n^u:R\longrightarrow W_n,\qquad
 x_i\longmapsto a_i'\bmod p^n
\]
define a local $W$-algebra homomorphism. Set
$
 X_n^u:=(\varphi_n^u)^*\mathfrak X.$
The congruence $a'\equiv a\pmod{p^s}$ gives
\[
 X_n^u\otimes_{W_n}W_s
 \simeq X_n\otimes_{W_n}W_s
\]
compatibly with the markings on $X$.

It remains to compute the change in the obstruction.
The first-order part of $\beta$ is
\[
 \beta|_{\mathcal A_{r,1}}=\varphi_r+t d_u,
 \qquad
 d_u\in\operatorname{Der}^{\mathrm{cont}}_W(R,W_r),
 \qquad
 d_u(x_i)=b_{1,i}\bmod p^r.
\]
Let
\[
 \mathrm{KS}:
 \operatorname{Der}^{\mathrm{cont}}_W(R,W_r)
 \longrightarrow H^1(X_r,T_r)
\]
denote the Kodaira--Spencer map of $\mathfrak X$ at $\varphi_r$,
where $W_r$ is regarded as an $R$-module via $\varphi_r$.
Explicitly, $\mathrm{KS}(d)$ is the class of the first-order
deformation
\[
 \mathfrak X\otimes_{R,\varphi_r+td}\mathcal A_{r,1},
 \qquad \mathcal A_{r,1}=W_r[t]/(t^2),
\]
with marking $\sigma_r$ on reduction modulo $t$.
Hence
\[
 \mathrm{KS}(d_u)
 =
 [\mathcal Y\otimes_{\mathcal A_{r,n}}\mathcal A_{r,1}]
 =u,
\]
where the last equality holds by the choice of $\mathcal Y$.

Let $\mathfrak m_R$ be the maximal ideal of $R$, and let
\[
 L:=\widehat L_{R/W}
 :=
 \mathbf R\varprojlim_{\nu\geq1}
 \left(
 L_{R/W}\otimes_R^{\mathbf L}R/\mathfrak m_R^\nu
 \right)
\]
be the derived $\mathfrak m_R$-adic completion of the cotangent
complex; see \cite[Tags~08PN and 0922]{Stacks}.
Since $\Lambda=W[[x_1,\ldots,x_m]]$ has completed cotangent
complex $\bigoplus_{i=1}^m\Lambda\,dx_i$ in degree zero,
the transitivity triangle and the conormal description give
\refstepcounter{equation}
\[
 \begin{gathered}
 \tau_{\geq-1}L
 \simeq
 \left[
 J/J^2\xrightarrow{\,d\,}
 \bigoplus_{i=1}^m R\,dx_i
 \right],\\
 d(f\bmod J^2)
 =
 \sum_{i=1}^m
 \left(\frac{\partial f}{\partial x_i}\bmod J\right)dx_i,
 \end{gathered}
 \tag{\theequation}\label{eq:pd-cotangent-presentation}
\]
with terms in degrees $-1$ and $0$;
see \cite[Tags~08QX and 08RA]{Stacks}.
In particular,
\[
 \operatorname{Ext}^1_R(L,k)
 \simeq
 \operatorname{coker}\left(
 \operatorname{Hom}_R\left(\bigoplus_{i=1}^m R\,dx_i,k\right)
 \xrightarrow{\,d^\vee\,}
 \operatorname{Hom}_R(J/J^2,k)
 \right),
 \qquad d^\vee(v)=v\circ d.
\]

Recall that $\varphi_n$ and $\varphi_n^u$ are induced by
evaluation at $a$ and $a'$ modulo $p^n$, respectively.
Evaluation at these coordinates gives lifts to
$\Lambda\to W_{n+1}$ of the composites
\[
 \Lambda\longrightarrow R\xrightarrow{\varphi_n}W_n,
 \qquad
 \Lambda\longrightarrow R\xrightarrow{\varphi_n^u}W_n,
\]
respectively.

For $z=a,a'$, we have $f(z)\in p^nW$ for every $f\in J$.
Thus evaluation on $J$, followed by the identification
$p^nW/p^{n+1}W\simeq k$, defines an $R$-linear map
\[
 q_z:J/J^2\longrightarrow k,
 \qquad
 f\bmod J^2\longmapsto\frac{f(z)}{p^n}\bmod p.
\]
This factors through $J/J^2$ because evaluation sends $J^2$
into $p^{2n}W\subseteq p^{n+1}W$.
Since $k$ is derived $\mathfrak m_R$-complete, derived
completion induces an isomorphism
\[
 \operatorname{Ext}^1_R(L,k)
 \simeq
 \operatorname{Ext}^1_R(L_{R/W},k);
\]
see \cite[Tag~0922]{Stacks}.

The construction in the proof of \cite[Tag~08S5]{Stacks}
identifies
\[
 [q_a],[q_{a'}]\in\operatorname{Ext}^1_R(L,k)
\]
with the obstruction classes for lifting $\varphi_n$ and
$\varphi_n^u$, respectively, to $W_{n+1}$.

Since $Df(a)b_1=c_1\in p^rW$, dividing the second congruence
in \eqref{eq:pd-evaluation-congruences} by $p^n=p^{r+s}$
gives
\[
 \begin{aligned}
 (q_{a'}-q_a)(f\bmod J^2)
 &=
 \frac{f(a')-f(a)}{p^n}\bmod p\\
 &=
 \frac{Df(a)b_1}{p^r}\bmod p.
 \end{aligned}
\]
Consequently,
\[
 q_{a'}-q_a=\lambda
 \quad\text{in }\operatorname{Hom}_R(J/J^2,k),
 \qquad
 \lambda(f\bmod J^2)
 :=
 \frac{Df(a)b_1}{p^r}\bmod p,
\]
and hence
\[
 [q_{a'}]-[q_a]=[\lambda]
 \quad\text{in }\operatorname{Ext}^1_R(L,k).
\]

Let
\[
 \kappa:\operatorname{Ext}^1_R(L,k)
 \longrightarrow H^2(X,T_X)
\]
be the obstruction map induced by $\mathfrak X$.
Compatibility of the Kodaira--Spencer morphism with
obstruction classes gives
\[
 \kappa([q_a])=o_n(X_n),\qquad
 \kappa([q_{a'}])=o_n(X_n^u);
\]
see \cite[Tags~08S3 and 07YJ]{Stacks}.

The exact sequence of $R$-modules (via $\varphi_{r+1}$)
\[
 0\longrightarrow k\xrightarrow{\,p^r\,}W_{r+1}
 \longrightarrow W_r\longrightarrow0
\]
induces a connecting homomorphism
\[
 \partial_L:
 \operatorname{Der}^{\mathrm{cont}}_W(R,W_r)
 =
 H^0\!\left(\mathbf R\operatorname{Hom}_R(L,W_r)\right)
 \longrightarrow
 \operatorname{Ext}^1_R(L,k).
\]

We first compute $\partial_L(d_u)$.
Recall that $d_u(x_i\bmod J)=b_{1,i}\bmod p^r$.
Under the presentation \eqref{eq:pd-cotangent-presentation},
the derivation $d_u$ is represented by
\[
 h:\bigoplus_iR\,dx_i\longrightarrow W_r,
 \qquad dx_i\longmapsto b_{1,i}\bmod p^r.
\]
Lift $h$ to the $R$-linear map
\[
 \widetilde h:\bigoplus_iR\,dx_i\longrightarrow W_{r+1},
 \qquad dx_i\longmapsto b_{1,i}\bmod p^{r+1}.
\]
For every $f\in J$, we have
\[
 (\widetilde h\circ d)(f\bmod J^2)
 =
 Df(a)b_1\bmod p^{r+1}.
\]
The differential in the Hom complex sends the degree-zero
map $\widetilde h$ to $-\widetilde h\circ d$.
Consequently, identifying $p^rW_{r+1}$ with $k$, we obtain
\[
 \partial_L(d_u)
 =
 \left[
 f\bmod J^2\longmapsto
 -\frac{Df(a)b_1}{p^r}\bmod p
 \right]
 =
 -[\lambda].
\]

The Kodaira--Spencer morphism of $\mathfrak X$,
pulled back along $\varphi_{r+1}$, gives a morphism
\[
 \mathbf R\operatorname{Hom}_R(L,M)
 \longrightarrow
 \mathbf R\Gamma\!\left(
 X_{r+1},
 T_{r+1}\otimes_{W_{r+1}}^{\mathbf L}M
 \right)[1]
\]
natural in the $W_{r+1}$-module $M$.
For $M=W_r$, its map on $H^0$ is $\mathrm{KS}$;
for $M=k$, its map on $H^1$ is $\kappa$.
Compatibility with the connecting homomorphisms therefore
gives the commutative diagram
\[
 \begin{array}{ccc}
 \operatorname{Der}^{\mathrm{cont}}_W(R,W_r)
 &\xrightarrow{\ \partial_L\ }&
 \operatorname{Ext}^1_R(L,k)
 \\
 {\scriptstyle\mathrm{KS}}\downarrow
 &&
 \downarrow{\scriptstyle\kappa}
 \\
 H^1(X_r,T_r)
 &\xrightarrow{\ -\delta_r\ }&
 H^2(X,T_X).
 \end{array}
\]
Here the lower arrow is $-\delta_r$ because the target
complex is shifted by $[1]$, which changes the sign
of its differential and hence of its connecting homomorphism.
Applying this diagram to $d_u$, we find
$
 \kappa(-[\lambda])
 =
 -\delta_r(\mathrm{KS}(d_u)).
$
Since $\mathrm{KS}(d_u)=u$, it follows that
$
 \kappa([\lambda])=\delta_r(u).
$

Therefore,
\[
 \begin{aligned}
 o_n(X_n^u)-o_n(X_n)
 &=\kappa([q_{a'}]-[q_a])\\
 &=\kappa([\lambda])\\
 &=\delta_r(u).
 \end{aligned}
\]
\end{proof}

\subsection{Arbitrary length and stabilization}

\begin{proof}[Proof of Proposition~\ref{tower:extension}]
Since $X$ admits the lifting $X_2^0/W_2$,
Lemma~\ref{tower:injectivity} with $r=1$ makes
$\theta_1$ injective.
Together with $a_2=0$, the argument of
Proposition~\ref{common:w3} therefore gives a lifting
$X_3/W_3$ of $X_2^0$.

Since $D$ is perfect over $W$, the module $H^1(D)$ is
finitely generated. Choose an integer $E\geq1$ such that
\[
 p^E H^1(D)_{\mathrm{tors}}=0.
\]
For $m\geq1$, let
\[
 \partial_m:H^0(D_m)\longrightarrow H^1(D)
\]
be the connecting homomorphism obtained by tensoring
\[
 0\longrightarrow W\xrightarrow{\,p^m\,}W
 \longrightarrow W_m\longrightarrow0
\]
with $D$, and write
\[
 \operatorname{red}_m:H^1(D)\longrightarrow H^1(D_m)
\]
for the natural reduction map.

Let $n\geq3$, and suppose that a lifting $X_n/W_n$
reducing to $X_2^0$ has been chosen. Put
\[
 r=\min\{n-2,E\}.
\]
For this step, write $X_s:=X_n|_{W_s}$ and
$T_s:=T_{X_s/W_s}$ for $1\leq s\leq n$.
Let $a_n\in H^0(D_n)$ be the graph class of $X_n$.

By \eqref{eqn:obstruction_period},
\eqref{eq:common-period-coordinate}, and naturality of
connecting homomorphisms, as in the proof of
Proposition~\ref{common:w3}, we have
\refstepcounter{equation}
\[
 \theta_1(o_n(X_n))
 =
 \operatorname{red}_1\bigl(\partial_n(a_n)\bigr).
 \tag{\theequation}\label{eq:tower-period-obstruction}
\]
Since $X_n|_{W_2}\simeq X_2^0$, the class $a_n$ reduces
to $a_2=0$.
Consider the following commutative
diagram with exact rows.
\[
 \begin{tikzcd}[column sep=small]
 0 \arrow[r] &
 W \arrow[r,"p^n"] \arrow[d,"p^{n-2}"'] &
 W \arrow[r] \arrow[d,equal] &
 W_n \arrow[r] \arrow[d] & 0\\
 0 \arrow[r] &
 W \arrow[r,"p^2"] &
 W \arrow[r] &
 W_2 \arrow[r] & 0.
 \end{tikzcd}
\]
After taking derived tensor products with $D$,
naturality of the connecting homomorphisms gives
\[
 p^{n-2}\partial_n(a_n)
 =
 \partial_2(a_n|_{W_2})
 =
 \partial_2(a_2)
 =
 0.
\]
Thus $\partial_n(a_n)$ is torsion and is also annihilated
by $p^E$. Since $r=\min\{n-2,E\}$, we obtain
\[
 p^r\partial_n(a_n)=0.
\]

Write $p^r$ for the map
\[
 H^2(X,T_X)\longrightarrow H^2(X_{r+1},T_{r+1})
\]
induced by the coefficient map
$k\to W_{r+1}$, $1\mapsto p^r$.
The composite $W\to k\to W_{r+1}$ sends $x$ to $p^r x$.
Consequently, \eqref{eq:tower-period-obstruction} and
\eqref{eq:relative-theta-coefficients} give
\[
 \theta_{r+1}\bigl(p^r o_n(X_n)\bigr)
 =
 \operatorname{red}_{r+1}
 \bigl(p^r\partial_n(a_n)\bigr)
 =0.
\]
Since $r+2\leq n$, the lifting $X_{r+1}$ extends to
$X_n|_{W_{r+2}}$.
Lemma~\ref{tower:injectivity} therefore makes
$\theta_{r+1}$ injective.
The cohomology sequence of
\eqref{eq:tower-tangent-coefficients} now gives
\[
 o_n(X_n)\in
 \ker\!\left(
 H^2(X,T_X)\xrightarrow{\,p^r\,}
 H^2(X_{r+1},T_{r+1})
 \right)
 =
 \operatorname{im}\delta_r.
\]

Choose $u\in H^1(X_r,T_r)$ with
$\delta_r(u)=-o_n(X_n)$.
Applying Lemma~\ref{tower:correction} gives a lifting
$\widetilde X_n/W_n$ such that
\[
 \widetilde X_n|_{W_{n-r}}\simeq X_n|_{W_{n-r}},
 \qquad
 o_n(\widetilde X_n)
 =
 o_n(X_n)+\delta_r(u)
 =
 0.
\]
Choose an extension $X_{n+1}/W_{n+1}$ of
$\widetilde X_n$, together with an identification
$X_{n+1}|_{W_n}\simeq\widetilde X_n$.
The chosen identifications give
\refstepcounter{equation}
\[
 \begin{gathered}
 X_{n+1}|_{W_{n-r}}
 \simeq
 \widetilde X_n|_{W_{n-r}}
 \simeq
 X_n|_{W_{n-r}},\\
 n-r=\max\{2,n-E\}.
 \end{gathered}
 \tag{\theequation}\label{eq:tower-preserved-reduction}
\]
In particular, $X_{n+1}|_{W_2}\simeq X_2^0$.
Starting with $X_3$, we perform this construction
successively for all $n\geq3$.

For the resulting sequence of liftings, set
\[
 Y_1=X,\qquad
 Y_m=X_{m+E}|_{W_m}\quad(m\geq2).
\]
At stage $n=m+E$, with $m\geq2$, we have
$r=E$ and $n-r=m$.
Thus \eqref{eq:tower-preserved-reduction} gives
\[
 \begin{aligned}
 Y_{m+1}|_{W_m}
 &=X_{m+E+1}|_{W_m}\\
 &\simeq\widetilde X_{m+E}|_{W_m}\\
 &\simeq X_{m+E}|_{W_m}\\
 &=Y_m.
 \end{aligned}
\]
Together with the special-fibre identifications, these
chosen isomorphisms and their composites make $(Y_m)$
a compatible tower, with $Y_2\simeq X_2^0$.

Lemma~\ref{common:algebraization} gives a smooth projective
lifting $\mathcal X/W$ inducing the compatible tower $(Y_m)$ as desired.
\end{proof}

\section{The first Hodge edge over finite Witt rings}\label{sec:edge}

Throughout this section, $X$ is a quasi-$F$-split strict
Calabi--Yau variety of dimension $d\geq3$.
For a smooth proper lift $Y/W_{n+1}$,
write $Y_n=Y\times_{W_{n+1}}W_n$. We prove that the map
\[
 H^{d-1}\bigl(Y_n,\Fil^1\Omega^\bullet_{Y_n/W_n}\bigr)
 \longrightarrow H^{d-2}(Y_n,\Omega^1_{Y_n/W_n})
\]
is surjective if $d\leq p+1$, or if $d=p+2$ and $H^0(X,T_X)=0$.
The lifting $Y/W_{n+1}$ is arbitrary; its reduction to
$W_2$ need not be the lifting constructed in
Lemma~\ref{common:canonical}.
The quasi-$F$-splitting hypothesis is imposed only on
the special fibre $X$.
Note that if $d<p$, the required surjectivity holds
even without assuming that $X$ is quasi-$F$-split.
Indeed, since $Y_n$ lifts to $W_{n+1}$, the
Hodge--de Rham spectral sequence of $Y_n/W_n$
degenerates at $E_1$
\cite[Theorem~8.2.6]{Ogus94}.

\subsection{A finite comparison}

Following \cite[Section~3.1]{TVX}, let $B$ be the $p$-adic
completion of the $p$-torsion-free quotient of the divided-power
envelope of $(v_+)$ in $A$.
Equivalently, $B$ is the $p$-adic completion of the graded
$W$-subalgebra of $A[1/p]$ generated by $v_-$ and
$v_+^m/m!$ for $m\geq1$.
Set
\[
 g(0)=0,\qquad
 g(s)=\inf_{m\geq s}\{m-v_p(m!)\}\quad(s\geq1).
\]
The homogeneous components of $B$ are
\begin{equation}
\label{eqn:Bdegree}
 B^s=W\frac{p^{g(s)}v_+^s}{p^s}\quad(s\geq0),
 \qquad B^{-s}=Wv_-^s\quad(s\geq0).
\end{equation}
In particular, $A^s\to B^s$ is an isomorphism for every $s<p$
\cite[Lemma~3.2 and Corollary~3.3]{TVX}.
Give $W_n[v_-]$ the grading with $W_n$ in degree $0$ and
$\deg(v_-)=-1$.
Let $\mathcal H_{Y_n}$ denote the Rees complex of
$(\RG_{\mathrm{dR}}(Y_n/W_n),\Fil^\bullet)$
over $W_n[v_-]$.

The following lemma is a finite-level Rees formulation of
Ogus's tame-gauge comparison
\cite[Proposition~4.5.3, Lemma~5.3.11, and
Theorem~7.3.1]{Ogus94}; see also the proof of
\cite[Theorem~8.2.5]{Ogus94}.
For smooth formal liftings over $W$, the comparison after
base change to $B$ is expressed in
\cite[Theorem~1 and Remark~4.4]{TVX}.
We record the finite-level formulation needed here,
which requires only a lifting over $W_{n+1}$.

Recall that
\[
 \mathcal N_X
 =\bigoplus_{i\in\mathbb Z}
   \mathbf R\Gamma(X,N^iW\Omega_X^\bullet),
\]
with the $i$-th summand in grading degree $i$.
Its structure as a graded
$A=W[v_-,v_+]/(v_-v_+-p)$-complex is induced by
\[
 v_-:N^iW\Omega_X^\bullet
       \hookrightarrow N^{i-1}W\Omega_X^\bullet,
 \qquad
 v_+:N^iW\Omega_X^\bullet
       \xrightarrow{\,p\,}N^{i+1}W\Omega_X^\bullet,
\]
where $\deg(v_-)=-1$ and $\deg(v_+)=1$.

\begin{lemma}\label{edge:comparison}
There is a natural isomorphism in the derived category of
graded $B/p^nB$-modules:
\[
 \bigl(\mathcal N_X\otimes_A^{\mathbf L} B\bigr)
       \otimes_W^{\mathbf L}W_n
 \simeq
 \mathcal H_{Y_n}\otimes_{W_n[v_-]}^{\mathbf L}(B/p^nB).
\]
\end{lemma}

\begin{proof}
Put $B_n=B/p^nB$.
We identify the degree-$j$ part of the right-hand side with
a two-term construction applied to the Hodge filtration, and
the degree-$j$ part of the left-hand side with a weighted sum
of Nygaard subcomplexes.
Ogus's comparison \cite{Ogus94} identifies these two complexes.
We first introduce the notation, then establish Ogus's comparison,
identify its two sides with the degree-$j$ parts of the Rees base
changes, and finally check compatibility with the $B_n$-action
and pass to derived global sections.

\medskip
\noindent
\textit{Step 1: Filtrations and the complexes used in the comparison.}

Extend $g$ to a function on $\mathbb Z$ by setting
\[
 \eta(s):=g(\max(s,0)),
\]
and we put $\epsilon_j(i):=\eta(j-i).$
We use that $g$ is nondecreasing and subadditive, with
$g(0)=0$ and $g(1)=1$.
In particular,
$0\leq\eta(s)-\eta(s-1)\leq1$.

A gauge is a nonincreasing function
$\epsilon:\mathbb Z\to\mathbb Z$.
It is $g$-tame if
\[
 0\leq\epsilon(i)-\epsilon(i+r)\leq g(r)
 \qquad(i\in\mathbb Z,\ r\geq0).
\]
See \cite[Section~4.2]{Ogus94}.
The properties of $g$ imply that every $\epsilon_j$ is $g$-tame.

For a decreasing filtration $F$ by subcomplexes of a termwise
$p$-torsion-free complex $M^\bullet$, write
\[
 F^\epsilon M^\bullet
 :=
 \sum_{i\in\mathbb Z}p^{\epsilon(i)}F^iM^\bullet
 \subset M^\bullet[1/p].
\]
This construction uses the specified subcomplexes, not merely the filtered derived object they represent.
We will apply this notation to the filtrations $C$ and $N$
defined below.

We also need a construction that does not require
$p$-torsion-freeness.
For an inverse system of complexes $F^i$ with transition maps
$\iota_i:F^i\to F^{i-1}$, define
\[
 T^\epsilon(F)
 :=
 \operatorname{Tot}\left[
 \bigoplus_{i\in\mathbb Z}F^i
 \xrightarrow{\partial_\epsilon}
 \bigoplus_{i\in\mathbb Z}F^i
 \right],
\]
where the displayed terms have auxiliary cohomological degrees
$-1$ and $0$.
For $x$ in the $i$-th summand, the nonzero components are
\refstepcounter{equation}
\[
 (\partial_\epsilon x)_{i-1}=\iota_i(x),
 \qquad
 (\partial_\epsilon x)_i
 =-p^{\epsilon(i-1)-\epsilon(i)}x.
 \tag{\theequation}\label{eq:lem41-tame-differential}
\]
For $F^i=F^iM^\bullet$, we also write $T^\epsilon(M^\bullet,F)$.
This is \cite[(4.5.0.1)]{Ogus94}.
For a filtration $F$ by subcomplexes of a termwise
$p$-torsion-free complex $M^\bullet$, define the augmentation
\[
 T^\epsilon(M^\bullet,F)\longrightarrow F^\epsilon M^\bullet
\]
by sending $x\in F^iM^q$ in the $i$-th summand of the
right-hand term to $p^{\epsilon(i)}x$, and sending the
left-hand term to zero.
The two components in \eqref{eq:lem41-tame-differential}
have images $p^{\epsilon(i-1)}x$ and
$-p^{\epsilon(i-1)}x$, so this is a morphism of complexes.

On the crystalline side, put $K^\bullet=W\Omega_X^\bullet$.
Its Nygaard filtration is
\[
 N^iK^q=
 \begin{cases}
 p^{i-q-1}VW\Omega_X^q,&q<i,\\
 W\Omega_X^q,&q\geq i.
 \end{cases}
\]
Thus our notation for the weighted Nygaard subcomplex is
\[
 N^{\epsilon_j}K^\bullet
 =
 \sum_{i\in\mathbb Z}
 p^{\epsilon_j(i)}N^iK^\bullet
 \subset K^\bullet.
\]

\medskip
\noindent
\textit{Step 2: The local PD calculation and Ogus's comparison.}

Work locally on $Y$ and choose a smooth $p$-adic formal
$W$-scheme $\widetilde Y$ with
$
 \widetilde Y\bmod p^{n+1}\simeq Y.
$
Such a lifting exists after replacing $Y$ by a sufficiently
small affine open, since $Y/W_{n+1}$ is smooth.
Differential forms on formal schemes are understood to be
$p$-adically completed. Put
\[
 E^\bullet=\Omega^\bullet_{\widetilde Y/W},
 \qquad
 J_Y=p^{n+1}\mathcal O_{\widetilde Y},
\]
with the usual de Rham differential on $E^\bullet$.
The ideal of $X$ in $\widetilde Y$ is
$p\mathcal O_{\widetilde Y}$, which already has the
canonical divided powers. Thus the completed PD envelope
of $X$ in $\widetilde Y$, relative to the canonical divided
powers on $pW$, is $\widetilde Y$ itself.
The ideal $J_Y$ is stable under these divided powers.

Set
\[
 h(0)=0,
 \qquad
 h(s)=\inf_{m\geq s}\{(n+1)m-v_p(m!)\}
 \quad(s\geq1).
\]
The divided powers of $p^{n+1}$ of degree at least $s$
generate the ideal
\[
 \left(
 \frac{p^{(n+1)m}}{m!}\;\middle|\;m\geq s
 \right)
 =p^{h(s)}W.
\]
Consequently,
\[
 J_Y^{[s]}=p^{h(s)}\mathcal O_{\widetilde Y}
 \qquad(s\geq0),
\]
where $J_Y^{[s]}$ is generated by products of divided powers
of total PD degree at least $s$, and
$J_Y^{[0]}=\mathcal O_{\widetilde Y}$.
Moreover,
\refstepcounter{equation}
\[
 h(1)=n+1,
 \qquad
 h(s)\geq g(s)+ns\geq g(s)+n
 \quad(s\geq1).
 \tag{\theequation}\label{eq:lem41-control-bound}
\]

We use the filtration
\refstepcounter{equation}
\[
 C^iE^q
 :=J_Y^{[\max(i-q,0)]}E^q
 =p^{h(\max(i-q,0))}\Omega^q_{\widetilde Y/W}.
 \tag{\theequation}\label{eq:lem41-coefficients}
\]
Since $h$ is nondecreasing, these terms form a subcomplex.
Each $E^q$ is $p$-torsion-free, and $C^iE^q=E^q$ for $i\leq q$.
The explicit formula also gives
\refstepcounter{equation}
\[
 p^rE^q\cap C^iE^q
 =
 p^rC^iE^q+
 \sum_{s\geq1}p^{h(s)}C^{i-s}E^q
 \qquad(0\leq r\leq n+1).
 \tag{\theequation}\label{eq:lem41-transversality}
\]
To check this, put $k=i-q$.
The left-hand side is
$p^{\max(r,h(\max(k,0)))}E^q$.
If $k\leq0$, this is $p^rE^q$, which is already the first
term on the right.
All other terms lie in it because $h(s)\geq n+1\geq r$.
If $k>0$, the left-hand side is $p^{h(k)}E^q$.
The term $s=k$ on the right is exactly this submodule,
and every other term lies in it because
\[
 h(s)+h(\max(k-s,0))\geq h(k).
\]
Here we use the monotonicity and subadditivity of $h$.


Under the natural identification
$E^\bullet/p^nE^\bullet=\Omega^\bullet_{Y_n/W_n}$,
use the image filtration
\[
 \overline C^i
 =
 \operatorname{im}
 \bigl(C^iE^\bullet\longrightarrow E^\bullet/p^nE^\bullet\bigr).
\]
By \eqref{eq:lem41-coefficients} and
\eqref{eq:lem41-control-bound}, its degree-$q$ term is zero
for $q<i$ and is
$\Omega^q_{Y_n/W_n}$ for $q\geq i$.
Hence
\refstepcounter{equation}
\[
 \overline C^i(E^\bullet/p^nE^\bullet)
 =
 \Omega^{\geq i}_{Y_n/W_n}
 \qquad(i\in\mathbb Z).
 \tag{\theequation}\label{eq:lem41-filtered-poincare}
\]

Now \eqref{eq:lem41-control-bound} and
\eqref{eq:lem41-transversality} allow us to apply
\cite[Proposition~4.5.3]{Ogus94}, termwise and then totalize:
\[
 T^{\epsilon_j}(E^\bullet/p^nE^\bullet,\overline C)
 \simeq
 C^{\epsilon_j}E^\bullet/p^nC^{\epsilon_j}E^\bullet.
\]
Combining this with \eqref{eq:lem41-filtered-poincare} gives
\refstepcounter{equation}
\[
 T^{\epsilon_j}
 \bigl(\Omega^\bullet_{Y_n/W_n},\operatorname{Fil}_{\mathrm H}\bigr)
 \simeq
 C^{\epsilon_j}E^\bullet\otimes_W^{\mathbf L}W_n,
 \qquad
 \operatorname{Fil}_{\mathrm H}^i=\Omega^{\geq i}_{Y_n/W_n}.
 \tag{\theequation}\label{eq:lem41-finite-hodge-comparison}
\]
Here $C^{\epsilon_j}E^\bullet$ is termwise $p$-torsion-free,
so its derived reduction is its ordinary quotient by $p^n$.
Note that this quotient is not the image of $C^{\epsilon_j}E^\bullet$
in $E^\bullet/p^nE^\bullet$.

To apply the crystalline comparison, put
$J_X=p\mathcal O_{\widetilde Y}$.
The divided powers on $J_X$ induce the filtration
\[
 \widetilde C^iE^\bullet
 :=
 \sum_{s\geq0}p^{g(s)}C^{i-s}E^\bullet,
 \qquad
 \widetilde C^iE^q
 =
 J_X^{[\max(i-q,0)]}E^q
 =p^{g(\max(i-q,0))}E^q.
\]
For a $g$-tame gauge $\epsilon$,
$\epsilon(i)+g(s)\geq\epsilon(i-s)$.
Consequently,
\refstepcounter{equation}
\[
 \widetilde C^\epsilon E^\bullet
 =
 \sum_{\substack{i\in\mathbb Z\\s\geq0}}
 p^{\epsilon(i)+g(s)}C^{i-s}E^\bullet
 =
 C^\epsilon E^\bullet.
 \tag{\theequation}\label{eq:lem41-saturation}
\]
The inclusion
$\widetilde C^\epsilon E^\bullet\subseteq C^\epsilon E^\bullet$
follows termwise from
\[
 p^{\epsilon(i)+g(s)}C^{i-s}E^\bullet
 \subseteq
 p^{\epsilon(i-s)}C^{i-s}E^\bullet
 \qquad(s\geq0),
\]
by $g$-tameness.
The terms with $s=0$ give the reverse inclusion,
since $g(0)=0$.
This is \cite[Lemma~4.3.7]{Ogus94}.

We spell out the change from Ogus's convention to the
preimage convention for the Nygaard filtration.
Let $\sigma$ be Frobenius on $W$, and let
\[
 \varphi:\sigma^*K^\bullet\longrightarrow K^\bullet
\]
be the linearized crystalline Frobenius, given in degree $q$
by $p^qF$.
Since $X$ is smooth, each $K^q$ is $p$-torsion-free and
$F$ is injective
\cite[Remark~2.2.3 and Theorem~4.4.12]{BLM}.
Hence $\varphi$ is termwise injective.
Using $FV=p$ and the termwise description of the Nygaard
filtration, we obtain
\refstepcounter{equation}
\[
 \sigma^*N^iK^\bullet
 =
 (\sigma^*K^\bullet)\cap
 \varphi^{-1}(p^iK^\bullet)
 \subset(\sigma^*K^\bullet)[1/p].
 \tag{\theequation}\label{eq:lem41-nygaard-preimage}
\]

For $r\in\mathbb Z$, define
\[
 I^rK^\bullet
 :=
 K^\bullet\cap p^r\varphi(\sigma^*K^\bullet)
 \subset K^\bullet[1/p].
\]

We use $K^\bullet$ to compute the crystalline complex
via \cite[Theorem~10.1.1 and its proof]{BLM},
with filtered compatibility as in
\cite[Section~7.1]{Ogus94}.
For the structure crystal $\mathcal O_{X/W}$,
\cite[Theorems~5.2.13 and 7.3.1]{Ogus94}
then gives, for every $g$-tame gauge $\epsilon$,
\[
 \sigma^*(\widetilde C^\epsilon E^\bullet)
 \simeq
 \sum_{r\in\mathbb Z}
 p^{\epsilon(-r)-r}I^rK^\bullet.
\]
Here the underlying Zariski spaces of $X$ and $X^{(p)}$
are identified via relative Frobenius.

On the other hand,
\eqref{eq:lem41-nygaard-preimage} gives
\[
 \begin{aligned}
 \varphi(\sigma^*N^iK^\bullet)
 &=
 p^iK^\bullet\cap\varphi(\sigma^*K^\bullet)\\
 &=p^iI^{-i}K^\bullet.
 \end{aligned}
\]
Consequently,
\[
 \begin{aligned}
 \varphi\bigl(\sigma^*(N^\epsilon K^\bullet)\bigr)
 &=
 \sum_{i\in\mathbb Z}
 p^{\epsilon(i)+i}I^{-i}K^\bullet\\
 &=
 \sum_{r\in\mathbb Z}
 p^{\epsilon(-r)-r}I^rK^\bullet.
 \end{aligned}
\]
This is \cite[(5.1.1.3)]{Ogus94} in our notation.
Since $\varphi$ is injective, we obtain
\[
 \sigma^*(\widetilde C^\epsilon E^\bullet)
 \simeq
 \sigma^*(N^\epsilon K^\bullet).
\]
Using \eqref{eq:lem41-saturation} and the fact that
$\sigma$ is an automorphism of $W$, we conclude that
\refstepcounter{equation}
\[
 C^{\epsilon_j}E^\bullet
 \simeq N^{\epsilon_j}K^\bullet.
 \tag{\theequation}\label{eq:lem41-ogus-nygaard}
\]

Finally, combining
\eqref{eq:lem41-finite-hodge-comparison} and
\eqref{eq:lem41-ogus-nygaard}, we obtain
\refstepcounter{equation}
\[
 T^{\epsilon_j}
 \bigl(\Omega^\bullet_{Y_n/W_n},\operatorname{Fil}_{\mathrm H}\bigr)
 \simeq
 N^{\epsilon_j}K^\bullet\otimes_W^{\mathbf L}W_n.
 \tag{\theequation}\label{eq:lem41-finite-comparison}
\]

\medskip
\noindent
\textit{Step 3: Identification with the two Rees base changes.}

Recall that $A=W[v_-,v_+]/(v_-v_+-p)$.
It is convenient to make the algebraic calculations before
$p$-adic completion.
Inside $A[v_-^{-1}]=W[v_-,v_-^{-1}]$, put
\[
 B_{\mathrm{alg}}
 :=
 \bigoplus_{s\in\mathbb Z}Wb_s,
 \qquad
 b_s:=p^{\eta(s)}v_-^{-s},
 \qquad
 \delta_s=\eta(s)-\eta(s-1).
\]
Thus $B_{\mathrm{alg}}$ is the graded algebra whose
$p$-adic completion is $B$, and
\refstepcounter{equation}
\[
 v_-b_s=p^{\delta_s}b_{s-1},
 \qquad
 b_sb_a
 =
 p^{\eta(s)+\eta(a)-\eta(s+a)}b_{s+a}.
 \tag{\theequation}\label{eq:lem41-b-relations}
\]
Both $B_{\mathrm{alg}}$ and $B$ are $W$-flat, and
\refstepcounter{equation}
\[
 B_{\mathrm{alg}}/p^nB_{\mathrm{alg}}
 \simeq B/p^nB=B_n.
 \tag{\theequation}\label{eq:lem41-finite-algebra}
\]
It therefore suffices to use $B_{\mathrm{alg}}$ in the
integral calculation and reduce modulo $p^n$ at the end.

\smallskip
\noindent
\textit{The Hodge side.}
Since
\[
 B_n=\bigoplus_{s\in\mathbb Z}W_n\overline b_s,
\]
the standard exact sequence in the proof of
\cite[Theorem~4.3.7]{Weibel94}, applied to $B_n$ over
$W_n[v_-]$, gives a graded free resolution
\refstepcounter{equation}
\[
 0\longrightarrow\bigoplus_{s\in\mathbb Z}W_n[v_-]f_s
 \xrightarrow{d}
 \bigoplus_{s\in\mathbb Z}W_n[v_-]e_s
 \xrightarrow{\rho}B_n
 \longrightarrow0,
 \tag{\theequation}\label{eq:lem41-hodge-resolution}
\]
where
\[
 \deg(f_s)=s-1,\qquad \deg(e_s)=s,
\]
and, by \eqref{eq:lem41-b-relations},
\[
 d(f_s)=v_-e_s-p^{\delta_s}e_{s-1},
 \qquad
 \rho(e_s)=\overline b_s.
\]
As ungraded modules, both free terms are identified with
$W_n[v_-]\otimes_{W_n}B_n$, with $f_s$ and $e_s$
corresponding to $1\otimes\overline b_s$ in the respective
copies. Both maps have degree zero.

Give both free modules the right $B_n$-action induced by
\[
 (a\otimes b)\cdot c=a\otimes bc.
\]
In terms of the chosen bases, this is
\refstepcounter{equation}
\[
 \begin{aligned}
 e_r\cdot\overline b_s
 &=p^{\eta(r)+\eta(s)-\eta(r+s)}e_{r+s},\\
 f_r\cdot\overline b_s
 &=p^{\eta(r)+\eta(s)-\eta(r+s)}f_{r+s}.
 \end{aligned}
 \tag{\theequation}\label{eq:lem41-resolution-b-action}
\]
The tensor descriptions
\[
 d(a\otimes b)=v_-a\otimes b-a\otimes v_-b,
 \qquad
 \rho(a\otimes b)=ab
\]
give
\[
 d(\xi c)=d(\xi)c,\qquad
 \rho(\zeta c)=\rho(\zeta)c
 \qquad(c\in B_n).
\]
Thus \eqref{eq:lem41-hodge-resolution} is an exact sequence
of graded $W_n[v_-]$--$B_n$ bimodules.

Let
$\mathscr H=\bigoplus_i\Omega^{\geq i}_{Y_n/W_n}$
be the sheaf-level Hodge Rees complex.
Tensoring this resolution with $\mathscr H$ over $W_n[v_{-}]$ gives
\[
 \left(\mathscr H\otimes_{W_n[v_-]}^{\mathbf L}B_n\right)^j
 \simeq
 \operatorname{Tot}
 \left[
 \bigoplus_{s\in\mathbb Z}
 \Omega^{\geq j-s+1}_{Y_n/W_n}
 \longrightarrow
 \bigoplus_{s\in\mathbb Z}
 \Omega^{\geq j-s}_{Y_n/W_n}
 \right].
\]
Set $i=j-s+1$ in the source.
The two components of its differential are the inclusion
into filtration level $i-1$ and multiplication by
$-p^{\delta_{j-i+1}}$ into filtration level $i$.
Since
\[
 \delta_{j-i+1}
 =
 \epsilon_j(i-1)-\epsilon_j(i),
\]
this is exactly the differential in
\eqref{eq:lem41-tame-differential} defining $T^{\epsilon_j}$.
Hence
\refstepcounter{equation}
\[
 \left(\mathscr H\otimes_{W_n[v_-]}^{\mathbf L}B_n\right)^j
 \simeq
 T^{\epsilon_j}
 \bigl(\Omega^\bullet_{Y_n/W_n},\operatorname{Fil}_{\mathrm H}\bigr).
 \tag{\theequation}\label{eq:lem41-hodge-rees}
\]
The bimodule resolution retains the natural $B_n$-action.

\smallskip
\noindent
\textit{The Nygaard side.}
Let $\mathscr N=\bigoplus_iN^iK^\bullet$ ($K^{\bullet} = W\Omega_X^{\bullet}$) be the Nygaard
Rees complex, with $v_-$ acting by inclusion and $v_+$ by $p$.
We first check that $\mathscr N$ is termwise $A$-flat.
At a stalk, put $M=K_x^q$ and $L=(VK^q)_x$.
Then $M$ is $p$-torsion-free and $pM\subseteq L\subseteq M$.
For each finitely generated $W$-submodule $M_\lambda\subseteq M$,
put $L_\lambda=M_\lambda\cap L$ and
\[
 F_\lambda^i=
 \begin{cases}
 M_\lambda,&i\leq q,\\
 p^{i-q-1}L_\lambda,&i>q.
 \end{cases}
\]
The module $M_\lambda$ is $W$-free, and
$pM_\lambda\subseteq L_\lambda\subseteq M_\lambda$.
Choose a basis adapted to $L_\lambda$:
\[
 M_\lambda=\bigoplus_a We_a\oplus\bigoplus_b Wf_b,
 \qquad
 L_\lambda=\bigoplus_a We_a\oplus\bigoplus_b pWf_b.
\]
The Rees module $\bigoplus_iF_\lambda^i$ is therefore
free over $A$, with generators $e_a$ in degree $q+1$
and $f_b$ in degree $q$.
Since
\[
 \operatorname*{colim}_\lambda F_\lambda^i=N^iK_x^q,
\]
$\mathscr N_x^q$ is $A$-flat.
The complex $\mathscr N$ is bounded, so its ordinary
tensor product computes derived base change.

Consequently, $B_{\mathrm{alg}}\subset A[v_-^{-1}]$ gives
a termwise injection
\[
 \mathscr N\otimes_A B_{\mathrm{alg}}
 \longrightarrow\mathscr N[v_-^{-1}].
\]
The degree-zero part of the target is $K^\bullet$.
Taking degree $j$ and multiplying by $v_-^j$ gives
\[
 x\otimes b_{j-i}\longmapsto p^{\epsilon_j(i)}x,
 \qquad x\in N^iK^\bullet.
\]
Its image is exactly $N^{\epsilon_j}K^\bullet$, and
the map commutes with the differentials. Thus
\refstepcounter{equation}
\[
 \left(
 \mathscr N\otimes_A^{\mathbf L}B_{\mathrm{alg}}
 \right)^j
 \simeq N^{\epsilon_j}K^\bullet.
 \tag{\theequation}\label{eq:lem41-nygaard-rees}
\]

\medskip
\noindent
\textit{Step 4: Compatibility with the $B_n$-action and globalization.}
We first prove that \eqref{eq:lem41-finite-comparison} is
$B_n$-linear, and then prove that it is independent of the
local formal lift.

Let $\alpha_j$ and $\gamma_j$ denote the comparisons
\eqref{eq:lem41-finite-hodge-comparison} and
\eqref{eq:lem41-ogus-nygaard}, respectively:
\[
 \begin{aligned}
 \alpha_j:
 T^{\epsilon_j}(\Omega^\bullet_{Y_n/W_n},\operatorname{Fil}_{\mathrm H})
 &\xrightarrow{\sim}C^{\epsilon_j}E^\bullet/p^nC^{\epsilon_j}E^\bullet,\\
 \gamma_j:C^{\epsilon_j}E^\bullet
 &\xrightarrow{\sim}N^{\epsilon_j}K^\bullet.
 \end{aligned}
\]
Thus the comparison in \eqref{eq:lem41-finite-comparison} is
\refstepcounter{equation}
\[
 \beta_j=(\gamma_j\otimes_W^{\mathbf L}W_n)\circ\alpha_j.
 \tag{\theequation}\label{eq:lem41-local-comparison}
\]
For $i+r=j$, write an element of the right-hand term of
the two-term model in \eqref{eq:lem41-hodge-rees} as
$\overline x\otimes e_r$, where
$\overline x\in(\Omega^{\geq i}_{Y_n/W_n})^q$, and choose a lift
$x\in C^iE^q$. The map $\alpha_j$ is given by
\refstepcounter{equation}
\[
 \alpha_j(\overline x\otimes e_r)=[p^{\eta(r)}x],
 \qquad
 \alpha_j\big|_{\text{left-hand term}}=0.
 \tag{\theequation}\label{eq:lem41-finite-augmentation}
\]
This is the augmentation inducing
\eqref{eq:lem41-finite-hodge-comparison}, as in
\cite[(4.5.0.2) and Proposition~4.5.3]{Ogus94}.
Its independence of the
lift follows from \eqref{eq:lem41-transversality},
\eqref{eq:lem41-control-bound}, and tameness, which give
\[
 p^{\epsilon_j(i)}(p^nE^q\cap C^iE^q)
 \subseteq p^nC^{\epsilon_j}E^q.
\]
By \eqref{eq:lem41-resolution-b-action} and
\eqref{eq:lem41-finite-augmentation}, we obtain
\refstepcounter{equation}
\[
 \begin{aligned}
 \alpha_{j+s}\bigl((\overline x\otimes e_r)\overline b_s\bigr)
 &=[p^{\eta(r)+\eta(s)-\eta(r+s)}p^{\eta(r+s)}x]\\
 &=p^{\eta(s)}\alpha_j(\overline x\otimes e_r).
 \end{aligned}
 \tag{\theequation}\label{eq:lem41-alpha-action}
\]
Equation \eqref{eq:lem41-alpha-action} also holds on the
left-hand term, since both maps vanish there by
\eqref{eq:lem41-finite-augmentation}.

On the Nygaard side, \eqref{eq:lem41-nygaard-rees} identifies
the action of $b_s$ with multiplication by $p^{\eta(s)}$.
Indeed, for a degree-$j$ element $z$ of
$\mathscr N\otimes_A B_{\mathrm{alg}}$,
\[
 v_-^{j+s}(b_sz)=p^{\eta(s)}v_-^jz.
\]
This multiplication also maps $C^{\epsilon_j}E^\bullet$ to
$C^{\epsilon_{j+s}}E^\bullet$, since
\[
 \epsilon_{j+s}(i)\leq\epsilon_j(i)+\eta(s).
\]
The comparison of tame-gauge filtered objects in
\cite[Theorem~7.3.1]{Ogus94}, used to construct $\gamma_j$,
respects gauge inclusions and the translations
$F^{\epsilon+a}=p^aF^\epsilon$. Applying it to
\[
 C^{\epsilon_j}E^\bullet
 \xrightarrow{\,p^{\eta(s)}\,}
 C^{\epsilon_j+\eta(s)}E^\bullet
 \hookrightarrow C^{\epsilon_{j+s}}E^\bullet
\]
and to
\[
 N^{\epsilon_j}K^\bullet
 \xrightarrow{\,p^{\eta(s)}\,}
 N^{\epsilon_j+\eta(s)}K^\bullet
 \hookrightarrow N^{\epsilon_{j+s}}K^\bullet
\]
gives
\refstepcounter{equation}
\[
 \gamma_{j+s}\circ p^{\eta(s)}
 =p^{\eta(s)}\circ\gamma_j.
 \tag{\theequation}\label{eq:lem41-gamma-action}
\]
Applying $-\otimes_W^{\mathbf L}W_n$ to
\eqref{eq:lem41-gamma-action} and using
\eqref{eq:lem41-local-comparison} and
\eqref{eq:lem41-alpha-action}, we obtain
\refstepcounter{equation}
\[
 \beta_{j+s}\circ(\cdot\overline b_s)
 =(\cdot p^{\eta(s)})\circ\beta_j.
 \tag{\theequation}\label{eq:lem41-action-compatibility}
\]
Thus $\bigoplus_j\beta_j$ is an equivalence
in the derived category of graded $B_n$-modules.

We now globalize \eqref{eq:lem41-local-comparison}.
The argument is the same as the simplicial PD de Rham argument
in \cite[Remark~6.1.8 and the proof of Theorem~7.3.1]{Ogus94};
we describe only the comparison maps and the results used.
Choose a finite open covering $Y=\bigcup_a U_a$ with smooth
formal liftings $\widetilde U_a/W$. For $I=(a_0,\ldots,a_r)$, put
\[
 \begin{gathered}
 U_I=\bigcap_{\nu=0}^r U_{a_\nu},\qquad
 U_{I,n}=U_I\times_{W_{n+1}}W_n,\\
 Z_I=\widetilde U_{a_0}\times_W\cdots\times_W\widetilde U_{a_r}.
 \end{gathered}
\]
To compare the local maps on $U_I$, use the induced embedding
$U_I\hookrightarrow Z_I$.
Let $D_{Z_I}$ be the completed PD envelope of
$U_I\times_{W_{n+1}}k\hookrightarrow Z_I$ relative to the
canonical divided powers on $pW$, and set
\[
 \begin{aligned}
 E_{Z_I}^\bullet
 &=\mathcal O_{D_{Z_I}}
   \widehat\otimes_{\mathcal O_{Z_I}}\Omega^\bullet_{Z_I/W},\\
 J_{Z_I}&=\ker(\mathcal O_{D_{Z_I}}\to\mathcal O_{U_I}),\\
 C_{Z_I}^iE_{Z_I}^q
 &=J_{Z_I}^{[\max(i-q,0)]}E_{Z_I}^q,\\
 \overline C_{Z_I}^i
 &=\operatorname{im}(C_{Z_I}^iE_{Z_I}^\bullet
             \to E_{Z_I}^\bullet/p^nE_{Z_I}^\bullet).
 \end{aligned}
\]
These are regarded as complexes of sheaves on $|U_I|$.

Apply the construction of Step~2 to $(E_{Z_I}^\bullet,C_{Z_I})$.
Here \cite[Proposition~3.3.3]{Ogus94} supplies the
transversality used in \eqref{eq:lem41-transversality};
together with \eqref{eq:lem41-control-bound} and
\cite[Proposition~4.5.3]{Ogus94}, it gives the first arrow below.
The second uses \eqref{eq:lem41-saturation} and the crystalline
comparison used in \eqref{eq:lem41-ogus-nygaard}:
\refstepcounter{equation}
\[
 \begin{aligned}
 \beta_{Z_I,j}:\quad
 &T^{\epsilon_j}(E_{Z_I}^\bullet/p^nE_{Z_I}^\bullet,
                         \overline C_{Z_I})\\
 &\xrightarrow{\sim}
 C_{Z_I}^{\epsilon_j}E_{Z_I}^\bullet
                         \otimes_W^{\mathbf L}W_n\\
 &\xrightarrow{\sim}
 \left.(N^{\epsilon_j}K^\bullet
                         \otimes_W^{\mathbf L}W_n)\right|_{U_I}.
 \end{aligned}
 \tag{\theequation}\label{eq:lem41-pd-beta}
\]
On the Hodge side, \cite[Corollary~6.1.7]{Ogus94}, applied
with the PD filtration and the $p$-adic filtration, together
with \eqref{eq:lem41-filtered-poincare}, shows that restriction
$\rho_{Z_I}:E_{Z_I}^\bullet/p^nE_{Z_I}^\bullet
\to\Omega^\bullet_{U_{I,n}/W_n}$ induces
\refstepcounter{equation}
\[
 \begin{aligned}
 T^{\epsilon_j}(\rho_{Z_I}):\quad
 &T^{\epsilon_j}(E_{Z_I}^\bullet/p^nE_{Z_I}^\bullet,
                         \overline C_{Z_I})\\
 &\xrightarrow{\sim}
 T^{\epsilon_j}(\Omega^\bullet_{U_{I,n}/W_n},
                         \operatorname{Fil}_{\mathrm H}).
 \end{aligned}
 \tag{\theequation}\label{eq:lem41-pd-rho}
\]

Use Ogus's functorial comparison diagrams for all these
products. For a projection or diagonal $f:Z_{I'}\to Z_I$
occurring in the \v{C}ech construction, pullback gives
\[
 \overline f^*:
 \left.(E_{Z_I}^\bullet/p^nE_{Z_I}^\bullet)\right|_{U_{I'}}
 \longrightarrow E_{Z_{I'}}^\bullet/p^nE_{Z_{I'}}^\bullet.
\]
Naturality of restriction and of
\cite[Proposition~4.5.3 and Theorem~7.3.1]{Ogus94} gives
\refstepcounter{equation}
\[
 \begin{aligned}
 T^{\epsilon_j}(\rho_{Z_{I'}})
       \circ T^{\epsilon_j}(\overline f^*)
 &=\left.T^{\epsilon_j}(\rho_{Z_I})\right|_{U_{I'}},\\
 \beta_{Z_{I'},j}\circ T^{\epsilon_j}(\overline f^*)
 &=\left.\beta_{Z_I,j}\right|_{U_{I'}}.
 \end{aligned}
 \tag{\theequation}\label{eq:lem41-cech-compatibility}
\]
Thus \eqref{eq:lem41-pd-rho} and \eqref{eq:lem41-pd-beta}
give morphisms of \v{C}ech diagrams, as in Ogus's construction.
Their compatibility with gauge inclusions and multiplication
by powers of $p$ also gives
\eqref{eq:lem41-alpha-action} and \eqref{eq:lem41-gamma-action}
on these diagrams.

Let $\pi:U_\bullet\to X$ be the augmented \v{C}ech nerve
of the open covering on $|Y|=|X|$, with
$U_r=\coprod_{I=(a_0,\ldots,a_r)}U_I$.
Cohomological descent
\cite[Tag~0DA5]{Stacks} gives
\[
 \mathscr F\xrightarrow{\sim}\mathbf R\pi_*\pi^*\mathscr F
 \qquad(\mathbf L\pi^*=\pi^*).
\]
Apply this to
\[
 \mathscr F=
 T^{\epsilon_j}(\Omega^\bullet_{Y_n/W_n},
                         \operatorname{Fil}_{\mathrm H})
 \quad\text{and}\quad
 \mathscr F=N^{\epsilon_j}K^\bullet\otimes_W^{\mathbf L}W_n.
\]
Applying $\mathbf R\pi_*$ to the functorial diagrams in
\eqref{eq:lem41-pd-rho} and \eqref{eq:lem41-pd-beta} then gives
\[
 \begin{aligned}
 \beta_j^{\mathrm{glob}}
 &:=(\mathbf R\pi_*\beta_{Z_\bullet,j})
       \circ(\mathbf R\pi_*T^{\epsilon_j}(\rho_{Z_\bullet}))^{-1}:\\
 &\quad T^{\epsilon_j}
   (\Omega^\bullet_{Y_n/W_n},\operatorname{Fil}_{\mathrm H})
 \xrightarrow{\sim}
 N^{\epsilon_j}K^\bullet\otimes_W^{\mathbf L}W_n.
 \end{aligned}
\]
By construction,
\refstepcounter{equation}
\[
 \beta_j^{\mathrm{glob}}\big|_{U_I}
 =\beta_{Z_I,j}\circ T^{\epsilon_j}(\rho_{Z_I})^{-1}.
 \tag{\theequation}\label{eq:lem41-global-restriction}
\]

By \eqref{eq:lem41-action-compatibility} and
\eqref{eq:lem41-global-restriction}, these comparisons satisfy
\[
 \beta_{j+s}^{\mathrm{glob}}\circ(\cdot\overline b_s)
 =
 (\cdot p^{\eta(s)})\circ\beta_j^{\mathrm{glob}}
 \qquad(j,s\in\mathbb Z).
\]
Combining all $j$ and using
\eqref{eq:lem41-hodge-rees},
\eqref{eq:lem41-nygaard-rees}, and
\eqref{eq:lem41-finite-algebra}, we obtain an equivalence
of sheaves of graded $B_n$-complexes
\refstepcounter{equation}
\[
 \mathscr N\otimes_A^{\mathbf L}B_n
 \simeq
 \mathscr H\otimes_{W_n[v_-]}^{\mathbf L}B_n.
 \tag{\theequation}\label{eq:lem41-global-sheaf-comparison}
\]

On the common noetherian finite-dimensional Zariski space
of $X$ and $Y_n$, derived global sections commute with
direct sums and have finite cohomological dimension.
A graded free resolution of $B_n$ therefore gives
\refstepcounter{equation}
\[
 \mathbf R\Gamma
 \bigl(X,\mathscr M\otimes_R^{\mathbf L}B_n\bigr)
 \simeq
 \mathbf R\Gamma(X,\mathscr M)
       \otimes_R^{\mathbf L}B_n
 \tag{\theequation}\label{eq:lem41-global-projection-formula}
\]
for
\[
 (\mathscr M,R)
 =
 (\mathscr N,A),
 \qquad
 (\mathscr H,W_n[v_-]).
\]
Applying derived global sections to
\eqref{eq:lem41-global-sheaf-comparison}
and using \eqref{eq:lem41-global-projection-formula}, we conclude
\[
 \begin{aligned}
 (\mathcal N_X\otimes_A^{\mathbf L}B)
       \otimes_W^{\mathbf L}W_n
 &\simeq
 \mathcal N_X\otimes_A^{\mathbf L}B_n\\
 &\simeq
 \mathbf R\Gamma
 \bigl(X,\mathscr N\otimes_A^{\mathbf L}B_n\bigr)\\
 &\simeq
 \mathbf R\Gamma
 \bigl(Y_n,
 \mathscr H\otimes_{W_n[v_-]}^{\mathbf L}B_n\bigr)\\
 &\simeq
 \mathcal H_{Y_n}\otimes_{W_n[v_-]}^{\mathbf L}B_n,
 \end{aligned}
\]
as required.
\end{proof}

\subsection{The finite-height summand on a lifting}
We transport the summand of Lemma~\ref{edge:lower-split} to the
Hodge filtration of a finite Witt lifting.

Put $B_n=B/p^nB$.
For a graded $W_n[v_-]$-complex $M$, write
\[
 \Fil^iM=M^{(i)},
 \qquad
 \operatorname{gr}^iM
 =\bigl(M\otimes_{W_n[v_-]}^{\mathbf L}W_n\bigr)^{(i)},
 \qquad W_n=W_n[v_-]/(v_-).
\]

Recall that $\mathcal U_h$ is the graded $A$-complex with
degree-$i$ part
\[
 \mathcal U_h^{(i)}
 =
 \begin{cases}
 \mathbf R\Gamma(X,W\mathcal O_X),&i\leq0,\\
 \mathbf R\Gamma(X,p^{i-1}VW\mathcal O_X),&i\geq1,
 \end{cases}
\]
where $v_-$ acts by inclusion and $v_+$ by multiplication by $p$.
Projection $W\Omega_X^\bullet\to W\mathcal O_X[0]$ induces
$q:\mathcal N_X\to\mathcal U_h$.
By Lemma~\ref{edge:lower-split}, there is a section
\[
 s:\mathcal U_h\longrightarrow\mathcal N_X,
 \qquad q\circ s=\operatorname{id}_{\mathcal U_h}.
\]

\begin{lemma}\label{edge:descent}
Let
\[
 c_n:
 \mathcal N_X\otimes_A^{\mathbf L}B_n
 \xrightarrow{\sim}
 \mathcal H_{Y_n}\otimes_{W_n[v_-]}^{\mathbf L}B_n
\]
be the comparison of Lemma~\ref{edge:comparison}.

There exist a bounded finite free graded $W_n[v_-]$-complex
$\mathcal U_{H,n}$ and a morphism
\[
 s_H:\mathcal U_{H,n}\longrightarrow\mathcal H_{Y_n}
\]
satisfying the following:
\begin{enumerate}
\item There is an isomorphism
\[
 \beta_n:
 \mathcal U_{H,n}\otimes_{W_n[v_-]}^{\mathbf L}B_n
 \xrightarrow{\sim}
 \mathcal U_h\otimes_A^{\mathbf L}B_n
\]
such that
\[
 s_H\otimes_{W_n[v_-]}^{\mathbf L}B_n
 =
 c_n\circ
 \bigl(s\otimes_A^{\mathbf L}B_n\bigr)\circ\beta_n
\]
in the derived category of graded $B_n$-modules.

\item The induced morphism
\[
 \operatorname{gr}^0(s_H):
 \operatorname{gr}^0\mathcal U_{H,n}
 \longrightarrow
 \operatorname{gr}^0\mathcal H_{Y_n}
 \simeq
 \mathbf R\Gamma(Y_n,\mathcal O_{Y_n})
\]
is an isomorphism in the derived category $D(W_n)$
of $W_n$-modules.

\item We have
\[
 \Fil^2\mathcal U_{H,n}=0,
 \qquad
 \Fil^1\mathcal U_{H,n}
 \simeq W_n^{\oplus(h-1)}[-d].
\]
\end{enumerate}
\end{lemma}

\begin{proof}
Put $A_n=A/p^nA$.
Reducing the model \eqref{eq:lower-split-free-model} modulo
$p^n$, we obtain
\[
 \mathcal U_h\otimes_A A_n
 =
 A_n\mathbf 1
 \oplus
 \left(A_nz\oplus\bigoplus_{a=1}^{h-1}A_ne_a\right)[-d],
 \qquad d_{\mathcal U_h}=0,
\]
with grading degrees
\[
 \deg(\mathbf 1)=\deg(z)=0,
 \qquad
 \deg(e_a)=1.
\]

As in part~(3) of the proof of Lemma~\ref{edge:lower-split},
choose a bounded minimal graded free model
$M^\bullet\simeq\mathcal H_{Y_n}$ over $W_n[v_-]$.
The component of grading degree $i$ of
$M^\bullet/v_-M^\bullet$ computes
$\mathbf R\Gamma(Y_n,\Omega^i_{Y_n/W_n})[-i]$.
Derived reduction modulo $p$ therefore gives
\[
 \left(M^\bullet\otimes_{W_n[v_-]}k\right)^{(i)}
 \simeq
 \mathbf R\Gamma(X,\Omega_X^i)[-i]
 \qquad(0\leq i\leq d),
\]
and zero for all other $i$.
By minimality, $M^\bullet\otimes_{W_n[v_-]}k$ has zero
differential. The preceding comparison therefore gives
\[
 \begin{gathered}
 \left(M^q\otimes_{W_n[v_-]}k\right)^{(i)}=0
 \qquad(i<0),\\
 \left(M^q\otimes_{W_n[v_-]}k\right)^{(0)}
 \simeq H^q(X,\mathcal O_X)
 \simeq
 \begin{cases}
 k,&q=0,d,\\
 0,&q\notin\{0,d\}.
 \end{cases}
 \end{gathered}
\]
Thus the only homogeneous generators of grading degree zero are
\[
 \mathbf 1_H\in M^0,\qquad z_H\in M^d,
 \qquad
 \deg(\mathbf 1_H)=\deg(z_H)=0,
\]
and all other homogeneous generators have grading degree
at least one.

Choose a homogeneous $W_n[v_-]$-basis
$\{z_H\}\cup\{y_\lambda\}_{\lambda\in\Lambda}$ of $M^d$,
where $\Lambda$ is a finite index set.
Put $j_\lambda:=\deg(y_\lambda)\geq1$. Thus
\[
 M^d
 =
 W_n[v_-]z_H
 \oplus
 \bigoplus_{\lambda\in\Lambda}W_n[v_-]y_\lambda,
 \qquad
 y_\lambda\in(M^d)^{(j_\lambda)}.
\]

Define
\[
 s_B:=c_n\circ\bigl(s\otimes_A^{\mathbf L}B_n\bigr):
 \mathcal U_h\otimes_A^{\mathbf L}B_n
 \longrightarrow
 \mathcal H_{Y_n}\otimes_{W_n[v_-]}^{\mathbf L}B_n.
\]
Using the chosen free models, represent this morphism by
a chain map, also denoted by $s_B$,
\[
 s_B:
 B_n\mathbf 1
 \oplus
 \left(B_nz\oplus\bigoplus_{a=1}^{h-1}B_ne_a\right)[-d]
 \longrightarrow
 M^\bullet\otimes_{W_n[v_-]}B_n.
\]
By \eqref{eqn:Bdegree},
\[
 B_n^1=W_nv_+,
 \qquad
 B_n^{-r}=W_nv_-^r\quad(r\geq0).
\]
Since $s_B$ preserves grading degree, its components have
the form
\[
 \begin{aligned}
 s_B(\mathbf 1)&=a_0\mathbf 1_H,\\
 s_B(z)&=uz_H+\sum_{\lambda\in\Lambda}
       \alpha_\lambda v_-^{j_\lambda}y_\lambda,\\
 s_B(e_a)&=b_av_+z_H+\sum_{\lambda\in\Lambda}
       \gamma_{a\lambda}v_-^{j_\lambda-1}y_\lambda,
 \end{aligned}
\]
with $a_0,u,b_a,\alpha_\lambda,\gamma_{a\lambda}\in W_n$.
The section condition gives
\[
 \operatorname{gr}^0(q\bmod p)\circ
 \operatorname{gr}^0(s\bmod p)=\mathrm{id}.
\]
Together with the equivalence
$\operatorname{gr}^0(q\bmod p)$ (Lemma~\ref{edge:lower-split}), this gives $\overline a_0,\overline u\in k^\times$. Hence we have $a_0,u\in W_n^\times.$

Change the $A_n$-basis by
\[
 e'_a=e_a-u^{-1}b_av_+z.
\]
Then
\[
 \begin{aligned}
 s_B(e'_a)
 &=
 \sum_{\lambda\in\Lambda}
 \left(
 \gamma_{a\lambda}v_-^{j_\lambda-1}
 -
 u^{-1}b_a\alpha_\lambda v_+v_-^{j_\lambda}
 \right)y_\lambda\\
 &=
 \sum_{\lambda\in\Lambda}
 \left(\gamma_{a\lambda}
       -pu^{-1}b_a\alpha_\lambda\right)
 v_-^{j_\lambda-1}y_\lambda
 \in M^d.
 \end{aligned}
\]
Define
\[
 \mathcal U_{H,n}
 =
 W_n[v_-]\mathbf 1
 \oplus
 \left(
 W_n[v_-]z
 \oplus\bigoplus_{a=1}^{h-1}W_n[v_-]e'_a
 \right)[-d],
 \qquad d_{\mathcal U_{H,n}}=0,
\]
\[
 \begin{aligned}
 s_H(\mathbf 1)&=s_B(\mathbf 1),&
 s_H(z)&=s_B(z),&
 s_H(e'_a)&=s_B(e'_a),\\
 \beta_n(\mathbf 1)&=\mathbf 1,&
 \beta_n(z)&=z,&
 \beta_n(e'_a)&=e_a-u^{-1}b_av_+z.
 \end{aligned}
\]
The change of basis is invertible, so $\beta_n$ is an
isomorphism. Moreover,
since $W_n[v_-]\hookrightarrow B_n$ induces $M^r\hookrightarrow M^r\otimes_{W_n[v_-]}B_n$, 
\[
 (d_Ms_H)\otimes_{W_n[v_-]}B_n
 =
 d_{M\otimes B_n}s_B\beta_n=0
\]
implies $d_Ms_H=0.$
Thus $s_H$ is a chain map, and
$
 s_H\otimes_{W_n[v_-]}B_n=s_B\circ\beta_n$ as desired.

\smallskip\noindent
\textit{(2)}
Using the generators of grading degree zero chosen above,
we have
\[
 \operatorname{gr}^0(s_H):
 W_n\mathbf 1\oplus W_nz[-d]
 \xrightarrow{\left(\begin{smallmatrix}a_0&0\\0&u\end{smallmatrix}\right)}
 W_n\mathbf 1_H\oplus W_nz_H[-d].
\]
Since $a_0,u\in W_n^\times$, this is an isomorphism.

\smallskip\noindent
\textit{(3)}
The grading degrees
\[
 \deg(\mathbf 1)=\deg(z)=0,
 \qquad
 \deg(e'_a)=1
\]
give
\[
 \Fil^2\mathcal U_{H,n}=0,
 \qquad
 \Fil^1\mathcal U_{H,n}
 =
 \left(\bigoplus_{a=1}^{h-1}W_ne'_a\right)[-d]
 \simeq W_n^{\oplus(h-1)}[-d].
\]
\end{proof}

\subsection{A matrix calculation}

\begin{lemma}\label{edge:kernel}
Let $V$ and $U$ be finite free $W_n$-modules.
Fix direct sum decompositions
\[
 V=\bigoplus_{i=1}^c V_i,
 \qquad
 U=\bigoplus_{j=1}^m U_j,
 \qquad c\geq2,\quad m\geq1,
\]
where each $V_i$ and $U_j$ is finite free over $W_n$;
zero summands are allowed.
Let $H:V\to U$, and write $H_{ji}:V_i\to U_j$ for its components.
Assume that $H_{ji}=0$ whenever $j<i$.
Define $H_c:V\to U$ by
\[
 (H_c)_{ji}=
 \begin{cases}
 p^{\min(j,c)-i}H_{ji},&j\geq i,\\
 0,&j<i.
 \end{cases}
\]
If
\[
 \length_{W_n}\ker H_c=\length_{W_n}\ker H,
\]
then the projection $V\to V_1$ induces a surjection
\[
 \ker H\longrightarrow
 \ker(H_{11}:V_1\to U_1).
\]
\end{lemma}

\begin{proof}
Put
\[
 V_{\geq2}=\bigoplus_{i=2}^c V_i,
 \qquad
 U_{\geq2}=\bigoplus_{j=2}^m U_j.
\]
Relative to $V=V_1\oplus V_{\geq2}$ and
$U=U_1\oplus U_{\geq2}$, write
\[
 H=\begin{pmatrix}a&0\\ b&e\end{pmatrix},
 \qquad
 H'=\begin{pmatrix}a&0\\ pb&e\end{pmatrix},
\]
where $a=H_{11}$.
An element $x\in\ker a$ lifts to $\ker H$ precisely when
$b(x)$ vanishes in $\coker e$.
Thus, if
\[
 \overline b:\ker a\longrightarrow\coker e,
 \qquad
 I=\im\overline b,
\]
there are exact sequences
\[
 \begin{aligned}
 0&\longrightarrow\ker e\longrightarrow\ker H
   \longrightarrow\ker\overline b\longrightarrow0,\\
 0&\longrightarrow\ker e\longrightarrow\ker H'
   \longrightarrow\ker(p\overline b)\longrightarrow0.
 \end{aligned}
\]
Consequently,
\begin{equation}
\label{eqn:HH'}
 \length_{W_n}\ker H'-\length_{W_n}\ker H
 =\length_{W_n}I-\length_{W_n}pI\geq0.
\end{equation}

We next compare $H'$ with $H_c$ by repeating the same
block operation.
Put $H^{[0]}:=H$.
For $1\leq r\leq c-1$, define $H^{[r]}:V\to U$ by
\[
 (H^{[r]})_{ji}
 =
 \begin{cases}
 p(H^{[r-1]})_{ji},&i\leq r<j,\\
 (H^{[r-1]})_{ji},&\text{otherwise}.
 \end{cases}
\]
For each such $r$, put
\[
 \begin{aligned}
 V_{\leq r}&:=\bigoplus_{i\leq r}V_i,
 &V_{>r}&:=\bigoplus_{i>r}V_i,\\
 U_{\leq r}&:=\bigoplus_{j\leq r}U_j,
 &U_{>r}&:=\bigoplus_{j>r}U_j.
 \end{aligned}
\]
Empty direct sums are understood to be zero.
Since $(H^{[r-1]})_{ji}=0$ whenever $j<i$,
the decompositions
\[
 V=V_{\leq r}\oplus V_{>r},
 \qquad
 U=U_{\leq r}\oplus U_{>r}
\]
give block presentations
\[
 H^{[r-1]}
 =
 \begin{pmatrix}
 a_r&0\\
 b_r&e_r
 \end{pmatrix},
 \qquad
 H^{[r]}
 =
 \begin{pmatrix}
 a_r&0\\
 pb_r&e_r
 \end{pmatrix}.
\]
Applying \eqref{eqn:HH'} to these block presentations gives
\[
 \length_{W_n}\ker H^{[r-1]}
 \leq
 \length_{W_n}\ker H^{[r]}
 \qquad(1\leq r\leq c-1).
\]

By construction, $H^{[1]}=H'$.
In passing from $H^{[0]}=H$ to $H^{[c-1]}$,
each $(j,i)$-block with $j\geq i$ is multiplied by $p$
once for each integer $r$ satisfying $1\leq r<c$ and
$i\leq r<j$. Since
\[
 \#\{r:1\leq r<c,\ i\leq r<j\}
 =\min(j,c)-i,
\]
we obtain $H^{[c-1]}=H_c$.
The successive kernel-length inequalities give
\refstepcounter{equation}
\[
 \length_{W_n}\ker H
 \leq \length_{W_n}\ker H'
 \leq \length_{W_n}\ker H_c.
 \tag{\theequation}\label{eqn:kernel-length-ineq}
\]

By the hypothesis, both inequalities are equalities.
Equation~\eqref{eqn:HH'} therefore gives
$\length_{W_n}I=\length_{W_n}pI$, so $I=pI$.
Thus $I=0$ and hence $\overline b=0$, so every element of
$\ker a=\ker H_{11}$ lifts to $\ker H$, as required.
\end{proof}

The next proposition establishes over finite Witt rings the
analogue of the first Hodge edge surjectivity assumed in
\cite[Theorem~7.16(3)]{BT}. It supplies the edge hypothesis of
Proposition~\ref{tower:extension}.

\begin{proposition}\label{edge:finite}
Suppose that
\[
 H^{d-1-i}(X,\Omega_X^i)=0\qquad(p<i\leq d-1).
\]
For every smooth proper lift $Y/W_{n+1}$, the first Hodge edge
\[
 H^{d-1}\bigl(Y_n,\Fil^1\Omega^\bullet_{Y_n/W_n}\bigr)
 \longrightarrow H^{d-2}(Y_n,\Omega^1_{Y_n/W_n})
\]
is surjective. In particular, this holds if $d\leq p+1$, or if
$d=p+2$ and $H^0(X,T_X)=0$.
\end{proposition}

\begin{proof}
Using the morphisms in Lemmas~\ref{edge:lower-split} and
\ref{edge:descent}, form the paired cofibres
\[
 \mathcal E=\cofib(\mathcal U_h\longrightarrow\mathcal N_X),
 \qquad
 \mathcal E_{H,n}
 =\cofib(\mathcal U_{H,n}\longrightarrow\mathcal H_{Y_n}).
\]
By Lemmas~\ref{edge:comparison} and~\ref{edge:descent}, we have
\begin{equation}
\label{eqn:Ecomparison}
     \mathcal E\otimes_A^{\mathbf L}B_n
 \simeq
 \mathcal E_{H,n}
 \otimes_{W_n[v_-]}^{\mathbf L}B_n.
\end{equation}

By \cite[Remark~3.3.6 and Lemma~3.4.11(1)]{BhattFG},
$\mathcal N_X$ is represented by a perfect graded $A$-complex.
Since $Y_n/W_n$ is smooth and proper, the complexes
$\RG(Y_n,\Omega^i_{Y_n/W_n})$ are perfect over $W_n$.
The Rees complex $\mathcal H_{Y_n}$ is perfect over $W_n[v_-]$,
since it is a finite iterated extension of grading shifts of
the perfect complexes
\[
 W_n[v_-]\otimes_{W_n}^{\mathbf L}
 \RG(Y_n,\Omega^i_{Y_n/W_n})[-i].
\]
Together with Lemmas~\ref{edge:lower-split} and
\ref{edge:descent}, this shows that $\mathcal E$ and
$\mathcal E_{H,n}$ are perfect over $A$ and $W_n[v_-]$,
respectively.

The rings $A$ and $W_n[v_-]$ are graded local, with homogeneous
maximal ideals
\[
 \mathfrak m_A=(v_-,v_+),
 \qquad
 \mathfrak m_H=(p,v_-).
\]
By \cite[Lemma~4.10(ii)]{DellAmbrogioStevenson2013},
choose minimal graded free resolutions
\[
 M^\bullet\simeq\mathcal E,
 \qquad
 M_H^\bullet\simeq\mathcal E_{H,n},
\]
over $A$ and $W_n[v_-]$, respectively. 
Thus each term is finite
graded free and
\[
 d_M(M^q)\subseteq\mathfrak m_A M^{q+1},
 \qquad
 d_{M_H}(M_H^q)\subseteq\mathfrak m_H M_H^{q+1}.
\]
These resolutions are bounded by finite Tor amplitude
and graded Nakayama. 
Moreover, we have
\[
 d_{M^\bullet\otimes_A k}=0,
 \qquad
 d_{M_H^\bullet\otimes_{W_n[v_-]}k}=0.
\]

Base change of 
\eqref{eqn:Ecomparison}
along $B_n\to k$ gives
\[
 M^\bullet\otimes_A k
 \simeq
 M_H^\bullet\otimes_{W_n[v_-]}k.
\]

Recall that (see \eqref{eqn:Nygaard's Hodge})
\[
 \left(\mathcal N_X\otimes_A^{\mathbf L}k\right)^{(i)}
 \simeq
 \RG(X,\Omega_X^i)[-i].
\]
Since $q\circ s=\operatorname{id}_{\mathcal U_h}$, we have
\[
 \mathcal N_X\otimes_A^{\mathbf L}k
 \simeq
 \left(\mathcal U_h\otimes_A^{\mathbf L}k\right)
 \oplus
 \left(M^\bullet\otimes_A k\right).
\]
Moreover, Lemma~\ref{edge:lower-split} gives
\[
 \mathcal U_h\otimes_A^{\mathbf L}k
 \simeq
 k\mathbf 1
 \oplus
 \left(kz\oplus\bigoplus_{a=1}^{h-1}ke_a\right)[-d],
 \qquad
 \deg(\mathbf 1)=\deg(z)=0,\quad \deg(e_a)=1.
\]
Consequently, the number $r_{q,i}$ of homogeneous generators
of grading degree $i$ in either $M^q$ or $M_H^q$ is
\begin{equation}
\label{eqn:rqiformula}
 \begin{aligned}
 r_{q,i}
 &:=
 \dim_k(M^q\otimes_A k)^{(i)}
 =
 \dim_k(M_H^q\otimes_{W_n[v_-]}k)^{(i)}\\
 &=
 h^{q-i}(X,\Omega_X^i)
 -\delta_{i,0}(\delta_{q,0}+\delta_{q,d})
 -(h-1)\delta_{i,1}\delta_{q,d},
 \end{aligned}
\end{equation}
where $\delta_{a,b}$ is the Kronecker delta and
$\Omega_X^i:=0$ for $i\notin[0,d]$.
By \eqref{eqn:rqiformula} and the strict Calabi--Yau
assumptions, we have
\[
 r_{q,i}=0\qquad(i\leq0).
\]
Since $d\geq3$, the subtraction terms vanish for $q=d-1$,
so
\[
 r_{d-1,i}=h^{d-1-i}(X,\Omega_X^i).
\]
The vanishing hypothesis therefore gives
\[
 r_{d-1,i}=0\qquad(i\notin[1,p]).
\]
This upper bound on the grading degrees in cohomological
degree $d-1$ is where the vanishing hypothesis enters the argument.
It is used in the matrix identities
\eqref{eqn:N+Np} and \eqref{eq:edge-Hp-rescaling} below,
and in applying Lemma~\ref{edge:kernel} with $c=p$.

Every homogeneous generator of $M^\bullet$ and
$M_H^\bullet$ has positive grading degree, and every
homogeneous generator of $M^{d-1}$ and $M_H^{d-1}$
has grading degree in $[1,p]$.

Put
\[
 M_-^\bullet:=M^\bullet|_{v_-=1},
 \qquad
 M_+^\bullet:=M^\bullet|_{v_+=1}.
\]
Choose homogeneous bases of $M^{d-1},M^d$ and
$M_H^{d-1},M_H^d$, with the same grading degrees
$i_\alpha,j_\beta$ in cohomological degrees $d-1,d$,
respectively. Recall that $1\leq i_\alpha\leq p$.
At the endpoints, use the specializations of these bases.
For a homogeneous basis vector $e$ of grading degree $r$,
the corresponding free summand after extension to $A/p^nA$
satisfies
\refstepcounter{equation}
\[
 \bigl((A/p^nA)e\bigr)^{(p)}
 =
 (A/p^nA)^{(p-r)}e
 =
 \begin{cases}
 W_n v_+^{p-r}e,&r\leq p,\\
 W_n v_-^{r-p}e,&r>p.
 \end{cases}
 \tag{\theequation}\label{eq:degree-p-basis}
\]
Use the displayed generators as the bases in homogeneous
degree $p$.

With respect to these bases, let $N_\pm$, $N_p$, $H_-$,
and $H_p$ denote the matrices of $d^{d-1}$ on
$M_\pm^\bullet\otimes_W W_n$,
$(M^\bullet\otimes_W W_n)^{(p)}$,
$M_H^\bullet|_{v_-=1}$, and
$(M_H^\bullet\otimes_{W_n[v_-]}(A/p^nA))^{(p)}$,
respectively.

Let $e_\alpha$ and $f_\beta$ be the source and target
basis vectors of grading degrees $i_\alpha$ and $j_\beta$,
respectively. Denote the degree-$p$ basis vector associated
with $f_\beta$ by $f_{\beta,p}$.
Since $i_\alpha\leq p$, the source basis vector in degree $p$
is $v_+^{p-i_\alpha}e_\alpha$.
For $j_\beta\geq i_\alpha$, the relation $v_-v_+=p$ gives
\[
 \begin{aligned}
 v_+^{p-i_\alpha}v_-^{j_\beta-i_\alpha}f_\beta
 &=
 \begin{cases}
 p^{j_\beta-i_\alpha}v_+^{p-j_\beta}f_\beta,
     &j_\beta\leq p,\\
 p^{p-i_\alpha}v_-^{j_\beta-p}f_\beta,
     &j_\beta>p
 \end{cases}\\
 &=p^{\min(j_\beta,p)-i_\alpha}f_{\beta,p}.
 \end{aligned}
\]
For $j_\beta\geq i_\alpha$, if the coefficient of $f_\beta$
in $d_M^{d-1}(e_\alpha)$ is $a v_-^{j_\beta-i_\alpha}$,
then, writing $a$ also for its image in $W_n$, we have
\[
 (N_-)_{\beta\alpha}=a,\qquad
 (N_+)_{\beta\alpha}=a p^{j_\beta-i_\alpha},\qquad
 (N_p)_{\beta\alpha}=a p^{\min(j_\beta,p)-i_\alpha}.
\]
For $j_\beta<i_\alpha$, an entry $a v_+^{i_\alpha-j_\beta}$
gives
\[
 (N_+)_{\beta\alpha}=(N_p)_{\beta\alpha}=a.
\]
Since $j_\beta<i_\alpha\leq p$, we also have
$\max(j_\beta-p,0)=0$.
Consequently,
\begin{equation}
\label{eqn:N+Np}
 (N_+)_{\beta\alpha}
 =
 p^{\max(j_\beta-p,0)}(N_p)_{\beta\alpha}.
\end{equation}

On the Hodge side, an entry $b v_-^{j_\beta-i_\alpha}$
with $j_\beta\geq i_\alpha$ specializes to $b$ at $v_-=1$.
If $j_\beta<i_\alpha$, the corresponding coefficient of
$d_{M_H}^{d-1}$ would have positive grading degree
$i_\alpha-j_\beta$, so it vanishes because $W_n[v_-]$
has no positive graded pieces.
For $j_\beta\geq i_\alpha$, the same degree-$p$ basis calculation
as on the Nygaard side applies.
Together with the vanishing for $j_\beta<i_\alpha$, this gives
\refstepcounter{equation}
\[
 \begin{aligned}
 (H_p)_{\beta\alpha}
 &=p^{\min(j_\beta,p)-i_\alpha}(H_-)_{\beta\alpha}
 &&(j_\beta\geq i_\alpha),\\
 (H_-)_{\beta\alpha}
 &=(H_p)_{\beta\alpha}=0
 &&(j_\beta<i_\alpha).
 \end{aligned}
 \tag{\theequation}\label{eq:edge-Hp-rescaling}
\]
Write $e_\alpha$ and $f_\beta$ for the chosen homogeneous
basis vectors of $M_H^{d-1}$ and $M_H^d$, respectively,
with grading degrees $i_\alpha$ and $j_\beta$.
Put
\[
 \begin{aligned}
 V_i&:=
 \bigoplus_{\alpha:\,i_\alpha=i}
 W_n(e_\alpha|_{v_-=1})
 &&(1\leq i\leq p),\\
 U_j&:=
 \bigoplus_{\beta:\,j_\beta=j}
 W_n(f_\beta|_{v_-=1})
 &&(1\leq j\leq d).
 \end{aligned}
\]
By \eqref{eqn:rqiformula} and $1\leq i_\alpha\leq p$,
these give decompositions
\[
 M_H^{d-1}|_{v_-=1}=\bigoplus_{i=1}^{p}V_i,
 \qquad
 M_H^d|_{v_-=1}=\bigoplus_{j=1}^{d}U_j.
\]

By \eqref{eq:degree-p-basis}, the degree-$p$ components
involve only coefficient degrees $p-r<p$, since the
chosen homogeneous basis vectors of $M^q$ and $M_H^q$
have grading degrees $r\geq1$.
Thus \eqref{eqn:Bdegree} gives
\[
 (A/p^nA)^{(p-r)}
 \xrightarrow{\sim}
 (B/p^nB)^{(p-r)}.
\]
Consequently, extension from $A/p^nA$ to $B/p^nB$
induces termwise isomorphisms on the degree-$p$ complexes.
Together with the $F$-gauge structure,
\eqref{eqn:Ecomparison} therefore gives equivalences in $\mathcal D(W_n)$
\refstepcounter{equation}
\[
 \begin{aligned}
 (M^\bullet\otimes_W W_n)^{(p)}
 &\simeq
 \bigl(M_H^\bullet
       \otimes_{W_n[v_-]}(A/p^nA)\bigr)^{(p)},\\
 M_-^\bullet\otimes_W W_n
 &\simeq M_H^\bullet|_{v_-=1},\\
 (\sigma^*M_+^\bullet)\otimes_W W_n
 &\simeq M_-^\bullet\otimes_W W_n.
 \end{aligned}
 \tag{\theequation}\label{eq:edge-three-comparisons}
\]

For each equivalence in \eqref{eq:edge-three-comparisons},
let $K^\bullet$ and $L^\bullet$ denote the complexes
on the left and right, respectively.
By \eqref{eqn:rqiformula} and \eqref{eq:degree-p-basis},
the complexes $K^\bullet$ and $L^\bullet$ are bounded
complexes of finite free $W_n$-modules satisfying
\[
 \operatorname{rank}_{W_n}K^q
 =
 \operatorname{rank}_{W_n}L^q
 =
 \sum_i r_{q,i}
 \qquad(q\in\mathbb Z).
\]
The equivalence $K^\bullet\simeq L^\bullet$ gives
\[
 \length_{W_n}H^q(K^\bullet)
 =
 \length_{W_n}H^q(L^\bullet)
 \qquad(q\in\mathbb Z).
\]

The exact sequences
\[
 \begin{aligned}
 0&\longrightarrow\ker d_K^{q-1}
   \longrightarrow K^{q-1}
   \longrightarrow\operatorname{im}d_K^{q-1}
   \longrightarrow0,\\
 0&\longrightarrow\operatorname{im}d_K^{q-1}
   \longrightarrow\ker d_K^q
   \longrightarrow H^q(K^\bullet)
   \longrightarrow0
 \end{aligned}
\]
give
\[
 \length_{W_n}\ker d_K^q
 =
 \length_{W_n}H^q(K^\bullet)
 +\length_{W_n}K^{q-1}
 -\length_{W_n}\ker d_K^{q-1}.
\]
Subtracting the formulas for $K^\bullet$ and $L^\bullet$
and inducting from a common lower bound gives
\[
 \length_{W_n}\ker d_K^q
 =
 \length_{W_n}\ker d_L^q
 \qquad(q\in\mathbb Z).
\]
Taking $q=d-1$ for each pair of complexes in
\eqref{eq:edge-three-comparisons}, we obtain
\[
 \begin{aligned}
 \length_{W_n}\ker H_p
 &=\length_{W_n}\ker N_p,\\
 \length_{W_n}\ker N_-
 &=\length_{W_n}\ker H_-,\\
 \length_{W_n}\ker N_+
 &=\length_{W_n}\ker N_-.
 \end{aligned}
\]

On the other hand, \eqref{eqn:N+Np}
gives $\ker N_p\subseteq\ker N_+$.
In view of \eqref{eq:edge-Hp-rescaling}, we may apply
\eqref{eqn:kernel-length-ineq} with
$H=H_-$, $H_c=H_p$, $c=p$, and $m=d$, obtaining
\[
 \length\ker H_-\leq\length\ker H_p.
\]
Combining these equalities and inequalities, we obtain
\[
 \begin{aligned}
 \length\ker H_p
 &=\length\ker N_p
 \leq\length\ker N_+
 =\length\ker N_-\\
 &=\length\ker H_-
 \leq\length\ker H_p.
 \end{aligned}
\]
Hence all inequalities are equalities.

By \eqref{eq:edge-Hp-rescaling} and
$\length\ker H_-=\length\ker H_p$,
Lemma~\ref{edge:kernel}, applied with
$H=H_-$, $H_c=H_p$, $c=p$, and $m=d$, gives a surjection
\refstepcounter{equation}
\[
 \ker H_-
 \twoheadrightarrow
 \ker\bigl((H_-)_{11}:V_1\to U_1\bigr).
 \tag{\theequation}\label{eq:edge-kernel-surjection}
\]

For each chosen homogeneous basis vector $e$ of
$M_H^\bullet$ of grading degree $r\geq1$, we have
$(W_n[v_-]e)^{(1)}=W_n v_-^{r-1}e$.
Evaluation at $v_-=1$ sends this basis vector to $e$.
Thus
\[
 \begin{aligned}
 \Fil^1\mathcal E_{H,n}
 &\simeq (M_H^\bullet)^{(1)}
 \xrightarrow[\ v_-=1\ ]{\sim}
 M_H^\bullet|_{v_-=1},\\
 \operatorname{gr}^1\mathcal E_{H,n}
 &\simeq
 (M_H^\bullet/v_-M_H^\bullet)^{(1)}.
 \end{aligned}
\]
Under these identifications, the map from $\Fil^1$ to
$\operatorname{gr}^1$ is the projection onto the basis
vectors of grading degree $1$: reduction modulo $v_-$
sends $v_-^{r-1}e$ to $e$ for $r=1$ and to zero for $r>1$.
The differential of the target in degree $d-1$ is
$(H_-)_{11}$.
Equation~\eqref{eq:edge-kernel-surjection} therefore gives
surjectivity on cycles in degree $d-1$ for this map.
Passing to cohomology gives
\[
e_E\colon
 H^{d-1}(\Fil^1\mathcal E_{H,n})
 \twoheadrightarrow
 H^{d-1}(\operatorname{gr}^1\mathcal E_{H,n}).
 \]

We now prove the required edge surjectivity for $Y_n$.
By Lemma~\ref{edge:descent}, the map
$\Fil^1\mathcal U_{H,n}\to
\operatorname{gr}^1\mathcal U_{H,n}$
identifies with the identity on
$W_n^{\oplus(h-1)}[-d]$.
Since
$H^{d-1}(W_n^{\oplus(h-1)}[-d])=0$,
the cohomology sequences give a commutative diagram
with exact rows
\[
 \begin{tikzcd}[column sep=small]
 0 \arrow[r] &
 H^{d-1}(\Fil^1\mathcal H_{Y_n})
   \arrow[r,"a"] \arrow[d,"e_H"] &
 H^{d-1}(\Fil^1\mathcal E_{H,n})
   \arrow[r,"\delta_F"] \arrow[d,two heads,"e_E"] &
 W_n^{\oplus(h-1)} \arrow[d,equal]\\
 0 \arrow[r] &
 H^{d-1}(\operatorname{gr}^1\mathcal H_{Y_n})
   \arrow[r,"b"] &
 H^{d-1}(\operatorname{gr}^1\mathcal E_{H,n})
   \arrow[r,"\delta_G"] &
 W_n^{\oplus(h-1)}.
 \end{tikzcd}
\]
By diagram chasing, we obtain the desired surjectivity.
\end{proof}

\begin{proof}[Proof of Theorem~\ref{main:witt}]
If $d=1$, then $X$ is a genus-one curve and admits a smooth
projective lifting over $W$.
If $d=2$, then $X$ is a K3 surface, and the conclusion follows
from \cite[Proposition~2.6]{Ito18}.
We may therefore assume that $d\geq3$.
Choose the initial lifting and the decomposition of
Lemma~\ref{common:canonical}.
By assumption, either $d\leq p+1$, or $d=p+2$ and
$H^0(X,T_X)=0$. Proposition~\ref{edge:finite} therefore
gives the edge-map hypothesis of Proposition~\ref{tower:extension},
which yields a smooth projective lifting over $W$.
\end{proof}

{
\noindent\textbf{Use of generative AI}

The proofs of Theorems~\ref{main:witt} and~\ref{main:W3},
as well as nearly all of the other original mathematical
arguments in this paper, were developed by the author
through extensive discussions with OpenAI's ChatGPT
(GPT-6 Astra).
The question originated with the author, who expected
Witt-vector liftability in view of the favorable properties
of the $W_2$-liftings arising from quasi-$F$-splittings
and asked ChatGPT to investigate this possibility.
ChatGPT was also used to refine the wording and improve
the clarity of the exposition.

At an early stage, ChatGPT claimed that the arguments also
covered globally $F$-split varieties in characteristic two
and quasi-$F$-split varieties in the boundary case $p=d-2$
without the additional assumption $H^0(X,T_X)=0$.
During subsequent verification, carried out with the
assistance of ChatGPT, the author was unable to reconstruct
complete proofs of these stronger claims.
The author therefore decided to exclude characteristic two
and to retain the boundary case only under the assumption
$H^0(X,T_X)=0$.
The author has checked all material retained in the final
manuscript and takes full responsibility for its mathematical
content, presentation, and scope, including these decisions.
}

\section*{Acknowledgements}
The author was supported by JSPS KAKENHI Grant number JP25K17228.

\bibliographystyle{amsalpha}
\bibliography{final}

@article{AZ,
  author  = {Achinger, Piotr and Zdanowicz, Maciej},
  title   = {{Serre--Tate} theory for {Calabi--Yau} varieties},
  journal = {J. Reine Angew. Math.},
  volume  = {780},
  year    = {2021},
  pages   = {139--196},
  doi     = {10.1515/crelle-2021-0041}
}

@article{BerthelotQ,
  author       = {Berthelot, Pierre},
  title        = {Rigid cohomology and de {Rham--Witt} complexes},
  journal      = {Rend. Sem. Mat. Univ. Padova},
  volume       = {128},
  year         = {2012},
  pages        = {287--344}
}

@misc{BhattFG,
  author       = {Bhatt, Bhargav},
  title        = {Prismatic {$F$}-gauges},
  howpublished = {Lecture notes, Princeton University},
  year         = {2022}
}

@article{BMS2,
  author       = {Bhatt, Bhargav and Morrow, Matthew and Scholze, Peter},
  title        = {Topological {Hochschild} homology and integral {$p$}-adic
                  {Hodge} theory},
  journal      = {Publ. Math. Inst. Hautes \'Etudes Sci.},
  volume       = {129},
  year         = {2019},
  pages        = {199--310}
}

@misc{BT,
  author       = {Brantner, Lukas and Taelman, Lenny},
  title        = {Deformations and lifts of {Calabi--Yau} varieties
                  in characteristic {$p$}},
  year         = {2025},
  note         = {arXiv:2407.09256v3}
}

@article{DellAmbrogioStevenson2013,
  author       = {Dell'Ambrogio, Ivo and Stevenson, Greg},
  title        = {On the derived category of a graded commutative
                  noetherian ring},
  journal      = {J. Algebra},
  volume       = {373},
  year         = {2013},
  pages        = {356--376}
}

@article{Illusie,
  author       = {Illusie, Luc},
  title        = {Complexe de de {Rham--Witt} et cohomologie cristalline},
  journal      = {Ann. Sci. \'Ecole Norm. Sup. (4)},
  volume       = {12},
  number       = {4},
  year         = {1979},
  pages        = {501--661}
}

@article{IR,
  author       = {Illusie, Luc and Raynaud, Michel},
  title        = {Les suites spectrales associ\'ees au complexe de de
                  {Rham--Witt}},
  journal      = {Publ. Math. Inst. Hautes \'Etudes Sci.},
  number       = {57},
  year         = {1983},
  pages        = {73--212}
}

@book{Ogus94,
  author       = {Ogus, Arthur},
  title        = {{$F$}-crystals, {Griffiths} transversality, and the
                  {Hodge} decomposition},
  series       = {Ast\'erisque},
  number       = {221},
  publisher    = {Soci\'et\'e Math\'ematique de France},
  address      = {Paris},
  year         = {1994}
}

@misc{Petrov,
  author       = {Petrov, Alexander},
  title        = {Decomposition of de {Rham} complex for
                  quasi-{$F$}-split varieties},
  year         = {2025},
  note         = {arXiv:2502.13356}
}

@misc{Stacks,
  author       = {{The Stacks Project Authors}},
  title        = {The {Stacks} project},
  howpublished = {\url{https://stacks.math.columbia.edu}},
  year         = {2026}
}

@article{TVX,
  author       = {Terentiuk, Gleb and Vologodsky, Vadim and Xu, Yujie},
  title        = {Prismatic {$F$}-gauges and {Fontaine--Laffaille} modules},
  journal      = {Compos. Math.},
  volume       = {162},
  number       = {2},
  year         = {2026},
  pages        = {277--326}
}

@article{YobukoQ,
  author       = {Yobuko, Fuetaro},
  title        = {Quasi-{Frobenius} splitting and lifting of
                  {Calabi--Yau} varieties in characteristic {$p$}},
  journal      = {Math. Z.},
  volume       = {292},
  number       = {1--2},
  year         = {2019},
  pages        = {307--316}
}

@article{ES05,
  author       = {Ekedahl, T. and Shepherd-Barron, N. I.},
  title        = {Tangent lifting of deformations in mixed characteristic},
  journal      = {J. Algebra},
  volume       = {291},
  number       = {1},
  year         = {2005},
  pages        = {108--128}
}

@article{Schroer03,
  author       = {Schr{\"o}er, Stefan},
  title        = {The {$T^1$}-lifting theorem in positive characteristic},
  journal      = {J. Algebraic Geom.},
  volume       = {12},
  number       = {4},
  year         = {2003},
  pages        = {699--714}
}

@book{BLM,
  author       = {Bhatt, B. and Lurie, J. and Mathew, A.},
  title        = {Revisiting the de {Rham--Witt} complex},
  series       = {Ast\'erisque},
  number       = {424},
  publisher    = {Soci\'et\'e Math\'ematique de France},
  year         = {2021}
}

@book{Weibel94,
  author       = {Weibel, C. A.},
  title        = {An Introduction to Homological Algebra},
  series       = {Cambridge Studies in Advanced Mathematics},
  volume       = {38},
  publisher    = {Cambridge University Press},
  address      = {Cambridge},
  year         = {1994}
}

@article{Ogus79,
  author       = {Ogus, Arthur},
  title        = {Supersingular {K3} crystals},
  journal      = {Ast\'erisque},
  volume       = {64},
  year         = {1979},
  pages        = {3--86}
}

@article{Ito18,
  author       = {Ito, Kazuhiro},
  title        = {Finiteness of {Brauer} groups of {K3} surfaces
                  in characteristic {$2$}},
  journal      = {Int. J. Number Theory},
  volume       = {14},
  number       = {6},
  year         = {2018},
  pages        = {1813--1825}
}

@incollection{Deligne81,
  author       = {Deligne, Pierre},
  title        = {Rel\`evement des surfaces {K3} en caract\'eristique nulle},
  booktitle    = {Algebraic surfaces (Orsay, 1976--78)},
  series       = {Lecture Notes in Mathematics},
  volume       = {868},
  publisher    = {Springer},
  address      = {Berlin--New York},
  year         = {1981},
  pages        = {58--79}
}

@article{AddingtonBragg,
  author  = {Addington, Nicolas and Bragg, Daniel},
  title   = {{Hodge} numbers are not derived invariants in positive characteristic},
  journal = {Math. Ann.},
  volume  = {387},
  number  = {1-2},
  year    = {2023},
  pages   = {847--878},
  doi     = {10.1007/s00208-022-02474-w},
  note    = {With an appendix by Alexander Petrov}
}

@article{DI87,
  author  = {Deligne, Pierre and Illusie, Luc},
  title   = {Rel{\`e}vements modulo $p^2$ et d{\'e}composition du complexe de de {Rham}},
  journal = {Invent. Math.},
  volume  = {89},
  number  = {2},
  year    = {1987},
  pages   = {247--270},
  doi     = {10.1007/BF01389078}
}

@article{Hirokado99,
  author  = {Hirokado, Masayuki},
  title   = {A non-liftable {Calabi--Yau} threefold in characteristic $3$},
  journal = {Tohoku Math. J. (2)},
  volume  = {51},
  number  = {4},
  year    = {1999},
  pages   = {479--487},
  doi     = {10.2748/tmj/1178224716}
}

@article{Schroer04,
  author  = {Schr{\"o}er, Stefan},
  title   = {Some {Calabi--Yau} threefolds with obstructed deformations over the {Witt} vectors},
  journal = {Compos. Math.},
  volume  = {140},
  number  = {6},
  year    = {2004},
  pages   = {1579--1592},
  doi     = {10.1112/S0010437X04000545}
}

@misc{Moore26,
  author        = {Moore, Samuel A.},
  title         = {Characterising canonical lifts of ordinary
                   {Calabi--Yau} varieties},
  year          = {2026},
  eprint        = {2610.35807},
  archivePrefix = {arXiv},
  primaryClass  = {math.AG},
  note          = {arXiv:2610.35807},
  url           = {https://arxiv.org/abs/2610.35807}
}

@misc{Ward26,
  author        = {Ward, Matthew},
  title         = {All strict {Calabi--Yau} threefolds of height one
                   lift to characteristic 0},
  year          = {2026},
  eprint        = {2610.37075},
  archivePrefix = {arXiv},
  primaryClass  = {math.AG},
  note          = {arXiv:2610.37075},
  url           = {https://arxiv.org/abs/2610.37075}
}

@article{ArtinMazur77,
  author  = {Artin, M. and Mazur, B.},
  title   = {Formal groups arising from algebraic varieties},
  journal = {Annales scientifiques de l'\'Ecole Normale Sup\'erieure},
  series  = {4},
  volume  = {10},
  number  = {1},
  year    = {1977},
  pages   = {87--131},
  doi     = {10.24033/asens.1322},
  url     = {https://numdam.org/articles/10.24033/asens.1322/}
}

@incollection{Katz81,
  author    = {Katz, Nicholas M.},
  title     = {{Serre--Tate} local moduli},
  booktitle = {Algebraic surfaces ({Orsay}, 1976--78)},
  series    = {Lecture Notes in Mathematics},
  volume    = {868},
  pages     = {138--202},
  publisher = {Springer},
  address   = {Berlin--New York},
  year      = {1981},
  doi       = {10.1007/BFb0090648}
}

@article{NormanOort80,
  author  = {Norman, Peter and Oort, Frans},
  title   = {Moduli of abelian varieties},
  journal = {Ann. of Math. (2)},
  volume  = {112},
  number  = {2},
  year    = {1980},
  pages   = {413--439},
  doi     = {10.2307/1971152}
}

@misc{LST64,
  author       = {Lubin, Jonathan and Serre, Jean-Pierre
                  and Tate, John},
  title        = {Elliptic curves and formal groups},
  year         = {1964},
  howpublished = {Seminar notes, Woods Hole Summer Institute
                  in Algebraic Geometry},
  url          = {https://web.ma.utexas.edu/users/voloch/lst.html}
}

@article{BBKW24,
  author  = {Bernasconi, Fabio and Brivio, Iacopo and
             Kawakami, Tatsuro and Witaszek, Jakub},
  title   = {Lifting globally {$F$}-split surfaces to
             characteristic zero},
  journal = {J. Reine Angew. Math.},
  volume  = {816},
  year    = {2024},
  pages   = {47--89},
  doi     = {10.1515/crelle-2024-0058}
}
\end{document}